\documentclass{article}

\usepackage[english]{babel}
\usepackage[utf8]{inputenc}

\usepackage[letterpaper,top=2cm,bottom=2cm,left=3cm,right=3cm,marginparwidth=1.75cm]{geometry}

\usepackage{stmaryrd}
\usepackage{mathtools}
\usepackage{chngcntr}
\usepackage{indentfirst}
\usepackage{amsmath}
\usepackage{graphicx}
\usepackage{amssymb}
\usepackage{amsthm}
\usepackage[colorlinks=true, allcolors=blue]{hyperref}
\usepackage{mathrsfs}
\usepackage{tikz-cd}
\usetikzlibrary{arrows.meta}
\usepackage{tikz}
\usetikzlibrary{arrows.meta,patterns}
\usepackage{stmaryrd}

\usepackage[doi = false, isbn = false, url = false, style = alphabetic, maxnames = 100]{biblatex}
    \DeclareFieldFormat{postnote}{#1}
    \AtBeginDocument{%
       \def\MR#1{}
    }
\usepackage{cleveref}
\newtheorem{thm}{Theorem}[subsection]
\newtheorem{lemma}[thm]{Lemma}

\newtheorem{rmk}[thm]{Remark}
\newtheorem{prop}[thm]{Proposition}
\newtheorem{defn}[thm]{Definition}
\newtheorem{conj}[thm]{Conjecture}
\newtheorem{eg}[thm]{Example}
\newtheorem{fact}{Fact}

\crefname{lemma}{lemma}{lemmas}
\Crefname{lemma}{Lemma}{Lemmas}
\crefname{cor}{corollary}{corollaries}
\Crefname{cor}{Corollary}{Corollaries}
\crefname{rmk}{remark}{remarks}
\Crefname{rmk}{Remark}{Remarks}
\crefname{prop}{proposition}{propositions}
\Crefname{prop}{Proposition}{Propositions}
\crefname{defn}{definition}{definitions}
\Crefname{defn}{Definition}{Definitions}
\crefname{conj}{conjecture}{conjectures}
\Crefname{conj}{Conjecture}{Conjectures}
\crefname{eg}{example}{examples}
\Crefname{eg}{Example}{Examples}

\title{Bump--Friedberg type periods beyond the cuspidal spectrum}
\author{Shenghao Li}
\numberwithin{equation}{subsubsection}
\renewcommand{\theequation}{%
  \ifnum\value{subsection}>0 %
    \ifnum\value{subsubsection}>0 %
      \thesubsubsection.\arabic{equation}% Has subsubsection
    \else
      \thesubsection.\arabic{equation}% Has subsection only
    \fi
  \else
    \thesection.\arabic{equation}% Section only
  \fi
}
\numberwithin{thm}{subsubsection}
\renewcommand{\thethm}{%
  \ifnum\value{subsection}>0 %
    \ifnum\value{subsubsection}>0 %
      \thesubsubsection.\arabic{thm}% Has subsubsection
    \else
      \thesubsection.\arabic{thm}% Has subsection only
    \fi
  \else
    \thesection.\arabic{thm}% Section only
  \fi
  }

\newtheorem{claim}{Claim}

\newtheoremstyle{normalexample}
  {3pt} % Space above
  {3pt} % Space below  
  {\normalfont} % Body font (normal!)
  {} % Indent amount
  {\bfseries} % Header font
  {.} % Punctuation after header
  {.5em} % Space after header
  {} % Theorem head spec

\theoremstyle{normalexample}

\begin{document}
\maketitle

\begin{abstract}
    In this article, we study several Bump--Friedberg type periods beyond the cuspidal spectrum. We first consider the twisted Bump--Friedberg period on $\textnormal{GL}_{2n}$, as well as a variant on $\textnormal{GL}_1\times \textnormal{GL}_{2n}$. Under suitable regularity conditions on the cuspidal datum, these periods extend continuously to automorphic functions of uniform moderate growth. Such extensions are characterized by entire Whittaker-type zeta integrals. We then introduce a Bump--Friedberg type period on $\textnormal{GL}_{2n+1}$, integrating over the subgroup $\textnormal{SL}_{n+1}\times \textnormal{GL}_n$. For certain Eisenstein series, we evaluate this period as a finite sum of products of special values of $L$-functions and normalized local zeta integrals. Assuming the expected global Langlands correspondence, the sum is indexed by the fixed points of the extended $L$-parameter on the conjectural dual variety, and the resulting $L$-factors agree with the tangent space prediction of the global numerical conjecture of Ben-Zvi-Sakellaridis-Venkatesh.
\end{abstract}

\tableofcontents

\section{Introduction}

\subsection{Relative Langlands duality and the global numerical conjecture}
A fundamental problem in the theory of automorphic forms is to understand the relation between automorphic periods and special values of $L$-functions. Many classical integral representations can be viewed as examples of such a relation: a period on the automorphic side produces an $L$-function attached to the corresponding Langlands parameter.

The relative Langlands duality proposed by Ben-Zvi, Sakellaridis and Venkatesh in \cite{ben2024relative} is intended to seek relations between automorphic periods and $L$-functions. In their framework, suitable Hamiltonian varieties for a reductive group $G$ over a global function field are expected to admit dual Hamiltonian varieties for the Langlands dual group $\check{G}$. On the automorphic side, the original variety gives rise to a period integral, whereas on the spectral side, the dual variety provides the corresponding $L$-functions. Later in \cite{mao2026relative}, Mao, Wan and Zhang formulated the associated period conjecture over general global fields, including number fields, for hyperspherical Hamiltonian varieties.

A feature of the BZSV numerical conjecture that is especially relevant to the present article is the appearance of fixed points of a Langlands parameter on the dual variety. Following \cite{mao2026relative}, we recall its expected form in the polarized case:

\begin{conj}[\textbf{The global numerical conjecture}]
    Let $M=T^*X$ and $\check{M}=T^*\check{X}$ be a pair of dual polarized hyperspherical Hamiltonian varieties, where $X$ is a $G$-spherical variety and $\check{X}$ is the corresponding $\check{G}$-spherical variety. Let $\pi$ be a tempered automorphic representation of $G(\mathbb{A})$ with $L$-parameter $\phi$. Suppose that the fixed point set
    $$\check{X}^{\phi}=\{x_1,\dots,x_r\}$$
    is finite. Then for a suitably normalized spherical vector $f\in\pi$, one has
    $$\mathcal{P}_X(f)\sim \sum_{i=1}^rL(0,(T_{x_i}\check{X})^{\talloblong}).$$
    Here $\sim$ suppresses the global constants and local factors at ramified places occurring in a precise period formula.
\end{conj}

The exact BZSV global numerical conjecture is formulated under the assumption that $G$ is over a global function field and all data are everywhere unramified. For number fields, the right-hand side should also include ramified and archimedean local factors.

In the global numerical conjecture, the spectral side need not be a single $L$-value. When the hyperspherical Hamiltonian varieties are not strongly tempered, the corresponding action of $\phi$ on $\check{X}$ might have multiple fixed points. In this case, the conjecture predicts a finite sum, with each summand coming from a fixed point. This phenomenon was recently established by Lu and Xi for certain periods of Eisenstein series in \cite{lu2025periodsdetectingeisensteinseries}, whose dual variety is of the form $T^*(\check{G}/\check{L})$, where $\check{G}$ is a general linear group and $\check{L}$ is a Levi subgroup. 

The goal of the present article is to study an analogous phenomenon for Bump--Friedberg type periods.

\subsection{Bump--Friedberg type periods beyond the cuspidal spectrum}

Throughout the article, we fix a number field $F$. Let $\mathbb{A}:=\mathbb{A}_F$ denote the ring of ad\'eles of $F$.  For any integer $n\geq 1$, denote $\textnormal{GL}_n$ by $G_n$ and $\textnormal{SL}_n$ by $S_n$. Let $N_{n}$ denote the unipotent radical of the standard Borel subgroup of $G_n$. We also fix a non-trivial additive character $\psi:F\backslash\mathbb{A}\rightarrow \mathbb{C}^{\times}$. We use the same symbol to denote the character on $N_n(\mathbb{A})$ defined by
$$\psi(u)=\psi(u_{1,2}+u_{2,3}+\cdots+u_{n-1,n}).$$

The classical Bump--Friedberg integral and its twisted version are two-variable Rankin–Selberg integrals for general linear groups. Their unramified computations produce a product of a standard $L$-function and an exterior-square L-function. One can check \cite{BF90} and \cite{leslie2025unitary} for more details. 

The original theory is only valid for cuspidal automorphic forms for two reasons: the first is that the original period is only defined for cuspidal forms. The second is that the cuspidality provides the vanishing of proper constant terms, which guarantees that the extra terms that show up in the unfolding step will vanish. Neither property holds for general automorphic forms: their proper constant terms need not vanish, and functions of uniform moderate growth do not satisfy the decay estimates required to define the original period directly.

To extend the classical Bump--Friedberg period to general automorphic forms, we need a few more definitions (for more details, see \Cref{langlands}). Let $\mathcal{S}([G])$ denote the space of automorphic Schwartz functions on $[G]$ and let $\mathcal{T}([G])$ denote the space of smooth functions of uniform moderate growth on $[G]$. Both are endowed with the natural structure of topological vector spaces. Let $\mathfrak{X}(G)$ denote the set of cuspidal data of $G$. We have the following Langlands spectral decomposition:
$$L^2([G])=\widehat{\bigoplus}_{\chi\in \mathfrak{X}(G)}L^2_{\chi}([G]).$$
For a cuspidal datum $\chi\in \mathfrak{X}(G)$, define $\mathcal{S}_{\chi}([G]):=\mathcal{S}([G])\cap L^2_{\chi}([G])$, and let $\mathcal{T}_{\chi}([G])$ denote the orthogonal complement of $\widehat{\bigoplus}_{\chi'\neq \chi}\mathcal{S}_{\chi'}([G])$ in $\mathcal{T}([G])$. Endowed with the subspace topology inherited from $\mathcal{T}([G])$, the space $\mathcal{S}_{\chi}([G])$ is dense in $\mathcal{T}_{\chi}([G])$.

Our approach is to begin with the period on $\mathcal{S}([G])$, where the integral is absolutely convergent. We then restrict to $\mathcal{S}_{\chi}([G])$, where $\chi\in\mathfrak{X}(G)$ satisfies certain regularity conditions. These conditions force the extra terms appearing in the unfolding step to vanish. The period can consequently be unfolded into a Whittaker-type zeta integral. After analytic continuation, this zeta integral defines a continuous extension of the original period to automorphic forms of uniform moderate growth.

To be more precise, we use the first case we study (see \Cref{twisted BF}) as an example to illustrate our approach. Let $G=G_{2n}$, $H=G_n\times G_n$ and $\iota:H\rightarrow G$ be the interlacing embedding (see \Cref{notation section 3} for a precise definition). Let $\eta$ be a Hecke character. For any $f\in\mathcal{T}([G])$, we define the corresponding Whittaker function
$$W_f(g)=\int_{[N_{2n}]}f(ug)\psi^{-1}(u)\mathrm{d}u.$$
For $f\in \mathcal{S}([G])$, $\Phi\in \mathcal{S}(\mathbb{A}^n)$ and $s_1,s_2\in \mathbb{C}$, consider the period
$$\mathcal{P}_{\eta}(f,\Phi,s_1,s_2)=\int_{[G_n]\times [G_n]}f(\iota(x,y))\eta^{-1}(\textnormal{det}(x))\Theta(y,\Phi)|x|^{s_1}|y|^{s_2}\mathrm{d}x\mathrm{d}y$$
where
$$\Theta(g,\Phi)=\sum_{v\in F^n}\Phi(vg)|g|^{\frac{1}{2}},g\in [G_n]$$
is the normalized $\Theta$-series associated with $\Phi$. Under an explicit $(\eta,\Delta^*)$-regularity condition (see \Cref{regularity}) on the cuspidal datum, we prove that the restriction of the functional
$$f\mapsto \mathcal{P}_{\eta}(f,\Phi,0,0)$$
to $\mathcal{S}_{\chi}([G])$ extends continuously to $\mathcal{T}_{\chi}([G])$. The extension is characterized by the Bump--Friedberg integral $Z^{\textnormal{BF}}_{\eta}(f,\Phi,s_1,s_2)$, which is defined by the entire continuation of the integral
$$\int_{N_n(\mathbb{A})\backslash G_n(\mathbb{A})}\int_{N_n(\mathbb{A})\backslash G_n(\mathbb{A})}W_f(\iota(x,y))\eta^{-1}(\textnormal{det}(x))\Phi(e_ny)|x|^{s_1}|y|^{s_2+\frac{1}{2}}\mathrm{d}x\mathrm{d}y.$$
See \Cref{zeta integrals} for more details.

There are two more cases studied in this article. For the second case (see \Cref{augmented}), we set $G=G_1\times G_{2n}$, $H=G_n\times G_n$ and we embed $H$ into $G$ by $\nu(x,y)=(\textnormal{det}(x)^{-1},\iota(x,y))$. For the third case (see \Cref{new}), we set $G=G_{2n+1}$ and $H=S_{n+1}\times G_n$. We embed $H$ into $G$ by $\iota'$, which is a slight modification of the interlacing embedding (see \Cref{notation 5} for a precise definition). The second and third cases are treated by the same general strategy: we first define the periods on automorphic Schwartz spaces, impose suitable regularity conditions to obtain the required vanishing in the unfolding, and then use the analytically continued zeta integrals to extend the periods to automorphic functions of uniform moderate growth. The specific analytic reductions are different and will be explained in \Cref{1.4}.

\subsection{Main result}\label{1.3}

Our main result is for the latter case, where $G=G_{2n+1}$, $H=S_{n+1}\times G_n$ and we embed $H$ into $G$ by $\iota'$. For this case, the corresponding $G$-spherical variety can be written as $X=G\times ^H \textnormal{std}_{G_n}$. We define a character $\psi'$ on $N_{2n+1}(\mathbb{A})$ as follows:
$$\psi'(u)=\psi(u_{2,3}+\cdots+u_{2n,2n+1}).$$
Note that this is a degenerate character. For any $f\in \mathcal{T}([G])$, we define the corresponding degenerate Whittaker function
$$V_f(g)=\int_{[N_{2n+1}]}f(ug)\psi'^{-1}(u)\mathrm{d}u$$
for any $g\in G(\mathbb{A})$. For any $f\in \mathcal{S}([G])$ and $\Phi\in \mathcal{S}(\mathbb{A}^n)$, the period of $f$ is defined by 
$$\mathcal{P}(f,\Phi)=\int_{[S_{n+1}]\times [G_n]}f(\iota'(x,y))\Theta(y,\Phi)\mathrm{d}x\mathrm{d}y.$$
Following the approach described in the preceding subsection, for a $\Delta_1^*$-regular cuspidal datum $\chi\in \mathfrak{X}(G)$ (see \Cref{delta1}), we can prove that the functional 
$$f\mapsto \mathcal{P}(f,\Phi)$$
restricted on $\mathcal{S}_{\chi}([G])$ extends to a continuous functional $\mathcal{P}^*(\cdot,\Phi)$ on $\mathcal{T}_{\chi}([G])$. The extension is characterized by the value taken at the point $(0,0)$ of the zeta integral $Z(f,\Phi,\lambda,s_2)$ which is defined to be the entire continuation of the integral
$$\int_{N_{n+1}^1(\mathbb{A})\backslash S_{n+1}(\mathbb{A})}\int_{N_n(\mathbb{A})\backslash G_n(\mathbb{A})}V_f(\iota'(x,y))\Phi(e_ny)e^{\langle \lambda,H_{P_{(1,n)}^1}(x)\rangle}|y|^{s_2+\frac{1}{2}}\mathrm{d}x\mathrm{d}y$$
for any $f\in \mathcal{T}_{\chi}([G])$, where $N_{n+1}^1$ denotes the unipotent radical of the standard Borel subgroup of $S_{n+1}$, $P_{(1,n)}^1$ denotes the standard parabolic subgroup of $S_{n+1}$ with respect to the partition $(1,n)$ and $\lambda\in \mathfrak{a}_{P_{(1,n)}^1,\mathbb{C}}^*$. In other words, we have 
$$\mathcal{P}^*(f,\Phi)=Z(f,\Phi,0,0)$$
for any $f\in \mathcal{T}_{\chi}([G])$.

Now let $\pi$ be a unitary automorphic cuspidal representation of a standard Levi subgroup $M=\prod_{i=1}^lG_{n_i}$ of $G$. Let $P$ denote the standard parabolic subgroup of $G$ containing $M$. Further assume that the cuspidal datum $\chi$ of $(M,\pi)$ is $\Delta_1^*$-regular. Define $J_1(\pi):=\{1\leq i\leq l\;|\;n_i=1\}$ and $\Pi_i=\boxplus_{j\neq i}\pi_j$. For $i\in J_1(\pi)$, write $\pi_i=\eta_i$ for a unitary Hecke character $\eta_i$. Take a section $\varphi\in \textnormal{Ind}_{P(\mathbb{A})}^{G(\mathbb{A})}\pi$ and let $f=E(\varphi)$ be the corresponding Eisenstein series.
Thus, $E(\varphi)\in \mathcal{T}_{\chi}([G])$. Now we can state the main theorem of this article:
\begin{thm}[\Cref{main}, \textnormal{rough form}]
    Fix $\Phi\in \mathcal{S}(\mathbb{A}^n)$. Let $S$ be a sufficiently large finite set of places of $F$ including all archimedean places. After certain normalization, we have
    \begin{equation}\label{mainthm}
    \mathcal{P}^*(E(\varphi),\Phi)=(\Delta^{S,*}_{G_n\times G_n})^{-1}L(1,\pi,\hat{\mathfrak{n}}_P^-)^{-1}\sum_{i\in J_1(\pi)}\mathcal{C}_{i,S}(\varphi_S,\Phi_S)L^S(1,\eta_i^{-1}\otimes \Pi_i)L^S(1,\eta_i\otimes \Pi_i^{\vee})L^S(\frac{1}{2},\Pi_i,\eta_i^{-1}\otimes \wedge^2)
    \end{equation}
    where $\mathcal{C}_{i,S}(\varphi_S,\Phi_S)$ is a normalized local zeta integral at the places in $S$.
\end{thm}

\begin{rmk}
    We make the following remarks:
    \begin{itemize}
        \item If there exists $1\leq i<j\leq l$ such that $\pi_i\simeq \pi_j$, then $L(s,\pi,\hat{\mathfrak{n}}_P^-)$ will have poles at $s=1$. In this case, the right-hand side of the equation will be understood as $0$. Meanwhile, $E(\varphi)=0$ under this assumption. So we still have $\mathcal{P}^*(E(\varphi),\Phi)=0$.
        \item The appearance of the global $L$-factor is due to different normalizations. Define
        $$W_{\varphi}^M(g)=\int_{[N_{2n+1}\cap M]}\varphi(ug)\psi^{-1}(u)\mathrm{d}u.$$
        For $S$ sufficiently large, we can write $W_{\varphi}^M=W_{\varphi,S}^MW_{\varphi}^{M,S}$ and $W_{E(\varphi)}=W_{E(\varphi),S}W_{E(\varphi)}^S$.

        In this article, we normalize such that $W_{\varphi}^{M,S}(1)=1$. Under this normalization, by \cite[\S 4]{shahidi2011certain}, $W_{E(\varphi)}^S(1)=L^S(1,\pi,\hat{\mathfrak{n}}_P^-)^{-1}$. Therefore, if we choose the normalization such that $W_{E(\varphi)}^S(1)=1$, then the global factor cancels. Moreover, if we further assume the data are everywhere unramified (in other words, $S=\varnothing$), then we can rewrite \Cref{mainthm} as
        $$\mathcal{P}^*(E(\varphi),\Phi)=(\Delta^*_{G_n\times G_n})^{-1}\sum_{i\in J_1(\pi)}L(1,\eta_i^{-1}\otimes \Pi_i)L(1,\eta_i\otimes \Pi_i^{\vee})L(\frac{1}{2},\Pi_i,\eta_i^{-1}\otimes \wedge^2),$$
        which is compatible with the fixed-point formula predicted in \cite{ben2024relative} for the candidate dual variety considered below.
    \end{itemize}
\end{rmk}

Motivated by \cite[Table 14, Line 19]{tang2026anomaly} and the central modifications discussed in Remark 5.1 in loc. cit., we consider the following candidate for the dual variety of $X$:
$$\check{X}=G_{2n+1}\times ^{G_1\times G_{2n}}(\textnormal{std}_{G_1}^{\vee}\boxtimes \wedge^2).$$
Assuming the hypothetical global Langlands correspondence, we show in \Cref{5.5} that the fixed points of the $L$-parameter on $\check{X}$ are naturally indexed by $J_1(\pi)$, and the tangent-space contribution at the fixed point indexed by $i$ is exactly
$$L(1,\eta_i^{-1}\otimes \Pi_i)L(1,\eta_i\otimes \Pi_i^{\vee})L(\frac{1}{2},\Pi_i,\eta_i^{-1}\otimes \wedge^2).$$
Thus, both the indexing set and the tangent space representations are compatible with the fixed-point prediction of the global numerical conjecture.

\subsection{Ideas of the proofs and organization of the article}\label{1.4}

We briefly describe the ideas of the proofs. For the twisted Bump--Friedberg period, set $G=G_{2n}$ and $H=G_n\times G_n$. We first establish the absolute convergence of the associated Whittaker-type zeta integral $Z^{\textnormal{BF}}_{\eta}(f,\Phi,s_1,s_2)$ in a suitable region in $\mathbb{C}^2$ for each $f\in\mathcal{T}([G])$. The unfolding is carried out through successive Fourier expansions. At each step, the nontrivial orbit produces the next zeta integral, while the trivial orbit gives a smaller automorphic period. The regularity conditions are chosen precisely so that all such smaller periods vanish. The Poisson summation then yields a functional equation, from which we obtain the entire continuation of the zeta integral $Z_{\eta}^{\textnormal{BF}}(f,\Phi,s_1,s_2)$ and the continuous extension of the period by setting $\mathcal{P}^*_{\eta}(f,\Phi)=Z_{\eta}^{\textnormal{BF}}(f,\Phi,0,0)$. 

The augmented period of \Cref{augmented} is reduced to the twisted Bump--Friedberg period by Fourier inversion in the $G_1$-variable. Set $G=G_1\times G_{2n}$ and $H=G_n\times G_n$. We first prove the absolute convergence of the following integral
$$Z^a(f,\Phi,s_1,s_2)=\int_{N_n(\mathbb{A})\backslash G_n(\mathbb{A})}\int_{N_n(\mathbb{A})\backslash G_n(\mathbb{A})}W_f(\nu(x,y))\Phi(e_ny)|x|^{s_1}|y|^{s_2+\frac{1}{2}}\mathrm{d}x\mathrm{d}y$$
in a suitable region for each $f\in\mathcal{T}([G])$. Using the same unfolding and functional equation techniques, we can obtain the entire continuation of $Z^a(f,\Phi,s_1,s_2)$ and the continuous extension of the corresponding periods.

Although this model is formally close to the twisted period, it is an essential intermediate step in the treatment of the new period on $G_{2n+1}$. To be more precise, let $G=G_{2n+1}$, $H=S_{n+1}\times G_n$ and $Q=P_{(1,2n)}$ denote the standard parabolic subgroup of $G$ with Levi component $G_1\times G_{2n}$. After applying the Iwasawa decomposition $S_{n+1}(\mathbb{A})=P_{(1,n)}^1(\mathbb{A})K^1$ to the $S_{n+1}$-variable, we get
$$Z(f,\Phi,\lambda,s_2)=\int_{K^1}Z^a((R(k)f)_Q,\Phi,s_{\lambda}+n+1,s_2)\mathrm{d}k.$$
Thus, its analytic continuation is reduced to the case of the augmented Bump--Friedberg period in \Cref{augmented}. 

Finally, for $f=E(\varphi)$ chosen in \Cref{1.3}, the Langlands constant term formula decomposes $(R(k)E(\varphi))_Q$ into terms indexed by Weyl group elements. Cuspidality implies that a term can be nonzero only when the corresponding Weyl group element moves a $G_1$-block to the upper left corner. The surviving terms are therefore indexed by the set $J_1(\pi)$. For each $i\in J_1(\pi)$, the corresponding augmented Bump--Friedberg period is reduced to a twisted Bump--Friedberg integral on $G_{2n}$, evaluated at $(s_1,s_2)=(1/2,-1/2)$. Its Euler decomposition gives
$$L^S(1,\eta_i^{-1}\otimes \Pi_i)L^S(\frac{1}{2},\Pi_i,\eta_i^{-1}\otimes \wedge^2)$$
while the normalization of the intertwining operator contributes
$$L^S(1,\eta_i\otimes \Pi_i^{\vee}).$$
Combining these factors gives the formula of \Cref{main}.

The article is arranged as follows: \Cref{section2} recalls the notation and analytic preliminaries. \Cref{twisted BF} establishes the unfolding, functional equation, analytic continuation, and Euler decomposition of the twisted Bump--Friedberg period. In \Cref{augmented}, we deal with the augmented Bump--Friedberg period on $G_1\times G_{2n}$. In \Cref{new}, this period appears after taking the constant term along the parabolic subgroup $Q$. \Cref{new} studies the new $S_{n+1}\times G_n$-period, evaluates it on $\Delta_1^*$-regular Eisenstein series, and compares the resulting formula with the fixed-point prediction for the candidate dual variety.

\; \\
\noindent \textit{Acknowledgments:} The author thanks Weixiao Lu for introducing this problem and many helpful discussions.

\section{Notation and preliminaries}\label{section2}
\subsection{General notation}
Throughout the article, unless otherwise specified, we fix a number field $F$. Let $\mathbb{A}:=\mathbb{A}_F$. We also fix a nontrivial additive character $\psi$ on $F\backslash \mathbb{A}$. For a place $v$ of $F$, we write $F_v$ for the completion of $F$ with respect to $v$. Let $|\cdot|_v$ denote the normalized absolute value on $F_v$. For any $a\in \mathbb{A}$, we use $|a|=\prod_v|a_v|_v$ to denote its ad\'elic norm. Let $S$ be a finite set of places. We write $F_S:=\prod_{v\in S}F_v$ and $F^S=\prod'_{v\notin S}F_v$.

For a positive integer $n$, we write $G_n:=\textnormal{GL}_n$ and $S_n:=\textnormal{SL}_n$. For any $s\in \mathbb{C}$ and $g\in G_n(\mathbb{A})$, we write $|g|^s$ for $|\textnormal{det}(g)|^s$. Let $E_{ij}$ denote the matrix which is equal to $1$ on the $(i,j)$-th position and $0$ on all other places. For a ring $R$, let $R^n$ denote the row vector of size $n$ with entries in $R$. Denote the row vector $(0,\dots,0,1)$ of $R^n$ by $e_n$. For any $c\in \mathbb{R}$, let $\mathcal{H}_{>c}$ denote the half-plane $\{s\in \mathbb{C}\;|\; \textnormal{Re}(s)>c\}$.

Let $f$ and $g$ be two positive functions on a set $X$. We write 
$$f\ll g$$
if there exists a constant $C>0$ such that $f(x)\leq Cg(x)$ for every $x\in X$. We write $f\approx g$ if $f\ll g$ and $g\ll f$. We also call such two functions equivalent functions. Further assume that $f(x),g(x)>1$ for every $x\in X$. We write $f\sim g$ if there exists a real number $r_0\geq1$ such that
$$f^{\frac{1}{r_0}}\ll g\ll f^{r_0}.$$

\subsection{Preliminaries}

Throughout this subsection, we adopt the notation and conventions in \cite[\S 2]{lu2025periodsdetectingeisensteinseries} and \cite[\S 2]{beuzart2022global}. We recall only the material needed in the sequel and refer to these works for further details.

\subsubsection{Algebraic groups over a global field}

Let $G$ be a connected linear algebraic group over a global field $F$. Let $X^*(G)$ be the group of characters of $G$ defined over $F$. Let $\mathfrak{a}_G^*=X^*(G)\otimes_{\mathbb{Z}}\mathbb{R}$ and $\mathfrak{a}_G=\textnormal{Hom}_{\mathbb{Z}}(X^*(G),\mathbb{R})$. We have a canonical pairing
$$\langle\cdot,\cdot\rangle:\mathfrak{a}_G^*\times \mathfrak{a}_G\rightarrow\mathbb{R}.$$
We also have a canonical homomorphism
$$H_G:G(\mathbb{A})\rightarrow \mathfrak{a}_G$$
defined by $\langle\eta,H_G(g)\rangle=\log|\eta(g)|$ for any $g\in G(\mathbb{A})$. Denote the kernel of $H_G$ by $G(\mathbb{A})^1$. We define $[G]:=G(F)\backslash G(\mathbb{A})$ and $[G]^1:=G(F)\backslash G(\mathbb{A})^1$.

From now on we assume that $G$ is reductive, and we keep this assumption throughout the whole section. Fix the Tamagawa measure $\mathrm{d}g$ on $G(\mathbb{A})$. We denote by $A_G^{\infty}$ the neutral component of real points of the maximal split central torus of $\textnormal{Res}_{F/\mathbb{Q}}G$.

Let $P=MN$ be a semi-standard parabolic subgroup of $G$. Set $[G]_P=M(F)N(\mathbb{A})\backslash G(\mathbb{A})$ and $A_P^{\infty}:=A_{M}^{\infty}$. Define $\mathfrak{a}_P^*:=X^*(P)\otimes_{\mathbb{Z}}\mathbb{R}$ and $\mathfrak{a}_P:=\textnormal{Hom}(X^*(P),\mathbb{R})$. Denote their complexification by $\mathfrak{a}_{P,\mathbb{C}}^*$ and $\mathfrak{a}_{P,\mathbb{C}}$. We extend the homomorphism $H_P:P(\mathbb{A})\rightarrow \mathfrak{a}_P$ into the Harish-Chandra map
$$H_P:G(\mathbb{A})\rightarrow\mathfrak{a}_P$$
in such a way that for every $g\in G(\mathbb{A})$ we have $H_P(g)=H_P(p)$ where $p\in P(\mathbb{A})$ is given by the Iwasawa decomposition $g\in pK$.

Fix an embedding $\iota:G\rightarrow G_N$ for some integer $N>0$. Using $\iota$, we define a height on $G(\mathbb{A})$ by
$$\|g\|=\prod_{v}\max_{1\leq i,j\leq N}(|\iota(g)_{i,j}|_v,|\iota(g^{-1})_{i,j}|_v).$$
Note that for another choice of embedding $\iota'$ yielding a height $\|\cdot \|'$, we have $\|\cdot\|\sim \|\cdot\|'$. By the definition of the height, we can see that $\|g\|\geq 1$, $\|g\|=\|g^{-1}\|$ and $\|g\|\|h\|\geq \|gh\|$ for any $g,h\in G(\mathbb{A})$. Let $P=MN\subset G$ be a standard parabolic subgroup. We set
$$\|x\|_P=\inf_{\gamma\in M(F)N(\mathbb{A})}\|\gamma x\|.$$
Therefore, $\|g\|_G=\inf_{\gamma\in G(F)}\|\gamma g\|_{G(\mathbb{A})}$ for any $g\in G(\mathbb{A})$.

Let $n$ be a positive integer. For $x\in\mathbb{A}^n$, we define
$$\|x\|_{\mathbb{A}^n}=\prod_v\max\{|x_{1,v}|_v,\dots,|x_{n,v}|_v,1\}.$$
Note that another choice of basis would yield equivalent functions.

We have the following lemma of absolute convergence:
\begin{lemma}[{\cite[Corollary 2.2.5]{lu2025periodsdetectingeisensteinseries}}]\label[lemma]{c>1}
    For any $c>1$, there exists $N_0$ such that for any $N\geq N_0$, the integral
    $$\int_{\mathbb{A}^{\times}}\|x\|_{\mathbb{A}}^{-N}|x|^s\mathrm{d}x$$
    is absolutely convergent for $1<\textnormal{Re}(s)<c$.
\end{lemma}

\subsubsection{Spaces of functions}

Let $P=MN$ be a standard parabolic subgroup of $G$. Let $\mathcal{S}([G]_P)$ be the space of Schwartz functions on $[G]_P$. For an integer $N>0$, let $L_{-N}^2([G]_P)^{\infty}$ denote the space of smooth functions $f$ on $[G]_P$ such that 
$$\int_{[G]_P}|f(g)|^2\|g\|^{-N}\mathrm{d}g<\infty$$
and let $\mathcal{T}_N([G]_P)$ consist of smooth functions $f$ on $[G]_P$ such that
$$\|f\|_{X,-N}:=\sup_{x\in [G]_P}|R(X)f(x)|\|x\|_P^{-N}<\infty$$
for any $X\in \mathcal{U}(\mathfrak{g}_{\infty})$. Let
$$\mathcal{T}([G]_P)=\bigcup_{N>0}\mathcal{T}_{N}([G]_P)=\bigcup_{N>0}L_{-N}^2([G]_P)^{\infty}$$
denote the space of functions of uniform moderate growth on $[G]_P$. Note that all the spaces of functions mentioned above are LF spaces. By \cite[\S 2.5.7]{beuzart2022global}, $\mathcal{S}([G]_P)$ is dense in $L_{-N}^2([G]_P)^{\infty}$ for any $N>0$. Let $Q$ be another standard parabolic subgroup of $G$ containing $P$. We have the following constant term map
$$\mathcal{T}([G]_Q)\ni f\mapsto f_P:=\left( g\mapsto\int_{[N_P]}f(ng)\mathrm{d}n\right)\in \mathcal{T}([G]_P)$$
and the pseudo-Eisenstein map
$$\mathcal{S}([G]_P)\ni \varphi\mapsto E_P^Q(\varphi,\cdot):=\left( g\mapsto\sum_{\gamma\in P(F)\backslash Q(F)}\varphi(\gamma g)\right)\in \mathcal{S}([G]_Q).$$

For a function $\Phi\in \mathcal{S}(\mathbb{A}^n)$, we associate the following $\Theta$-series:
$$\Theta(g,\Phi)=\sum_{v\in F^n}\Phi(vg)|g|^{\frac{1}{2}},g\in [G_n].$$
The convergence and growth of the $\Theta$-series are justified by \cite[Lemma 2.3.3]{lu2025periodsdetectingeisensteinseries}. By the Poisson summation formula, $\Theta$-series satisfies
\begin{equation}\label{theta series}
    \Theta(g,\Phi)=\Theta({^tg^{-1},\hat{\Phi}}).
\end{equation}

Next, we introduce some estimates on Fourier coefficients. Let $P=MN$ be a standard parabolic subgroup of $G$. Recall that we fix a nontrivial additive character $\psi$ on $F\backslash\mathbb{A}$. Let $V_P$ be the vector space of additive algebraic characters $N\rightarrow \mathbb{G}_a$. Let $l\in V_P(F)$ and set $\psi_l:=\psi\circ l_{\mathbb{A}}:[N]\rightarrow \mathbb{C}^{\times}$ where $l_{\mathbb{A}}$ denotes the homomorphism between ad\'elic points $N(\mathbb{A})\rightarrow \mathbb{A}$. For any $f\in \mathcal{T}([G])$, we set
$$f_{N,\psi_l}(g)=\int_{[N]}f(ug)\psi_l(u)^{-1}\mathrm{d}u$$
for any $g\in G(\mathbb{A})$. Denote the action of $M$ on $V_P$ induced by the adjoint action of $M$ on $N$ by $\textnormal{Ad}^*$. We have the following lemma:
\begin{lemma}[{\cite[Lemma 2.3.5]{lu2025periodsdetectingeisensteinseries}}]\label[lemma]{important lemma}
    Let $G(\mathbb{A})=P(\mathbb{A})K$ be the Iwasawa decomposition with respect to $P$. We have the following statement:
    \begin{itemize}
        \item[(1)] There exists $c>0$ such that for every $N_1,N_2\geq 0$, 
        $$f\mapsto \sup_{m\in M(\mathbb{A})}\sup_{k\in K}\|\textnormal{Ad}^*(m^{-1})l\|_{V_P(\mathbb{A})}^{N_1}\|m\|_M^{N_2}\delta_P(m)^{cN_2}|f_{N,\psi_l}(mk)|$$
        is a continuous semi-norm on $\mathcal{S}([G])$.
        \item[(2)] Let $N>0$. Then, for every $N_1\geq 0$, 
        $$f\mapsto \sup_{m\in M(\mathbb{A})}\sup_{k\in K}\|\textnormal{Ad}^*(m^{-1})l\|_{V_P(\mathbb{A})}^{N_1}\|m\|_M^{-N}|f_{N,\psi_l}(mk)|$$
        is a continuous semi-norm on $\mathcal{T}_N([G])$. 
    \end{itemize}
\end{lemma}

\subsubsection{Automorphic Forms and Eisenstein series}

Let $P=M_PN_P$ be a standard parabolic subgroup of $G$. An automorphic form on $[G]_P$ is, by definition, a $Z(\mathfrak{g}_{\infty})$-finite function in $\mathcal{T}([G]_P)$. We denote by $\mathcal{A}_P(G)$ the set of automorphic forms on $[G]_P$. Let $\mathcal{A}_{P,\textnormal{cusp}}(G)$ denote the subspace of $\mathcal{A}_P(G)$ consisting of cuspidal automorphic forms. When $P=G$, we will omit it from the subscript.

A cuspidal automorphic representation (or simply cuspidal representation) of $G$ is defined to be a topologically irreducible subrepresentation $\pi\subset \mathcal{A}_{\textnormal{cusp}}(G)$ of $G(\mathbb{A})$. Note that a cuspidal representation $\pi$ is unitary if and only if its central character is unitary.

Now let $\pi$ be a cuspidal representation of $M_P(\mathbb{A})$. Let $\mathcal{A}_{\pi,\textnormal{cusp}}(M_P)$ denote the sum of all cuspidal representations of $M_P(\mathbb{A})$ that are isomorphic to $\pi$. Let $\Pi=\textnormal{Ind}_{P(\mathbb{A})}^{G(\mathbb{A})}\pi$ be the normalized smooth induction of $\pi$ that we identify with the space of forms $\varphi\in \mathcal{A}_P(G)$ such that
$$m\in [M_P]\mapsto \delta_P^{-1/2}(m)\varphi(mg)$$
belongs to $\pi$ for any $g\in G(\mathbb{A})$.

For $\varphi\in \Pi$ and $\lambda\in \mathfrak{a}_{P,\mathbb{C}}^*$, we define the Eisenstein series as
$$E(g,\varphi,\lambda)=\sum_{\gamma\in P(F)\backslash G(F)}\exp(\langle\lambda,H_P(\gamma g)\rangle)\varphi(\gamma g).$$
The above expression is absolutely convergent when $\textnormal{Re}(\lambda)$ lies in some cones and has meromorphic continuation by \cite{langlands2006functional}, \cite{bernstein2024meromorphic} and \cite{lapid2008remark}.

\subsubsection{Cuspidal data and Langlands decomposition}\label{langlands}

Let $\underline{\mathfrak{X}}(G)$ denote the set of pairs $(M_P,\pi)$, where $M_P$ is a standard Levi component of a standard parabolic subgroup $P\subset G$ and $\pi$ is a cuspidal representation of $M_P(\mathbb{A})$ with central character trivial on $A_P^{\infty}$. Two elements $(M_P,\pi)$ and $(M_Q,\pi')$ are called equivalent if there exists $g\in G(F)$ such that $gM_Pg^{-1}=M_Q$ and $g\pi g^{-1}=\pi'$. Let $\mathfrak{X}(G)$ denote the equivalence class of $\underline{\mathfrak{X}}(G)$, and an element $\chi\in \mathfrak{X}(G)$ will be called a cuspidal datum. Sometimes we write $\chi=(M,\pi)$, which means that we pick a representative element $(M,\pi)$ in the equivalence class. Then for $\chi=(M,\pi)$, we denote by $\chi^{\vee}$ the cuspidal datum associated to $(M,\pi^{\vee})$, where $\pi^{\vee}$ stands for the complex conjugate of $\pi$.

For a standard parabolic subgroup $P\subset G$, there exists a natural map $\underline{\mathfrak{X}}(M_P)\rightarrow \underline{\mathfrak{X}}(G)$. It induces a map $\mathfrak{X}(M_P)\rightarrow \mathfrak{X}(G)$ with finite fiber. For each element $\chi\in \mathfrak{X}(G)$, denote its inverse image in $\mathfrak{X}(M_P)$ by $\chi^{M_P}$.

For any $\chi\in\mathfrak{X}(G)$, and $P=M_PN_P$ a standard parabolic subgroup of $G$, we write $\mathfrak{O}_{\chi}^P\subset \mathcal{S}([G]_P)$ the set of pseudo-Eisenstein series with respect to $\chi$ (see \cite[\S 2.9]{beuzart2022global} for details). Let $L_{\chi}^2([G]_P)$ denote the closure of $\mathfrak{O}_{\chi}^P$ in $L^2([G]_P)$, then we have the Langlands decomposition:
\begin{equation}
    L^2([G]_P)=\widehat{\bigoplus}_{\chi\in\mathfrak{X}(G)}L^2_{\chi}([G]_P),
\end{equation}
where the decomposition is orthogonal. For a subset $\mathfrak{X}\subset \mathfrak{X}(G)$, we write $L^2_{\mathfrak{X}}([G]_P):=\hat{\oplus}_{\chi\in \mathfrak{X}}L^2_{\chi}([G]_P)$. Then we define
\begin{equation}
    \mathcal{S}_{\mathfrak{X}}([G]_P):=L^2_{\mathfrak{X}}([G]_P)\cap \mathcal{S}([G]_P).
\end{equation}
Note that $ \mathcal{S}_{\mathfrak{X}}([G]_P)$ is a closed subspace of $\mathcal{S}([G]_P)$, since it is the orthogonal complement of $\cup_{\chi\notin \mathfrak{X}}\mathfrak{O}_{\chi}^P$ in $\mathcal{S}([G]_P)$.

We define $\mathcal{T}_{\mathfrak{X}}([G]_P)$ (resp. $L^2_{-N,\mathfrak{X}}([G]_P)^{\infty}$) be the orthogonal complement of $\mathcal{S}_{\mathfrak{X}^c}([G]_P)$ in $\mathcal{T}([G]_P)$ (resp. $L^2_{-N}([G]_P)^{\infty}$), where $\mathfrak{X}^c$ denotes the complement of $\mathfrak{X}$ in $\mathfrak{X}(G)$.

\section{Twisted Bump--Friedberg period}\label{twisted BF}

In this section, we mainly discuss the twisted Bump--Friedberg period of a general automorphic form, not necessarily cuspidal.

\subsection{Statements of the main results}
\subsubsection{Notations}\label{notation section 3}
For any positive integer $n$, let $G=G_{2n}$ and $H=G_n\times G_n$. We define an embedding $\iota:H\rightarrow G$ as follows: put the first $G_n$ into all positions with odd index, and put the second $G_n$ into all positions with even index. The following is an example:

\begin{eg}
    When $n=3$, the embedding of $H$ into $G$ is as follows:
    $$\begin{pmatrix}
        * &  & * &  & *&  \\
         & \circ & & \circ & & \circ \\
         * &  & * &  & *& \\
         & \circ & & \circ & & \circ \\
         * &  & * &  & *& \\
         & \circ & & \circ & & \circ \\
    \end{pmatrix}.$$
\end{eg}

Similarly, for $G=G_{2n+1}$ and $H=G_{n+1}\times G_n$, we can also define an embedding from $H$ to $G$, which is also called by $\iota$, by putting the first $G_{n+1}$ factor into all positions with odd index, and putting the second $G_n$ factor into all positions with even index. The following is an example:

\begin{eg}
    When $n=2$, the embedding of $H$ into $G$ is as follows:
    $$\begin{pmatrix}
        * &  & * &  & *\\
         & \circ & & \circ &\\
         * &  & * &  & *\\
         & \circ & & \circ &\\
         * &  & * &  & *\\
    \end{pmatrix}.$$
\end{eg}
We denote both embedding by the same symbol $\iota$. For a matrix with even size, we mean the first embedding; and for a matrix with odd size, we mean the second embedding.

Throughout this section, we fix a positive integer $n$, and let $G=G_{2n}$, $H=G_n\times G_n$. We embed $H$ into $G$ by $\iota$. Let $\mathcal{P}_r$ denote the standard mirabolic subgroup of $G_r$, and if $r\leq n$, we embed $\mathcal{P}_r$ into $G_n$ by putting it into the upper left corner.

We also introduce the following definitions:
\begin{itemize}
    \item For $0\leq r\leq n$, let $P_{r,n}$ denote the standard parabolic subgroup of $G_n$ with Levi component $G_r\times (G_1)^{n-r}$ and let $N_{r,n}$ denote the unipotent radical of $P_{r,n}$. Let $N_n=N_{0,n}$ be the upper triangular maximal unipotent subgroup of $G_n$.
    \item We also write $\psi$ for the character on $N_n(\mathbb{A})$ defined by $\psi(u)=\psi(u_{1,2}+\dots +u_{n-1,n})$. The character induced by $\psi$ on the subgroup $N_{r,n}$ will also be denoted by $\psi$. 
\end{itemize}

We also fix a character $\eta$ on $[G_1]$. We have the following fact:
\begin{fact}\label{fact 1}
    There exists a real number a such that $|\eta^{-1}(g)|=|g|^a$ for all $g\in [G_1]$
\end{fact}.

\subsubsection{The period}\label{twisted BF period}

We define a bilinear functional $\mathcal{S}([G])\times\mathcal{S}(\mathbb{A}^n)\times \mathbb{C}^2 \mapsto \mathbb{C}$ by
$$\mathcal{P}_{\eta}(f,\Phi,s_1,s_2)=\int_{[G_n]\times [G_n]}f(\iota(x,y))\eta^{-1}(\textnormal{det}(x))\Theta(y,\Phi)|x|^{s_1}|y|^{s_2}\mathrm{d}x\mathrm{d}y.$$

By \cite[Lemma 2.3.3]{lu2025periodsdetectingeisensteinseries} and \cite[Proposition A.1.1]{Beuzart-Plessis_2021}, together with \Cref{fact 1}, the integral defining $\mathcal{P}_{\eta}$ is absolutely convergent and defines a continuous bilinear functional on $\mathcal{S}([G])\times\mathcal{S}(\mathbb{A}^n)$ for any $s_1,s_2\in \mathbb{C}$. Moreover, for any $f\in \mathcal{S}([G])$ and $\Phi\in \mathcal{S}(\mathbb{A}^n)$, $\mathcal{P}_{\eta}(f,\Phi,s_1,s_2)$ defines an entire function on $\mathbb{C}^2$ which is bounded on vertical strips.

\subsubsection{Zeta integrals}\label{zeta integrals}

For $f\in \mathcal{T}([G])$ and $0\leq r\leq 2n$, we define the corresponding Fourier coefficient
$$f_{N_{r,2n,\psi}}(g)=\int_{[N_{r,2n}]}f(ug)\psi^{-1}(u)\mathrm{d}u.$$
We associate its Whittaker function 
$$W_f(g)=f_{N_{2n},\psi}(g)=\int_{[N_{2n}]}f(ug)\psi^{-1}(u)\mathrm{d}u.$$

Let $f\in \mathcal{T}([G])$, $\Phi\in \mathcal{S}(\mathbb{A}^n)$ and $s_1,s_2\in \mathbb{C}$. Define the following integral:
$$Z^{\textnormal{BF}}_{\eta}(f,\Phi,s_1,s_2)=\int_{N_n(\mathbb{A})\backslash G_n(\mathbb{A})}\int_{N_n(\mathbb{A})\backslash G_n(\mathbb{A})}W_f(\iota(x,y))\eta^{-1}(\textnormal{det}(x))\Phi(e_ny)|x|^{s_1}|y|^{s_2+\frac{1}{2}}\mathrm{d}x\mathrm{d}y.$$
The following lemma will provide the absolute convergence of the above expression of integral for $\textnormal{Re}(s_1)$ and $\textnormal{Re}(s_2)$ large enough.

\begin{lemma}\label[lemma]{absolute convergence 1}
    For any $N\in \mathbb{N}$, there exists $c_N>0$ such that
    \begin{itemize}
        \item[(1)] For any $f\in \mathcal{T}_N([G])$ and $\Phi\in \mathcal{S}(\mathbb{A}^n)$, the expression defining $Z^{\textnormal{BF}}_{\eta}(f,\Phi,s_1,s_2)$ converges absolutely for $s_1,s_2\in \mathcal{H}_{>c_N}$, and defines a holomorphic function on $\mathcal{H}_{>c_N}\times \mathcal{H}_{>c_N}$.
        \item[(2)] The functional $f\mapsto Z^{\textnormal{BF}}_{\eta}(f,\Phi,s_1,s_2)$ is continuous on $\mathcal{T}_N([G])$ for any $s_1,s_2\in \mathcal{H}_{>c_N}$ and $\Phi\in \mathcal{S}(\mathbb{A}^n)$.
    \end{itemize}
\end{lemma}

The proof of the above lemma will be given in \Cref{proof of absolute convergence}.

\subsubsection{$(\eta,\Delta^*)$-regular cuspidal datum}\label{regularity}

Let $\chi\in \mathfrak{X}(G)$ be a cuspidal datum of $G$. Write $\chi=(M,\pi)$, where $M=\prod_{i=1}^lG_{n_i}$ and $\pi=\pi_1\boxtimes \pi_2\boxtimes \dots \boxtimes \pi_l$, with each $\pi_i$ a cuspidal representation of $G_{n_i}$. 

\begin{defn}\label[defn]{definition}
We call $\chi$ $(\eta,\Delta^*)$-regular, if it satisfies the following conditions:
\begin{itemize}
    \item[(1)] For any $i\neq j$ and $s\in \mathbb{C}$, $\pi_i\otimes \eta^{-1}|\cdot|^{s}$ is not isomorphic to $\pi_j^{\vee}$;
    \item[(2)] If $n_i$ is even, the $L$-function $L(x,\pi_i,\wedge^2\otimes \eta^{-1}|\cdot|^s)$ is holomorphic at $x=1$ for any $s\in \mathbb{C}$.
    \item[(3)] For any $1\leq i\leq l$ such that $n_i=1$, $\eta|_{[G_1]^1}\neq \pi_i|_{[G_1]^1}$.
\end{itemize}
\end{defn}

Let $\mathfrak{X}_{(\eta,\Delta^*)}(G)\subset \mathfrak{X}(G)$ denote the set of $(\eta,\Delta^*)$-regular cuspidal datum. We write $\mathcal{S}_{(\eta,\Delta^*)}([G])$ (resp. $\mathcal{T}_{(\eta,\Delta^*)}([G])$) for $\mathcal{S}_{\mathfrak{X}_{(\eta,\Delta^*)}(G)}([G])$ (resp. $\mathcal{T}_{\mathfrak{X}_{(\eta,\Delta^*)}(G)}([G])$).

\begin{rmk}\label[rmk]{Remark 1}
    The notion of $(\eta,\Delta^*)$-regularity is not restricted to cuspidal data. More generally, let $M$ be a standard Levi subgroup of $G$ and $\pi$ a cuspidal representation on $M$. We may define $(\eta,\Delta^*)$-regularity for $\pi$ by the same conditions, without requiring its central character to be trivial on $A_M^{\infty}$.

    However, the notion is invariant under twists in the $A_M^{\infty}$-direction. More precisely, suppose that $\pi$ and $\pi'$ are two cuspidal representations of $M(\mathbb{A})$ such that 
    $$\pi'=\pi\otimes e^{\langle\lambda,H_M(\cdot)\rangle}$$
    for some $\lambda\in i\mathfrak{a}_{M}^*$, then $\pi$ is $(\eta,\Delta^*)$-regular if and only if $\pi'$ is $(\eta,\Delta^*)$-regular.

    Moreover, for any two Hecke characters $\eta,\eta_0$ such that  $\eta=\eta^0|\cdot|^{u_\eta}$ for some $u_{\eta}\in \mathbb{C}$, $(\eta,\Delta^*)$-regularity is equivalent to $(\eta_0,\Delta^*)$-regularity.
\end{rmk}

\subsubsection{Main results}

We state the main result of this section:

\begin{thm}\label[thm]{main result}
    Let $\chi$ be a $(\eta,\Delta^*)$-regular cuspidal datum. Fix $\Phi\in \mathcal{S}(\mathbb{A}^n)$. We have the following statements:
    \begin{itemize}
        \item[(1)] The continuous functional $\mathcal{P}_{\eta}(\cdot,\Phi,0,0)$ on $\mathcal{S}_{\chi}([G])$ extends by continuity to a continuous functional $\mathcal{P}^*_{\eta}(\cdot,\Phi)$ on $\mathcal{T}_{\chi}([G])$.
        \item[(2)] For every $f\in \mathcal{T}_{\chi}([G])$, the function $(s_1,s_2)\mapsto Z^{\textnormal{BF}}_{\eta}(f,\Phi,s_1,s_2)$, a priori defined on the region $\mathcal{H}_{>c_N}\times \mathcal{H}_{>c_N}$, extends to an entire function on $\mathbb{C}^2$. Moreover, we have
        $$\mathcal{P}^*_{\eta}(f,\Phi)=Z^{\textnormal{BF}}_{\eta}(f,\Phi,0,0).$$
    \end{itemize}
\end{thm}

\subsection{Absolute convergence of Zeta integrals}\label{proof of absolute convergence}

We will introduce two more Zeta integrals. For $f\in \mathcal{S}([G])$, $\Phi\in\mathcal{S}(\mathbb{A}^n)$ and $1\leq r\leq n$, we put
$$Z_{r,\eta}(f,\Phi,s_1,s_2)=\int_{\mathcal{P}_r(F)N_{r,n}(\mathbb{A})\backslash G_n(\mathbb{A})}\int_{G_{r-1}(F)N_{r-1,n}(\mathbb{A})\backslash G_n(\mathbb{A})}f_{N_{2r-1,2n,\psi}}(\iota(x,y))\eta^{-1}(\textnormal{det}(x))\Phi(e_ny)|x|^{s_1}|y|^{s_2+\frac{1}{2}}\mathrm{d}x\mathrm{d}y.$$
and
$$Z_{r,\eta}'(f,\Phi,s_1,s_2)=\int_{G_r(F)N_{r,n}(\mathbb{A})\backslash G_n(\mathbb{A})}\int_{\mathcal{P}_{r}(F)N_{r,n}(\mathbb{A})\backslash G_n(\mathbb{A})}f_{N_{2r,2n,\psi}}(\iota(x,y))\eta^{-1}(\textnormal{det}(x))\Phi(e_ny)|x|^{s_1}|y|^{s_2+\frac{1}{2}}\mathrm{d}x\mathrm{d}y.$$
Note that $Z_{1,\eta}(f,\Phi,s_1,s_2)=Z^{\textnormal{BF}}_{\eta}(f,\Phi,s_1,s_2)$. We have the following lemma:
\begin{lemma}\label[lemma]{absolute convergence 3}
    For any $f\in \mathcal{S}([G])$, $\Phi\in\mathcal{S}(\mathbb{A}^n)$, $s_1,s_2\in \mathbb{C}$ and $1\leq r\leq n$, the expressions defining $Z_{r,\eta}(f,\Phi,s_1,s_2)$ and $Z_{r,\eta}'(f,\Phi,s_1,s_2)$ are absolutely convergent.
\end{lemma}
The proof of this lemma will be given below. 

\subsubsection{Proof of \Cref{absolute convergence 1}}\label{T_N proof}
\begin{proof}
    Let $B_H$ denote the standard Borel subgroup of $H$. By the Iwasawa decomposition $H(\mathbb{A})=B_H(\mathbb{A})K_H$ and \Cref{fact 1}, we need to find a real number $c_N>0$ such that the following integral
    \begin{align}\label{BF1}
        \int_{K_H}\int_{(\mathbb{A}^{\times})^{2n}} & |W_f(\iota(D(a_1,\dots,a_n),D(b_1,\dots,b_n))k)| |\Phi(b_ne_nk)| \delta_{B_H}(D(a_1,\dots,a_n),D(b_1,\dots,b_n))^{-1}\notag \\ & \prod_{i=1}^n|a_i|^{\textnormal{Re}(s_1)+a}\prod_{i=1}^n|b_i|^{\textnormal{Re}(s_2)+\frac{1}{2}} \prod_{i=1}^n\mathrm{d}a_i\mathrm{d}b_i\;\mathrm{d}k
    \end{align}
    is absolutely convergent when $s_1,s_2\in \mathcal{H}_{>c_N}$, where
    $$D(a_1,\dots,a_n)=\textnormal{diag}(a_1,\dots,a_n)$$
    and the modulus character is given by
    $$\delta_{B_H}(D(a_1,\dots,a_n),D(b_1,\dots,b_n))=\prod_{i=1}^n|a_ib_i|^{n+1-2i}.$$
    By \Cref{important lemma}(2), for every $N_1\geq 0$,
    $$|W_f(\iota(D(a_1,\dots,a_n),D(b_1,\dots,b_n))k)|\ll \prod_{i=1}^n\|\frac{a_i}{b_i}\|_{\mathbb{A}}^{-N_1}\prod_{i=1}^{n-1}\|\frac{b_i}{a_{i+1}}\|_{\mathbb{A}}^{-N_1}\prod_{i=1}^n\|a_i\|^N_{G_1}\|b_i\|^N_{G_1}$$
    for any $(k,a_1,\dots,a_n,b_1,\dots,b_n)\in K_H\times (\mathbb{A}^{\times})^{2n}$. Since $\Phi\in \mathcal{S}(\mathbb{A}^n)$, for all $N_1>0$, we have $|\Phi(b_ne_nk)|\ll\|b_n\|_{\mathbb{A}}^{-N_1}$ for any $(k,b_n)\in K_H\times \mathbb{A}^{\times}$. Therefore, for any $N_2>0$, there exists $N_1$ such that
    $$\prod_{i=1}^n\|\frac{a_i}{b_i}\|_{\mathbb{A}}^{-N_1}\prod_{i=1}^{n-1}\|\frac{b_i}{a_{i+1}}\|_{\mathbb{A}}^{-N_1}\|b_n\|_{\mathbb{A}}^{-N_1}\ll \prod_{i=1}^n\|a_i\|_{\mathbb{A}}^{-N_2}\prod_{i=1}^n\|b_i\|_{\mathbb{A}}^{-N_2}.$$
    Therefore, \Cref{BF1} is essentially bounded by
    \begin{equation}\label{BF2}
        \prod_{i=1}^n\int_{G_1}\|a_i\|_{\mathbb{A}^{\times}}^N|a_i|^{\textnormal{Re}(s_1)+a+2i-n-1}\|a_i\|_{\mathbb{A}}^{-N_2}\prod_{i=1}^n\int_{\mathbb{A}^{\times}}\|b_i\|_{G_1}^N|b_i|^{\textnormal{Re}(s_2)-\frac{1}{2}+2i-n}\|b_i\|_{\mathbb{A}}^{-N_2}.
    \end{equation}
    Note that there exists $M>0$ such that $\|x\|_{G_1}\ll \max(|x|^M,|x|^{-M})$. Together with \cite[Lemma 2.2.5]{lu2025periodsdetectingeisensteinseries}, there exists $c_N>0$ such that when $(s_1,s_2)\in \mathcal{H}_{>c_N}\times \mathcal{H}_{>c_N}$, \Cref{BF2} is absolutely convergent. Moreover, the above estimate is uniform when $(s_1,s_2)$ ranges over a compact subset of $\mathcal{H}_{>c_N}\times \mathcal{H}_{>c_N}$. Hence the integral converges locally uniformly and defines a holomorphic function of $(s_1,s_2)$.

    Finally, the continuity follows from the fact that the implied constant in the Whittaker estimate is controlled by a continuous seminorm on $\mathcal{T}_N([G])$, as follows from \Cref{important lemma}(2). The proof is complete.
\end{proof}

\subsubsection{Proof of \Cref{absolute convergence 3}}

\begin{proof}
 We fix the functions $f\in\mathcal{S}([G])$ and $\Phi\in\mathcal{S}(\mathbb{A}^n)$.  We first prove that $Z_{r,\eta}(f,\Phi,s_1,s_2)$ is absolutely convergent for any $1\leq r\leq n$ and $s_1,s_2\in \mathbb{C}$.

 By the Iwasawa decomposition $G_n(\mathbb{A})=P_{r,n}(\mathbb{A})K$ and $G_n(\mathbb{A})=P_{r-1,n}(\mathbb{A})K$, we need to prove the integral
 \begin{align}\label{BF3}
     \int_{K}\int_{\mathcal{P}_r(F)\backslash G_r(\mathbb{A})\times (\mathbb{A}^{\times})^{n-r}} \int_{K}\int_{[G_{r-1}]\times (\mathbb{A}^{\times})^{n-r+1}} & f_{N_{2r-1,2n,\psi}} (\iota(D(m_1,a_{r+1},\dots,a_n)k_1,D(m_2,b_r,\dots,b_n)k_2)) \notag \\
     & \delta_{P_{r,n}}(D(m_1,a_{r+1},\dots,a_n))^{-1}\delta_{P_{r-1,n}}(D(m_2,b_r,\dots,b_n))^{-1}|\Phi(b_ne_nk_2)|\notag \\
     & |m_1|^{\textnormal{Re}(s_1)+a}|m_2|^{\textnormal{Re}(s_2)+\frac{1}{2}}\prod_{i=r+1}^n|a_i|^{\textnormal{Re}(s_1)+a}\prod_{i=r}^n|b_i|^{\textnormal{Re}(s_2)+\frac{1}{2}} \notag \\
     &\mathrm{d} m_1 \mathrm{d}m_2\prod_{i=r+1}^n\mathrm{d}a_i\prod_{i=r}^n\mathrm{d}b_i\;\mathrm{d}k_1\mathrm{d}k_2
 \end{align}
 is absolutely convergent for any $(s_1,s_2)\in \mathbb{C}^2$, where the modulus character is given as follows:
 $$\delta_{P_{r,n}}(D(m_1,a_{r+1},\dots,a_n))=|m_1|^{n-r}\prod_{i=r+1}^n|a_i|^{n+1-2i};$$
 $$\delta_{P_{r-1,n}}(D(m_2,b_r,\dots,b_n))=|m_2|^{n-r+1}\prod_{i=r}^n|b_i|^{n+1-2i}.$$

 Let $m^1=D(m_1,a_{r+1},\dots,a_n)$ and $m^2=D(m_2,b_r,\dots,b_n)$. Denote the upper left $(2r-1)\times (2r-1)$-block of $\iota(m^1,m^2)$ by $m^0$. By \Cref{important lemma}(1), there exists $c>0$ such that for every $N_1,N_2\geq 0$ and $k_1,k_2\in K$, 
 $$|f_{N_{2r-1,2n,\psi}}(\iota(m^1k_1,m^2k_2))|\ll \|\textnormal{Ad}^*(\iota(m^1,m^2)^{-1})l\|_{\mathbb{A}^{2r-1}}^{-N_1}\|\iota(m^1,m^2)\|_{M_{P_{2r-1,2n}}}^{-N_2}\delta_{P_{2r-1,2n}}(\iota(m^1,m^2))^{-cN_2}$$
 where the modulus character is given by
 $$\delta_{P_{2r-1,2n}}(\iota(m^1,m^2))=|m_1|^{2n-2r+1}|m_2|^{2n-2r+1}\prod_{i=r+1}^n|a_i|^{2n+3-4i}\prod_{i=r}^n|b_i|^{2n+1-4i}$$
 and $l=(e_{2r-1},1,\dots,1)\in \mathbb{A}^{2n-1}$ (the first $1$ appears at the $2r-1$-th position). Here, the action on $l$ is given by
 $$\textnormal{Ad}^*(\iota(m^1,m^2)^{-1})l=(b_r^{-1}e_{2r-1}m^0, \frac{b_r}{a_{r+1}},\frac{a_{r+1}}{b_{r+1}},\dots,\frac{b_{n-1}}{a_n},\frac{a_n}{b_n}).$$
 Therefore, we have the following estimate:
 \begin{align*}
 \|\textnormal{Ad}^*(\iota(m^1,m^2)^{-1})l\|_{\mathbb{A}^{2r-1}} & \sim \|b_r^{-1}e_{2r-1}m^0\|_{\mathbb{A}^{2r-1}}\prod_{i=r}^{n-1}\|b_ia_{i+1}^{-1}\|_{\mathbb{A}}\prod_{i=r+1}^n\|a_ib_i^{-1}\|_{\mathbb{A}} \\ & = \|b_r^{-1}e_{r}m_1\|_{\mathbb{A}^{r}}\prod_{i=r}^{n-1}\|b_ia_{i+1}^{-1}\|_{\mathbb{A}}\prod_{i=r+1}^n\|a_ib_i^{-1}\|_{\mathbb{A}}.
 \end{align*}
 Since $\Phi\in\mathcal{S}(\mathbb{A}^n)$, for any $N_1> 0$ and $k_2\in K$, we have $|\Phi(b_ne_nk_2)|\ll\|b_n\|_{\mathbb{A}}^{-N_1}$. Therefore, for any $N_3>0$, there exists $N_1>0$ such that
 $$\|b_r^{-1}e_{r}m_1\|_{\mathbb{A}^{r}}^{-N_1}\prod_{i=r}^{n-1}\|b_ia_{i+1}^{-1}\|_{\mathbb{A}}^{-N_1}\prod_{i=r+1}^n\|a_ib_i^{-1}\|_{\mathbb{A}}^{-N_1} \|b_n\|_{\mathbb{A}}^{-N_1}\ll \|e_{r}m_1\|_{\mathbb{A}^{r}}^{-N_3}\prod_{i=r}^{n}\|b_i\|_{\mathbb{A}}^{-N_3}\prod_{i=r+1}^n\|a_i\|_{\mathbb{A}}^{-N_3}.$$
 Moreover, we have 
 $$\|\iota(m^1,m^2)\|_{M_{P_{2r-1,2n}}}\sim \|m_1\|_{G_r}\|m_2\|_{G_{r-1}}\prod_{i=r+1}^n\|a_i\|_{G_1}\prod_{i=r}^n\|b_i\|_{G_1}.$$
 Let $a_1=\textnormal{Re}(s_1)$ and $a_2=\textnormal{Re}(s_2)$. Then \Cref{BF3} is essentially bounded by the product of the following integrals:
 \begin{equation}\label{BF4}
     \int_{\mathcal{P}_r(F)\backslash G_r(\mathbb{A})}\|e_rm_1\|_{\mathbb{A}^{r}}^{-N_3}\|m_1\|_{G_r}^{-N_2}|m_1|^{-cN_2(2n-2r+1)+a+a_1+r-n}\;\mathrm{d}m_1;
 \end{equation}
 \begin{equation}\label{BF5}
     \prod_{i=r+1}^n\int_{\mathbb{A}^{\times}}\|a_i\|_{\mathbb{A}}^{-N_3}\|a_i\|_{G_1}^{-N_2}|a_i|^{-cN_2(2n+3-4i)+a_1+a+2i-n-1}\;\mathrm{d}a_i;
 \end{equation}
 \begin{equation}\label{BF6}
     \int_{[G_{r-1}]}\|m_2\|_{G_{r-1}}^{-N_2}|m_2|^{-cN_2(2n-2r+1)+a_2+\frac{1}{2}+r-n-1}\;\mathrm{d}m_2;
 \end{equation}
 \begin{equation}\label{BF7}
     \prod_{i=r}^n\int_{\mathbb{A}^{\times}}\|b_i\|_{\mathbb{A}}^{-N_3}\|b_i\|_{G_1}^{-N_2}|b_i|^{-cN_2(2n+1-4i)+a_2-\frac{1}{2}+2i-n}\;\mathrm{d}b_i
 \end{equation}
 where $N_2,N_3>0$ can be chosen arbitrarily. Note that 
 $$|-cN_2(2n-2r+1)+a+a_1+r-n-(-cN_2(2n-2r+1)+a_2+\frac{1}{2}+r-n-1)|\leq |a|+\frac{1}{2}+|a_1-a_2|.$$
 Let $C_{|a_1-a_2|}:=|a|+\frac{1}{2}+|a_1-a_2|$, where the subindex means that this number is only determined by $|a_1-a_2|$. Now fix $|a_1-a_2|$, and let $C=3C_{|a_1-a_2|}$. By \cite[Corollary 2.3.4]{lu2025periodsdetectingeisensteinseries}, there exists $N_0>0$ such that for any $N,N'>N_0$, the integrals
 $$\int_{\mathcal{P}_r(F)\backslash G_r(\mathbb{A})}\|e_rm_1\|_{\mathbb{A}^{r}}^{-N}\|m_1\|_{G_r}^{-N'}|m_1|^{\textnormal{Re}(s)}\;\mathrm{d}m_1$$
 and
 $$\int_{[G_{r-1}]}\|m_2\|_{G_{r-1}}^{-N'}|m_2|^{\textnormal{Re}(s)}\;\mathrm{d}m_2$$
 are both convergent for $0<\textnormal{Re}(s)<C$. When $a_1$ is large enough, we can always find $N'>N_0$ such that 
 $$C_{|a_1-a_2|}<-cN_2(2n-2r+1)+a+a_1+r-n<2C_{|a_1-a_2|}.$$
 Let $N_2=N'$. Note that under this assumption, we have 
 $$0<-cN_2(2n-2r+1)+a_2+\frac{1}{2}+r-n-1<3C_{|a_1-a_2|}=C.$$
 Therefore, for $N_3$ and $a_1$ large enough, $|a_1-a_2|$ fixed, we can always find a suitable $N_2$ such that \Cref{BF4} and \Cref{BF6} are convergent. We can see that for such choice of $N_2$, 
 $$-cN_2(2n+3-4i)+a_1+a+2i-n-1\geq-cN_2(2n-2r+1)+a+a_1+r-n+1>1$$
 for any $r+1\leq i\leq n$ and
 $$-cN_2(2n+1-4i)+a_2-\frac{1}{2}+2i-n\geq-cN_2(2n-2r+1)+a_2+\frac{1}{2}+r-n-1+1>1$$
 for any $r\leq i\leq n$. Therefore, by \Cref{c>1}, for $N_3$ large enough, \Cref{BF5} and \Cref{BF7} are also convergent.

 Now our conclusion is that for fixed $|\textnormal{Re}(s_1)-\textnormal{Re}(s_2)|$, and $\textnormal{Re}(s_1)$ large enough, $Z_{r,\eta}(f,\Phi,s_1,s_2)$ is absolutely convergent. Now for an arbitrary pair $(s_1,s_2)$, from the above conclusion, there exists $b>0$ such that $Z_{r,\eta}(f,\Phi,s_1+b,s_2+b)$ is absolutely convergent for any $f\in \mathcal{S}([G])$ and $\Phi\in\mathcal{S}(\mathbb{A}^n$). Define $f'\in \mathcal{S}([G])$ by
 $$f'(g)=f(g)|g|^{-b}.$$
 Then we can see that $Z_{r,\eta}(f,\Phi,s_1,s_2)=Z_{r,\eta}(f',\Phi,s_1+b,s_2+b)$, which is absolutely convergent. This finishes the proof of $Z_{r,\eta}(f,\Phi,s_1,s_2)$ for any $1\leq r\leq n$. 

 The absolute convergence of $Z_{r,\eta}'(f,\Phi,s_1,s_2)$ can be proved similarly for any $1\leq r\leq n$ and $s_1,s_2\in\mathbb{C}$. 
 The proof is complete.
\end{proof}

\subsection{Unfolding}\label{unfolding}

In this section, we will prove the unfolding equation, mainly the following proposition:
\begin{prop}\label[prop]{unfolding twisted BF}
    Let $\chi\in \mathfrak{X}_{\eta,\Delta^*}(G)$. Then for any $(f,\Phi)\in \mathcal{S}_{\chi}([G])\times \mathcal{S}(\mathbb{A}^n)$ and $s_1,s_2\in\mathbb{C}$, we have
    $$\mathcal{P}_{\eta}(f,\Phi,s_1,s_2)=Z_{n,\eta}'(f,\Phi,s_1,s_2)=Z_{n,\eta}(f,\Phi,s_1,s_2)=\cdots=Z_{1,\eta}'(f,\Phi,s_1,s_2)=Z_{1,\eta}(f,\Phi,s_1,s_2).$$
\end{prop}

The proof is divided into three parts:
\begin{itemize}
    \item[(1)] $\mathcal{P}_{\eta}(f,\Phi,s_1,s_2)=Z'_{n,\eta}(f,\Phi,s_1,s_2)$;
    \item[(2)] $Z_{r+1,\eta}(f,\Phi,s_1,s_2)=Z_{r,\eta}'(f,\Phi,s_1,s_2)$ for any $1\leq r\leq n-1$;
    \item[(3)] $Z_{r,\eta}'(f,\Phi,s_1,s_2)=Z_{r,\eta}(f,\Phi,s_1,s_2)$ for any $1\leq r\leq n$.
\end{itemize}

To start with, we need the following two lemmas:

\begin{lemma}\label[lemma]{lemma 1}
    Choose an arbitrary positive integer $r$. Let $G=G_{2r}$ and $H=G_r\times G_r$. We embed $H$ into $G$ by $\iota$. Let $\chi\in \mathfrak{X}_{(\eta,\Delta^*)}(G)$. Then for any $f\in \mathcal{S}_{\chi}([G])$, we have
    \begin{equation}
    \int_{[H]}f(\iota(x,y))\eta^{-1}(\textnormal{det}(x))|x|^{s_1}|y|^{s_2}\mathrm{d}x\mathrm{d}y=0
    \end{equation}
    for any $s_1,s_2\in \mathbb{C}$.
\end{lemma}
\begin{proof}
    Let $\chi=(M,\pi)$, where $M=\prod_{i=1}^lG_{r_i}$ and $\pi=\boxtimes_{i=1}^l\pi_i$. Let $P$ denote the standard parabolic subgroup of $G$ with Levi component $M$. For any $s_1,s_2\in \mathbb{C}$, we can see that the above integral is absolutely convergent and defines a functional from $\mathcal{S}_{\chi}([G])$ to $\mathbb{C}$. By \cite[Lemma 2.5.2]{lu2025periodsdetectingeisensteinseries}, we only need to consider the case where $f$ is a pseudo Eisenstein series. Therefore, we can write
    $$f(g)=\sum_{\gamma\in P(F)\backslash G(F)}\phi(\gamma g),$$
    where $\phi$ is a function in $\textnormal{Ind}_{P(\mathbb{A})}^{G(\mathbb{A})} \pi$ twisted by a Schwartz function on $A_P^{\infty}$. Denote the space of such functions by $\mathcal{S}_{\textnormal{Ind}\chi}([G]_P)$. %Let $\phi_i=\phi|_{G_{n_i}}\in \mathcal{S}_{\pi_i\otimes|\cdot|^a}([G_{n_i}])$ for some $a\in \mathbb{R}$ (here the twisting term comes from the modulus character $\delta_P$; we don't write it out explicitly because it can be somehow absorbed into the terms $|\cdot|^{s_1}$ and $|\cdot|^{s_2}$). 

    For such $f$, we can rewrite the integral as
    \begin{align}\label{BF8}
        \int_{[H]}f(\iota(x,y))\eta^{-1}(\textnormal{det}(x))|x|^{s_1}|y|^{s_2}\mathrm{d}x\mathrm{d}y & = \int_{[H]}\sum_{\gamma\in P(F)\backslash G(F)}\phi(\gamma\iota(x,y))\eta^{-1}(\textnormal{det}(x))|x|^{s_1}|y|^{s_2}\mathrm{d}x\mathrm{d}y \notag \\
        & = \int_{[H]}\sum_{w\in P(F)\backslash G(F)/H(F)}\sum_{\delta \in H_w(F)\backslash H(F)}\phi(w\delta \iota(x,y))\eta^{-1}(\textnormal{det}(x))|x|^{s_1}|y|^{s_2}\mathrm{d}x\mathrm{d}y \notag\\
        & = \sum_{w\in P(F)\backslash G(F)/H(F)}\int_{[H]}\sum_{\delta \in H_w(F)\backslash H(F)}\phi(w\delta \iota(x,y))\eta^{-1}(\textnormal{det}(x))|x|^{s_1}|y|^{s_2}\mathrm{d}x\mathrm{d}y.
    \end{align}
    where $H_w=w^{-1}Pw\cap H$. Here, we can regard $\eta^{-1}(\textnormal{det}(x))$, $|x|^{s_1}$ and $|y|^{s_2}$ as three Hecke characters on $[H]$. Denote them by $\xi_1$, $\xi_2$ and $\xi_3$ respectively. Thus, we can rewrite \Cref{BF8} as
    \begin{align}\label{BF9}
        \sum_{w\in P(F)\backslash G(F)/H(F)} & \int_{[H]}\sum_{\delta \in H_w(F)\backslash H(F)}\phi(w\delta \iota(x,y))\eta^{-1}(\textnormal{det}(x))|x|^{s_1}|y|^{s_2}\mathrm{d}x\mathrm{d}y \notag \\
        & =\sum_{w\in P(F)\backslash G(F)/H(F)}\int_{[H]}\sum_{\delta \in H_w(F)\backslash H(F)}\phi(w\delta h)\xi_1(h)\xi_2(h)\xi_3(h)\mathrm{d}h \notag \\
        & =\sum_{w\in P(F)\backslash G(F)/H(F)}\int_{H_w(F)\backslash H(\mathbb{A})}\phi(wh)\xi_1(h)\xi_2(h)\xi_3(h)\mathrm{d}h.
    \end{align}
    We have the following claim:
    \begin{claim}\label{claim 1}
        For any $\phi\in \mathcal{S}_{\textnormal{Ind}\chi}([G]_P)$ and $w\in P(F)\backslash G(F)/H(F)$, We have
        \begin{equation}
            \int_{[H_w]}\phi(wh)\xi_1(h)\xi_2(h)\xi_3(h)\mathrm{d}h=0.
        \end{equation}
    \end{claim}
    Note that for any $\phi \in \mathcal{S}_{\textnormal{Ind}\chi}([G]_P)$ and $g\in G$, the function $R(g)\phi$ is still in $\mathcal{S}_{\textnormal{Ind}\chi}([G]_P)$. Therefore, the vanishing of \Cref{BF9} follows directly from the above claim. 
\end{proof}
\begin{proof}[Proof of \Cref{claim 1}]
    For a double coset in $P(F)\backslash G(F)/H(F)$, we choose a representative in $G(F)$, and still denote it by $w$. Replacing $h$ by $w^{-1}hw$, we only need to prove that 
    $$\int_{[P\cap wHw^{-1}]}\phi(hw)\xi_1(h)\xi_2(h)\xi_3(h)\mathrm{d}h=0.$$
    Here we still use the same notations $\xi_1,\xi_2$ and $\xi_3$ to denote the new characters on $P\cap wHw^{-1}$ after conjugation. Note that $\phi(hw)=(R(w)\phi)(h)$, thus it suffices to prove that 
    \begin{equation}\label{BF10}
        \int_{[P\cap wHw^{-1}]}\phi(h)\xi_1(h)\xi_2(h)\xi_3(h)\mathrm{d}h=0
    \end{equation}
    for any $\phi\in \mathcal{S}_{\textnormal{Ind}\chi}([G]_P)$. By \cite[Proposition 3.2]{Matringe+2015+119+170}, $w$ corresponds to a set
    $$s_w=\{r_{i,j},1\leq i<j\leq l,(r_{k,k}^+,r_{k,k}^-),1\leq k\leq l\}$$
    such that if we set $r_{i,j}=r_{j,i}$, and $r_{k,k}=r_{k,k}^++r_{k,k}^-$, then $r_i=\sum_{j=1}^lr_{i,j}$ and $\sum_{k=1}^lr_{k,k}^+=\sum_{k=1}^lr_{k,k}^-$. Note that the set $s_w$ is a partition of $(r_1,\dots,r_l)$. By Proposition 3.3 and 3.4 in loc. cit., we can write $P\cap wHw^{-1}=P_{s_w}\cap wHw^{-1}$ as the semi-direct product of $M_{s_w}\cap wHw^{-1}$ and $N_{s_w}\cap wHw^{-1}$, where $P_{s_w}$ is the parabolic subgroup of $G$ corresponding to the set $s_w$, with the Levi component $M_{s_w}\subset M$ and unipotent radical $N_{s_w}$.

    Let $N'_{s_w}$ denote the unipotent radical of the parabolic subgroup of $M$ corresponding to $s_w$. We first assume that the $N_{s_w}'\neq 1$. There is a natural projection from $N_{s_w}$ to $N'_{s_w}$. Then by Proposition 3.5 in loc. cit., the natural projection restricted on $N_{s_w}\cap wHw^{-1}$ is still surjective. Therefore, the integral
    $$\int_{[N_{s_w}\cap wHw^{-1}]}\phi(u)\mathrm{d}u$$
    vanishes because $\phi|_M$ is a cuspidal form (the twisted Schwartz function on $A_P^{\infty}$ will not affect $\phi$ being a cuspidal form). As a result, \Cref{BF10} holds, and the proof is complete.

    Therefore, we only need to consider the case when $N_{s_w}'$ is trivial. In other words, for each $i$, there must exist $j$ such that $r_i=r_{i,j}$. Under this assumption, we can see that $N_{s_w}=N$, and the vanishing of \Cref{BF10} is reduced to the vanishing of the corresponding integral on $M_{s_w}\cap wHw^{-1}=M\cap wHw^{-1}$ (because $\phi$ is trivial on $N(\mathbb{A})=N_{s_w}(\mathbb{A})$). By Proposition 3.4 in loc. cit., elements in $M_{s_w}\cap wHw^{-1}$ can be expressed as
    $$m=\textnormal{diag}(m_{1,1}^+,m_{1,1}^-,m_{1,2},\cdots,m_{i,j},\cdots,m_{j,i},\cdots,m_{l,l}^+,m_{l,l}^-),$$
    with $m_{k,k}^+\in G_{r_{k,k}^+}$, $m_{k,k}^-\in G_{r_{k,k}^-}$ and $m_{i,j}=m_{j,i}\in G_{r_{i,j}}$.

    We need to firstly analyze how $\xi_1$, $\xi_2$ and $\xi_3$ act on such $m$. Let $\epsilon=\textnormal{diag}(1,-1,\dots,1,-1)$ and $V=\mathbb{A}^{2r}$. By Proposition 3.2 in loc. cit., $\epsilon_{s_w}=w\epsilon w^{-1}$ is the block diagonal matrix with $(i,i)$-block $\textnormal{diag}(I_{r_{i,i}^+},-I_{r_{i,i}^-})$, and $((i<j),(j>i))$-th diagonal block $\begin{pmatrix}
          &  I_{r_i,j}  \\
         I_{r_{i,j}} &
    \end{pmatrix}$. Let us be more concrete: decompose $V$ as
    $$V=\bigoplus_{i=1}^{l}(\mathbb{A}^{r_{i,i}^+}\oplus \mathbb{A}^{r_{i,i}^-})\oplus\bigoplus_{1\leq i<j\leq l}(\mathbb{A}^{r_{i,j}}\oplus \mathbb{A}^{r_{j,i}}).$$
    The action of $\epsilon_{s_w}$ on $V$ is as follows:
    \begin{itemize}
        \item[(1)] For any $1\leq i\leq l$ and any vector $v\in \mathbb{A}^{r_{i,i}^+}$, $\epsilon_{s_w}(v)=v$;
        \item[(2)] For any For any $1\leq i\leq l$ and any vector $v\in \mathbb{A}^{r_{i,i}^-}$, $\epsilon_{s_w}(v)=-v$;
        \item[(3)] For any $1\leq i<j\leq l$, and any vector $(v_1,v_2)\in \mathbb{A}^{r_{i,j}}\oplus \mathbb{A}^{r_{j,i}}$, $\epsilon_{s_w}(v_1,v_2)=(v_2,v_1)$.
    \end{itemize}

    Note that the only eigenvalues of $\epsilon_{s_w}$ are $1$ and $-1$. Denote the eigenspace corresponding to the eigenvalue $1$ by $V^+_{s_w}\subset V$ and the eigenspace corresponding to the eigenvalue $-1$ by $V^-_{s_w}\subset V$. Then the determinant comes from the first $G_r$ factor actually is equal to the $\textnormal{det}(m|_{V^+_{s_w}})$ and the determinant comes from the second $G_r$ factor actually is equal to the $\textnormal{det}(m|_{V^-_{s_w}})$ (since we only care about determinant, choice of basis does not matter). Therefore, we have the following three equations:
    $$\xi_1(m)=\prod_{i=1}^l\eta^{-1}(\textnormal{det}(m_{i,i}^+))\prod_{1\leq i<j\leq l}\eta^{-1}(\textnormal{det}(m_{i,j}));$$
    $$\xi_2(m)=\prod_{i=1}^l|m_{i,i}^+|^{s_1}\prod_{1\leq i<j\leq l}|m_{i,j}|^{s_1};$$
    $$\xi_3(m)=\prod_{i=1}^l|m_{i,i}^-|^{s_2}\prod_{1\leq i<j\leq l}|m_{j,i}|^{s_2}.$$

    To prove \Cref{BF10}, it suffices to prove that the equation holds when $\phi$ is a pure tensor product $\phi_1\otimes\dots\otimes \phi_l$, where each function $\phi_i$ is a cusp form in $\pi_i$ twisted by a Schwartz function on $A_{G_{r_i}}^{\infty}$ (the actual $\phi_i$ should include a twist by $|\cdot|^a$ provided by the modulus character $\delta_P$, but we ignore it because it can be absorbed into the $s_1$ and $s_2$ term). For such $\phi$, let us rewrite the integral:
    \begin{align*}
         \int_{[M_{s_w}\cap wHw^{-1}]}\phi(h)\xi_1(h)\xi_2(h)\xi_3(h)\mathrm{d}h 
        & =\prod_{i=1}^l\int_{[G_{r_{i,i}^+}]\times [G_{r_{i,i}^-}]}\phi_i\begin{pmatrix}
            m_{i,i}^+ & \\ & m_{i,i}^-
        \end{pmatrix}\eta^{-1}(\textnormal{det}(m_{i,i}^+))|m_{i,i}^+|^{s_1}|m_{i,i}^-|^{s_2}\mathrm{d}m_{i,i}^+\mathrm{d}m_{i,i}^- \\ &\;\prod_{1\leq i<j\leq l}\int_{[G_{r_{i,j}}]}\phi_i(m_{i,j})\phi_j(m_{j,i})\eta^{-1}(\textnormal{det}(m_{i,j}))|m_{i,j}|^{s_1+s_2}\mathrm{d}m_{i,j}.
    \end{align*}

    Recall our previous assumption: for each $i$, there must exist $j$ such that $r_i=r_{i,j}$. If there exists $1\leq i<j\leq l$ such $r_{i,j}>0$, we can get $r_i=r_{i,j}=r_{j,i}=r_j$. Then the vanishing of the integral
    $$\int_{[G_{r_i}]^1} \phi_i(m)\eta^{-1}(\textnormal{det}(m))|m|^{s_1}\phi_j(m)|m|^{s_2}\mathrm{d}m$$
    can be followed from the first condition of the definition of $(\eta,\Delta^*)$-regularity. Thus, \Cref{BF10} holds, and the proof is complete.

    Otherwise, we can assume that $r_{i,j}=0$ for any $i\neq j$ and $r_{i,i}=r_i$ for all $1\leq i\leq l$. If there exists $i$ such that $r_{i,i}^+\neq r_{i,i}^-$ and $r_i\geq 2$, then the integral
    $$\int_{[G_{r_{i,i}^+}]\times [G_{r_{i,i}^-}]}\phi_i(m_1,m_2)\eta^{-1}(\textnormal{det}(m_1))|m_1|^{s_1}|m_2|^{s_2}\mathrm{d}m_1\mathrm{d}m_2$$
    vanishes by \cite[Lemma 7.1]{matringe2025intertwining}, which implies that \Cref{BF10} holds. If there exists $i$ such that $r_{i,i}^+= r_{i,i}^-$, then the integral
    $$\int_{[G_{r_{i,i}^+}]\times [G_{r_{i,i}^-}]}\phi_i(m_1,m_2)\eta^{-1}(\textnormal{det}(m_1))|m_1|^{s_1}|m_2|^{s_2}\mathrm{d}m_1\mathrm{d}m_2$$
    vanishes by the second condition of the definition of $(\eta,\Delta^*)$-regularity and \cite[Proposition 3.1]{xue2025twisted}, which implies that \Cref{BF10} holds. The last case is that all blocks are $G_1$ block, with half assigned to the negative sign, and half assigned to the positive sign. Then the integral on each  positive block will be
    $$\int_{[G_1]}\phi_i(g)\eta^{-1}(\textnormal{det}(g))|g|^{s_1}\mathrm{d}g,$$
    which vanishes by the third condition of the definition of $(\eta,\Delta^*)$-regularity, and therefore \Cref{BF10} holds. The proof is complete.
    
\end{proof}

\begin{lemma}\label[lemma]{lemma 2}
    Choose an arbitrary positive integer $r$. Let $G=G_{2r+1}$ and $H=G_{r+1}\times G_r$. We embed $H$ into $G$ by $\iota$. Let $\chi\in \mathfrak{X}_{(\eta,\Delta^*)}(G)$. Then for any $f\in \mathcal{S}_{\chi}([G])$, we have
    \begin{equation}
    \int_{[H]}f(\iota(x,y))\eta^{-1}(\textnormal{det}(x))|x|^{s_1}|y|^{s_2}\mathrm{d}x\mathrm{d}y=0
    \end{equation}
    for any $s_1,s_2\in \mathbb{C}$.
\end{lemma}
\begin{proof}
    The proof is very similar to the proof of \Cref{lemma 1}. By \cite[Proposition 3.2]{Matringe+2015+119+170}, a representative $w$ of a double coset in $P(F)\backslash G(F)/H(F)$ corresponds to a set
    $$s_w=\{r_{i,j},1\leq i<j\leq l,(r_{k,k}^+,r_{k,k}^-),1\leq k\leq l\}$$
    such that if we set $r_{i,j}=r_{j,i}$, and $r_{k,k}=r_{k,k}^++r_{k,k}^-$, then $r_i=\sum_{j=1}^lr_{i,j}$ and $\sum_{k=1}^lr_{k,k}^+=\sum_{k=1}^lr_{k,k}^-+1$. We can also reduce to the case when $P\cap wHw^{-1}=M_{s_w}\cap wHw^{-1}$. Then we can follow exactly the same discussion as in the proof of \Cref{lemma 1} to finish the proof. We leave the details to the reader.

\end{proof}

Now we turn to the proof of the unfolding identity:

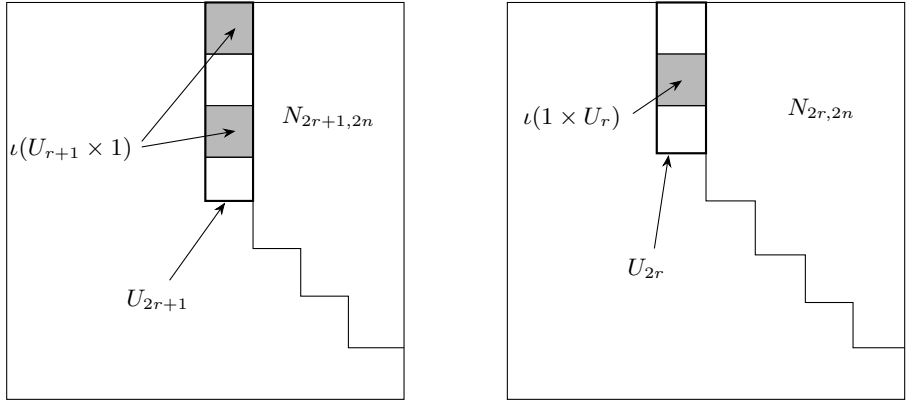
\begin{figure}[ht] 
    \centering 
    \begin{tikzpicture}[
        scale=1.05,
        >=Stealth,
        every node/.style={font=\small},
        arrow/.style={->,thin} ]
        \begin{scope}
            \draw (0,0) rectangle (5,5);
            \draw (3.1,5) -- (3.1,1.9)
                  -- (3.7,1.9)
                  -- (3.7,1.3)
                  -- (4.3,1.3)
                  -- (4.3,0.65)
                  -- (5,0.65);
            \draw[fill=gray!55] (2.5,4.35) rectangle (3.1,5);
            \draw (2.5,3.70) rectangle (3.1,4.35);
            \draw[fill=gray!55] (2.5,3.05) rectangle (3.1,3.70);
            \draw (2.5,2.50) rectangle (3.1,3.05);
            \draw[thick] (2.5,2.50) rectangle (3.1,5);
            \node at (4.05,3.55) {$N_{2r+1,2n}$};
            \node[anchor=east] at (1.70,3.15) {$\iota(U_{r+1}\times 1)$};
            \draw[arrow] (1.70,3.18) -- (2.88,3.38);
            \draw[arrow] (1.70,3.25) -- (2.88,4.68);
            \node at (1.90,1.25) {$U_{2r+1}$};
            \draw[arrow] (2.05,1.50) -- (2.75,2.48);
        \end{scope}
        \begin{scope}[xshift=6.3cm]
            \draw (0,0) rectangle (5,5);
            \draw (1.88,4.35) rectangle (2.50,5);
            \draw[fill=gray!55] (1.88,3.70) rectangle (2.50,4.35); \draw (1.88,3.10) rectangle (2.50,3.70);
            \draw[thick] (1.88,3.10) rectangle (2.50,5);
            \draw (2.50,3.10) -- (2.50,2.50)
                  -- (3.12,2.50)
                  -- (3.12,1.82)
                  -- (3.75,1.82)
                  -- (3.75,1.22)
                  -- (4.35,1.22)
                  -- (4.35,0.65)
                  -- (5,0.65);
            \node at (3.95,3.65) {$N_{2r,2n}$};
            \node[anchor=east] at (1.55,3.55) {$\iota(1\times U_r)$};
            \draw[arrow] (1.58,3.60) -- (2.20,4.02);
            \node at (1.75,1.65) {$U_{2r}$};
            \draw[arrow] (1.85,1.90) -- (2.03,3.08);
        \end{scope}
    \end{tikzpicture}
    \caption{Illustrations of the proof of (2) and (3) for $n=4$, $r=2$.}
    \label{fig:staircase-illustration}
\end{figure}

\begin{proof}[Proof of \textnormal{(2)}]
    Let $U_{2r+1}(\mathbb{A})=\{I_{2n}+\sum_{i=1}^{2r}b_iE_{i,2r+1}|b_i\in \mathbb{A}\}\subset G(\mathbb{A})$ and $U_{r+1}(\mathbb{A})=\{I_{n}+\sum_{i=1}^{r}b_iE_{i,r+1}|b_i\in \mathbb{A}\}\subset G_n(\mathbb{A})$ (one can check Figure 1 for a concrete example). Using the Fourier analysis on the compact abelian group $U_{2r+1}(\mathbb{A})/\iota(U_{r+1}\times \{1\})(\mathbb{A})U_{2r+1}(F)$, for any $g\in G(\mathbb{A})$, we can write
    \begin{align*}
        \int_{[U_{r+1}]}f_{N_{2r+1,2n,\psi}}(\iota(u,1)g)\mathrm{du} & = (f_{N_{2r+1,2n,\psi}})_{U_{2r+1}}(g)+\sum_{\gamma\in \mathcal{P}_r(F)\backslash G_r(F)}(f_{N_{2r+1,2n,\psi}})_{U_{2r+1},\psi}(\iota(1,\gamma)g) \\
        & = (f_{N_{2r+1,2n,\psi}})_{U_{2r+1}}(g)+\sum_{\gamma\in \mathcal{P}_r(F)\backslash G_r(F)}f_{N_{2r,2n,\psi}}(\iota(1,\gamma)g),
    \end{align*}
    where $(f_{N_{2r+1,2n,\psi}})_{U_{2r+1}}$ stands for the constant term on the subgroup $U_{2r+1}$ and $(f_{N_{2r+1,2n,\psi}})_{U_{2r+1},\psi}$ stands for the Fourier coefficient of $\psi|_{U_{2r+1}}$. Note that $[U_{r+1}]=U_{r+1}(F)\backslash U_{r+1}(\mathbb{A})$, $\mathcal{P}_{r+1}(F)=G_r(F)U_{r+1}(F)$ and $N_{r,n}(\mathbb{A})=N_{r+1,n}(\mathbb{A})U_{r+1}(\mathbb{A})$. Together with \Cref{absolute convergence 3}, we can write
    \begin{align*}
        Z_{r+1,\eta}(f,\Phi,s_1,s_2) & =\int_{\mathcal{P}_{r+1}(F)N_{r+1,n}(\mathbb{A})\backslash G_n(\mathbb{A})}\int_{G_{r}(F)N_{r,n}(\mathbb{A})\backslash G_n(\mathbb{A})}f_{N_{2r+1,2n,\psi}}(\iota(x,y))\eta^{-1}(\textnormal{det}(x))\Phi(e_ny)|x|^{s_1}|y|^{s_2+\frac{1}{2}}\mathrm{d}x\mathrm{d}y \\
        & =\int_{G_r(F)N_{r,n}(\mathbb{A})\backslash G_n(\mathbb{A})}\int_{\mathcal{P}_{r}(F)N_{r,n}(\mathbb{A})\backslash G_n(\mathbb{A})}f_{N_{2r,2n,\psi}}(\iota(x,y))\eta^{-1}(\textnormal{det}(x))\Phi(e_ny)|x|^{s_1}|y|^{s_2+\frac{1}{2}}\mathrm{d}x\mathrm{d}y \\ & \;\;\;\;+F_{r,f}(s_1,s_2)\\
        & = Z_{r,\eta}'(f,\Phi,s_1,s_2)+F_{r,f}(s_1,s_2),
    \end{align*}
    where
    $$F_{r,f}(s_1,s_2)=\int_{G_r(F)N_{r,n}(\mathbb{A})\backslash G_n(\mathbb{A})}\int_{G_r(F)N_{r,n}(\mathbb{A})\backslash G_n(\mathbb{A})}(f_{N_{2r+1,2n,\psi}})_{U_{2r+1}}(\iota(x,y))\eta^{-1}(\textnormal{det}(x))\Phi(e_ny)|x|^{s_1}|y|^{s_2+\frac{1}{2}}\mathrm{d}x\mathrm{d}y,$$
    and the absolute convergence of the above expression of integral can be proved similarly as the proof for \Cref{absolute convergence 3}. By the next lemma, we can see that $F_{r,f}(s_1,s_2)=0$ for any $s_1,s_2\in \mathbb{C}$ and $f\in \mathcal{S}_{\chi}([G])$ when $\chi$ is $(\eta,\Delta^*)$-regular. The proof is complete.

\end{proof}
\begin{lemma}\label[lemma]{F_r=0}
    For any $s_1,s_2\in \mathbb{C}$ and $f\in \mathcal{S}_{\chi}([G])$, $F_{r,f}(s_1,s_2)$ vanishes when $\chi$ is $(\eta,\Delta^*)$-regular.
\end{lemma}
\begin{proof}
     For any $1\leq k\leq 2n$, let $P_{(k,2n-k)}$ denote the standard parabolic subgroup of $G_{2n}$ with Levi component $G_{k}\times G_{2n-k}$. Thus, we can rewrite the function $(f_{N_{2r+1,2n,\psi}})_{U_{2r+1}}$ as
    $$(f_{N_{2r+1,2n,\psi}})_{U_{2r+1}}\begin{pmatrix} a &  \\  & b \end{pmatrix}=\int_{[N_{2n-2r}]}f_{P_{(2r,2n-2r)}}\begin{pmatrix} a &  \\  & ub \end{pmatrix} \psi^{-1}(u)\mathrm{d}u.$$

    By the Iwasawa decomposition $G_n(\mathbb{A})=P_{r,n}(\mathbb{A})K$, we can rewrite $F_{r,f}(s_1,s_2)$ as
    \begin{align*}
    F_{r,f}(s_1,s_2)  = &\int_{K\times K}\int_{[G_r\times G_r]}\int_{(\mathbb{A}^{\times})^{2n-2r}}\int_{[N_{2n-2r}]} \\& (R(\iota(k_1,k_2))f)_{P_{(2r,2n-2r)}}\begin{pmatrix} \iota(m_1,m_2) &  \\  & u\cdot \textnormal{diag}(a_{r+1},b_{r+1},\dots,a_n,b_n) \end{pmatrix} \psi^{-1}(u) \\
    & \eta^{-1}(\textnormal{det}(m_1))\prod_{i=r+1}^n\eta^{-1}(a_i)(R(k_2)\Phi)(b_ne_n)|m_1|^{s_1+r-n}|m_2|^{s_2+\frac{1}{2}+r-n} \\ & \prod_{i=r+1}^n|a_i|^{s_1+2i-n-1}|b_i|^{s_2-\frac{1}{2}+2i-n}\mathrm{d}u\prod_{i=r+1}^n\mathrm{d}a_i\mathrm{d}b_i\;\mathrm{d}m_1\mathrm{d}m_2\mathrm{d}k_1\mathrm{d}k_2.
    \end{align*}
    For any $f\in \mathcal{S}_{\chi}([G])$ and $\Phi\in \mathcal{S}(\mathbb{A}^n)$, let
    \begin{align*}
        I_{r,f}(s_1,s_2)  = &\int_{[G_r\times G_r]}\int_{(\mathbb{A}^{\times})^{2n-2r}}\int_{[N_{2n-2r}]} \\& f_{P_{(2r,2n-2r)}} \begin{pmatrix} \iota(m_1,m_2) &  \\  & u\cdot \textnormal{diag}(a_{r+1},b_{r+1},\dots,a_n,b_n) \end{pmatrix} \psi^{-1}(u)   \eta^{-1}(\textnormal{det}(m_1))\prod_{i=r+1}^n\eta^{-1}(a_i)\\ &\Phi(b_ne_n) |m_1|^{s_1+r-n}|m_2|^{s_2+\frac{1}{2}+r-n} \prod_{i=r+1}^n|a_i|^{s_1+2i-n-1}|b_i|^{s_2-\frac{1}{2}+2i-n}\mathrm{d}u\prod_{i=r+1}^n\mathrm{d}a_i\mathrm{d}b_i\;\mathrm{d}m_1\mathrm{d}m_2.
    \end{align*}
    By the same proof as for \Cref{absolute convergence 3}, $F_{r,f}(s_1,s_2)$ and $I_{r,f}(s_1,s_2)$ are absolutely convergent for any $s_1,s_2\in \mathbb{C}$. Moreover, the two spaces $\mathcal{S}_{\chi}([G])$ and $\mathcal{S}(\mathbb{A}^n)$ are preserved by the right regular action of $G(\mathbb{A})$ and $G_n(\mathbb{A})$, respectively. Therefore, to prove that $F_{r,f}(s_1,s_2)$ vanishes, it suffices to prove that $I_{r,f}(s_1,s_2)$ vanishes.

    We can always find a series of compactly supported smooth functions $\{\kappa_j\}_{j\in \mathbb{N}}$ on $\mathfrak{a}_{P_{(2r,2n-2r)}}$ satisfying the following conditions:
    \begin{itemize}
        \item[(1)] For any $j\in \mathbb{N}$ and $\alpha\in \mathfrak{a}_{P_{(2r,2n-2r)}}$, $|\kappa_j(\alpha)|\leq 1$;
        \item[(2)] For any $j\in\mathbb{N}$, $\kappa_j(0)=1$;
        \item[(3)] For any $\alpha\in \mathfrak{a}_{P_{(2r,2n-2r)}}$, $\lim_{j\to \infty}\kappa_j(\alpha)=1$.
    \end{itemize}

    Define
    \begin{align*}
        I^j_{r,f}(s_1,s_2)=&\int_{[G_r\times G_r]}\int_{(\mathbb{A}^{\times})^{2n-2r}}\int_{[N_{2n-2r}]} (\kappa_j\circ H_{P_{(2r,2n-2r)}})\begin{pmatrix} \iota(m_1,m_2) &  \\  & u\cdot \textnormal{diag}(a_{r+1},b_{r+1},\dots,a_n,b_n) \end{pmatrix}\\ & f_{P_{(2r,2n-2r)}} \begin{pmatrix} \iota(m_1,m_2) &  \\  & u\cdot \textnormal{diag}(a_{r+1},b_{r+1},\dots,a_n,b_n) \end{pmatrix} \psi^{-1}(u)   \eta^{-1}(\textnormal{det}(m_1))\prod_{i=r+1}^n\eta^{-1}(a_i)\\ &\Phi(b_ne_n) |m_1|^{s_1+r-n}|m_2|^{s_2+\frac{1}{2}+r-n} \prod_{i=r+1}^n|a_i|^{s_1+2i-n-1}|b_i|^{s_2-\frac{1}{2}+2i-n}\mathrm{d}u\prod_{i=r+1}^n\mathrm{d}a_i\mathrm{d}b_i\;\mathrm{d}m_1\mathrm{d}m_2.
    \end{align*}
    Note that $I^j_{r,f}(s_1,s_2)$ is absolutely convergent because $|\kappa_j(\alpha)|\leq 1$ for any $\alpha\in \mathfrak{a}_{P_{(2r,2n-2r)}}$.

    By \cite[2.3.3]{lu2025periodsdetectingeisensteinseries}, $(\kappa_j\circ H_{P_{(2r,2n-2r)}})\cdot f_{P_{(2r,2n-2r)}}\in \mathcal{S}_{\chi}([G]_{P_{(2r,2n-2r)}})$. Let $\chi^{M_{P_{(2r,2n-2r)}}}$ denote the inverse image of $\chi$ in $\mathfrak{X}(M_{P_{(2r,2n-2r)}})$. Note that if $\chi\in \mathfrak{X}_{(\eta,\Delta^*)}(G)$, then $\chi^{M_{P_{(2r,2n-2r)}}}\subset \mathfrak{X}_{(\eta,\Delta^*)}(M_{P_{(2r,2n-2r)}})$. Moreover, the projection from $\mathfrak{X}(M_{P_{(2r,2n-2r)}})$ to $\mathfrak{X}(G_{2r})$ preserves $(\eta,\Delta^*)$-regularity. Combining with Lemmas 2.5.4 and 2.5.5 in loc. cit., the function $(\kappa_j\circ H_{P_{(2r,2n-2r)}})\cdot f_{P_{(2r,2n-2r)}}|_{[G_{2r}]}$ lies in $\mathcal{S}_{(\eta,\Delta^*)}([G_{2r}])$. By \Cref{lemma 1}, we can see that $I^j_{r,f}(s_1,s_2)=0$.

    The construction of the function series $\{\kappa_j\}_{j\in \mathbb{N}}$ guarantees that 
    $$\lim_{j\to \infty}(\kappa_j\circ H_{P_{(2r,2n-2r)}})\cdot f_{P_{(2r,2n-2r)}}\begin{pmatrix} a &  \\  & b \end{pmatrix}=f_{P_{(2r,2n-2r)}}\begin{pmatrix} a &  \\  & b \end{pmatrix}$$
    for any $\begin{pmatrix} a &  \\  & b \end{pmatrix}\in M_{P_{(2r,2n-2r)}}$. Then, the dominated convergence theorem gives
    $$I_{r,f}(s_1,s_2)=\lim_{j\to \infty}I^j_{r,f}(s_1,s_2)=0.$$
    The proof is complete.

\end{proof}

\begin{proof}[Proof of \textnormal{(3)}]
    The proof is similar to the proof of part (2). Let $U_{2r}(\mathbb{A})=\{I_{2n}+\sum_{i=1}^{2r-1}b_iE_{i,2r}|b_i\in \mathbb{A}\}\subset G(\mathbb{A})$ and $U_{r}(\mathbb{A})=\{I_{n}+\sum_{i=1}^{r-1}b_iE_{i,r}|b_i\in \mathbb{A}\}\subset G_n(\mathbb{A})$ (one can check Figure 1 for a concrete example). Using the Fourier analysis on $U_{2r}(\mathbb{A})/\iota(1\times U_r)(\mathbb{A})U_{2r}(F)$, we have
    $$\int_{[U_r]}f_{N_{2r,2n,\psi}}(\iota(1,u)g)\mathrm{d}u=(f_{N_{2r,2n,\psi}})_{U_{2r}}(g)+\sum_{\gamma\in \mathcal{P}_r(F)\backslash G_r(F)}f_{N_{2r-1,2n,\psi}}(\iota(\gamma,1)g).$$

    By \Cref{absolute convergence 3}, we can write
    $$Z'_{r,\eta}(f,\Phi,s_1,s_2)=Z_{r,\eta}(f,\Phi,s_1,s_2)+F'_{r,f}(s_1,s_2),$$
    where
    $$F'_{r,f}(s_1,s_2)=\int_{G_r(F)N_{r,n}(\mathbb{A})\backslash G_n(\mathbb{A})}\int_{G_{r-1}(F)N_{r-1,n}(\mathbb{A})\backslash G_n(\mathbb{A})}(f_{N_{2r,2n,\psi}})_{U_{2r}}(\iota(x,y))\eta^{-1}(\textnormal{det}(x))\Phi(e_ny)|x|^{s_1}|y|^{s_2+\frac{1}{2}}\mathrm{d}x\mathrm{d}y$$
    and the absolute convergence of the above expression of integral can be proved similarly as the proof for \Cref{absolute convergence 3}. Then, by \Cref{lemma 2} and the proof of \Cref{F_r=0} , we can show that $F'_{r,f}(s_1,s_2)=0$ for any $s_1,s_2\in \mathbb{C}$. The proof is complete

\end{proof}

\begin{proof}[Proof of \textnormal{(1)}]
    By \Cref{absolute convergence 3}, we can write
    \begin{align*}
        \mathcal{P}_{\eta}(f,\Phi,s_1,s_2) & =\int_{[G_n]\times [G_n]}f(\iota(x,y))\eta^{-1}(\textnormal{det}(x))\Theta(y,\Phi)|x|^{s_1}|y|^{s_2}\mathrm{d}x\mathrm{d}y \\
        & =\int_{[G_n]}\int_{[G_n]}f(\iota(x,y))\eta^{-1}(\textnormal{det}(x))\sum_{v\in F^n}\Phi(vy)|x|^{s_1}|y|^{s_2+\frac{1}{2}}\mathrm{d}x\mathrm{d}y \\
        & = \int_{[G_n]}\int_{[G_n]}f(\iota(x,y))\eta^{-1}(\textnormal{det}(x))\sum_{\gamma\in \mathcal{P}_n(F)\backslash G_n(F)}\Phi(e_n\gamma y)|x|^{s_1}|y|^{s_2+\frac{1}{2}}\mathrm{d}x\mathrm{d}y+F_{n,f}(s_1,s_2) \\ 
        &=Z_{n,\eta}'(f,\Phi,s_1,s_2)+F_{n,f}(s_1,s_2),
    \end{align*}
    where
    $$F_{n,f}(s_1,s_2)=\Phi(0)\int_{[G_n]}\int_{[G_n]}f(\iota(x,y))\eta^{-1}(\textnormal{det}(x))|x|^{s_1}|y|^{s_2+\frac{1}{2}}\mathrm{d}x\mathrm{d}y.$$
    The vanishing of $F_{n,f}(s_1,s_2)$ directly follows from \Cref{lemma 1}. The proof is complete.
\end{proof}

\begin{proof}[Proof of \Cref{main result}]
    By the unfolding identity, we can get
    $$\mathcal{P}_{\eta}(f,\Phi,s_1,s_2)=Z^{\textnormal{BF}}_{\eta}(f,\Phi,s_1,s_2)$$
    for any $f\in \mathcal{S}_{\chi}([G])$, $\Phi\in\mathcal{S}(\mathbb{A}^n)$ and $s_1,s_2\in\mathbb{C}$. Write $\tilde{f}(g)=f({^tg^{-1}})$. If $f\in \mathcal{S}_{\chi}([G])$, then $\tilde{f}\in \mathcal{S}_{\chi^{\vee}}([G])$. Now by the change of variable $(x,y)\mapsto({^tx^{-1}},{^ty^{-1}})$ and \Cref{theta series}, we have
    $$Z_{\eta}^{\textnormal{BF}}(f,\Phi,s_1,s_2)=Z_{\eta^{-1}}^{\textnormal{BF}}(\tilde{f},\hat{\Phi},-s_1,-s_2).$$
    Note that for a $(\eta,\Delta^*)$-regular cuspidal data $\chi$, $\chi^{\vee}$ is a $(\eta^{-1},\Delta^*)$-regular cuspidal data. Therefore, the integral $Z_{\eta^{-1}}^{\textnormal{BF}}(\tilde{f},\hat{\Phi},-s_1,-s_2)$ is well defined and satisfies all the previous discussion. Fix a Schwartz function $\Phi\in \mathcal{S}(\mathbb{A}^n)$ and choose a positive integer $N$. Consider the space $W=L^2_{-N,\chi}([G])^{\infty}$. If we can prove that for any $f\in W$, the functional $Z_{\eta}^{\textnormal{BF}}(f,\Phi,s_1,s_2)$ extends to an entire function on $\mathbb{C}^2$, and let $N$ vary, then the proof is complete. So it suffices to find a holomorphic extension of $Z_{\eta}^{\textnormal{BF}}(f,\Phi,s_1,s_2)$ when $f\in W$. 
    
    Now fix a positive integer $N$. There exists a positive integer $N'$ such that for any $f\in L^2_{-N,\chi}([G])^{\infty}$, $\tilde{f}\in L^2_{-N',\chi^{\vee}}([G])^{\infty}$. Also there exists another positive integer $M$ such that $L^2_{-N,\chi}([G])^{\infty}\subset \mathcal{T}_{M,\chi}([G])$ and $L^2_{-N',\chi^{\vee}}([G])^{\infty}\subset \mathcal{T}_{M,\chi^{\vee}}([G])$. Now by \Cref{absolute convergence 1}, there exists a real number $a_N$ such that for any $f\in W$, $Z_{\eta}^{\textnormal{BF}}(f,\Phi,s_1,s_2)$ and $Z_{\eta^{-1}}^{\textnormal{BF}}(\tilde{f},\hat{\Phi},s_1,s_2)$ are both absolutely convergent for any $(s_1,s_2)\in \mathcal{H}_{>a_N}\times \mathcal{H}_{>a_N}$. Let $u=s_1+s_2$ and $v=s_1-s_2$. Define the following two functionals
    $$A_+(u,v,f)=Z_{\eta}^{\textnormal{BF}}(f,\Phi,\frac{u+v}{2},\frac{u-v}{2})$$
    and
    $$A_-(u,v,f)=Z_{\eta^{-1}}^{\textnormal{BF}}(\tilde{f},\hat{\Phi},\frac{u-v}{2},\frac{u+v}{2}).$$
    Note that $A_{\pm}(u,v,f)$ are both absolutely convergent in the region $\mathcal{D}_N=\{(u,v)\in \mathbb{C}^2|\;\textnormal{Re}(u)>|\textnormal{Re}(v)|+2a_N\}$. And the functional equation has become
    $$A_+(u,v,f)=A_-(-u,v,f).$$
    To prove that $Z^{\textnormal{BF}}_{\eta}(f,\Phi,s_1,s_2)$ admits an analytic continuation to $\mathbb{C}^2$ for any $f\in W$, we only need to prove that $A_+(u,v,f)$ extends to an entire function on $\mathbb{C}^2$ for any $f\in W$.

    Note that $\mathcal{S}_{\chi}([G])$ is dense in $W$. For every fixed $v\in \mathbb{C}$, applying \cite[Corollary A.0.11.2]{beuzart2022global} to $W$, $\mathcal{S}_{\chi}([G])$ and $A_{\pm}(u,v,\cdot)$, we can get holomorphic extensions of $A_{\pm}(u,v,f)$, with values in the space of continuous functionals on $W$. Denote the resulting continuation by $\tilde{A}_{\pm}(u,v,f)$. We know the following properties of $\tilde{A}_{\pm}$:
    \begin{itemize}
        \item For every fixed $v\in \mathbb{C}$ and $f\in W$, the functions $u\mapsto \tilde{A}_{\pm}(u,v,f)$ are holomorphic and of finite order in vertical strips.
        \item For every fixed pair $(u,v)\in \mathbb{C}^2$, the functionals $f\mapsto \tilde{A}_{\pm}(u,v,f)$ is continuous on $W$.
        \item For any $u,v\in \mathbb{C}$ and $f\in W$, we have the following functional equation:
        $$\tilde{A}_+(u,v,f)=\tilde{A}_{-}(-u,v,f).$$
    \end{itemize}

    Now, fix $f\in W$, and choose a relatively compact open subset $V\subset \mathbb{C}$. Choose a number $c_V$ such that $c_V>2a_N+\sup_{v\in V}|\textnormal{Re}(v)|$. For any $u\in \mathcal{H}_{>c_V}$, define a function $B_+(u)$ on $V$ by
    $$B_+(u)(v)=A_+(u,v,f).$$
    Similarly, define
    $$B_-(u)(v)=A_-(u,v,f).$$
    By \Cref{absolute convergence 1}, the functions $B_{\pm}(u)$ are holomorphic functions for $v\in V$. Moreover, the $\mathcal{O}(V)$-valued functionals $u\mapsto B_{\pm}(u)$ are holomorphic on $\mathcal{H}_{>c_V}$, and of finite order (actually of order $0$ by absolute convergence) in vertical strips, where $\mathcal{O}(V)$ denotes the Fr\'echet space of holomorphic functions on $V$. 

    For any $v\in V$, define the continuous functional $\lambda_v$ on $\mathcal{O}(V)$ by $\lambda_v(h)=h(v)$ for any $h\in \mathcal{O}(V)$. Denote the vector space generated by all such $\lambda_v$'s by $H_V$. Then $H_V$ is a total subspace since if $h\in\mathcal{O}(V)$ satisfies that $\lambda_v(h)=0$ for any $v\in V$, then $h=0$.

    Note that for any $v\in V$, the functions $\lambda_v\circ B_{\pm}(u)=A_{\pm}(u,v,f)$ extends to holomorphic functions $\tilde{A}_{\pm}(u,v,f)$ satisfying $\tilde{A}_+(u,v,f)=\tilde{A}_{-}(-u,v,f)$. Applying \cite[Lemma A.0.10.1]{beuzart2022global} to $\mathcal{O}(V)$, $H_V$ and $B_{\pm}$, we get entire continuations of $B_{\pm}$. We denote them by $\tilde{B}_{\pm}$. In other words, $\tilde{B}_+(u)$ and $\tilde{B}_-(u)$ are holomorphic functions on $V$ for any $u\in \mathbb{C}$. Since for any $v\in V$, $\lambda_v$ is a continuous functional on $\mathcal{O}(V)$, and the function $u\mapsto \tilde{B}_+(u)$ is entire as an $\mathcal{O}(V)$-valued map, we have that for any fixed $v\in V$, the function 
    $$u\mapsto \lambda_v\circ \tilde{B}_+(u)=\tilde{B}_+(u)(v)$$
    is entire. Moreover, fix $v\in V$, and for $u\in \mathcal{H}_{>c_V}$, we have 
    $$\tilde{B}_+(u)(v)=B_+(u)(v)=A_+(u,v,f)=\tilde{A}_+(u,v,f).$$
    By the uniqueness of analytic continuation, we can deduce that $\tilde{B}(u)(v)$ is the same function as $\tilde{A}(u,v,f)$ for any $(u,v)\in \mathbb{C}\times V$. Thus, by Hartogs' theorem,  $\tilde{A}_+(u,v,f)$ is holomorphic for $(u,v)\in \mathbb{C}\times V$. Let $V$ vary. We prove that $A_+(u,v,f)$ extends to an entire function on $\mathbb{C}^2$ (on the overlap of two choices of $V$, the constructions agree because they agree in the original region and hence everywhere by uniqueness of analytic continuation).
    This proves (2). Let $P^*_{\eta}(f,\Phi)=Z_{\eta}^{\textnormal{BF}}(f,\Phi,0,0)$. This proves (1). The proof is complete.

\end{proof}

\subsection{Euler decomposition}

For future application, we will choose $f=E(\varphi)$, the Eisenstein series of a function $\varphi\in \Pi=\textnormal{Ind}_{P(\mathbb{A})}^{G(\mathbb{A})}\pi$, where $\pi$ is a central twist of $\chi$. Then $f\in \mathcal{T}_{\chi}([G])$.

Let $S$ be a sufficiently large set of places of $F$, that we assume to contain Archimedean places as well as the places where $\Pi$, $\varphi$, $\eta$ and $\Phi$ is ramified. We then have a decomposition $W_{E(\varphi)}=W_{E(\varphi),S}W_{E(\varphi)}^S$ such that $W_{E(\varphi)}^S(1)=1$. We also write $\Phi=\Phi_S\Phi^S$, where $\Phi^S$ is the characteristic function of $(\mathcal{O}_F^S)^n$ and $\Phi_S\in \mathcal{S}(F_S^n)$.

By the unramified computation for the Bump--Friedberg period in \cite[Proposition 4.4]{leslie2025unitary}, we have
\begin{equation}\label{L function BF}
    Z_{\eta}^{\textnormal{BF}}(E(\varphi),\Phi,s_1,s_2)=(\Delta_H^{S,*})^{-1}\tilde{Z}_{\eta,S}^{\textnormal{BF}}(W_{E(\varphi),S},\Phi_S,s_1,s_2)L^S(s_1+\frac{1}{2}, \Pi\otimes \eta^{-1})L^S(s_1+s_2+\frac{1}{2}, \Pi, \wedge^2\otimes \eta^{-1}).
\end{equation}
where $\Delta_H^{S,*}$ denotes the leading coefficient of the partial $L$-function $L^S_H(s)$ at $s=0$ (see \cite[\S 2.2.1]{lu2025periodsdetectingeisensteinseries} for details) and $\tilde{Z}_{\eta,S}^{\textnormal{BF}}(W_{E(\varphi),S},\Phi_S,s_1,s_2)$ is defined to be the meromorphic continuation of
$$\int_{N_n(F_S)\backslash G_n(F_S)}\int_{N_n(F_S)\backslash G_n(F_S)}W_{E(\varphi),S}(\iota(x,y))\eta_S^{-1}(\textnormal{det}(x))\Phi_S(e_ny)|x|_S^{s_1}|y|_S^{s_2+\frac{1}{2}}\mathrm{d}x\mathrm{d}y.$$
Indeed, by \cite[Proposition 4.4]{leslie2025unitary}, the above integral is absolutely convergent on the region $\{(s_1,s_2)\in \mathbb{C}^2|\; \textnormal{Re}(s_1)>0, \textnormal{Re}(s_2)>\textnormal{Re}(s_1)+1\}$, and \Cref{L function BF} holds in this region. Since its left hand-side is entire by \Cref{main result} and the two partial $L$-functions appearing on its right-hand side admit meromorphic continuation, \Cref{L function BF} defines a meromorphic continuation of the above integral.

\section{An augmented Bump--Friedberg period}\label{augmented}

In this section, we will introduce an augmented Bump--Friedberg period. All statements in this section will be very similar to \Cref{twisted BF}, and all proofs will be brief.

\subsection{Statements of the main results}

Throughout this section, we fix a positive integer $n$. Let $G=G_1\times G_{2n}$, $H=G_n\times G_n$ and we embed $H$ into $G$ by $\nu(x,y)=(\textnormal{det}(x)^{-1},\iota(x,y))$.

We define a bilinear functional $\mathcal{S}([G])\times \mathcal{S}(\mathbb{A}^n)\times \mathbb{C}^2\mapsto \mathbb{C}$ by
$$\mathcal{P}^a(f,\Phi,s_1,s_2)=\int_{[G_n]\times [G_n]}f(\nu(x,y))\Theta(y,\Phi)|x|^{s_1}|y|^{s_2}.$$
By the same argument as in \Cref{twisted BF period}, the integral defining $\mathcal{P}^a$ is absolutely convergent and defines a continuous bilinear functional on $\mathcal{S}([G])\times \mathcal{S}(\mathbb{A}^n)$ for any $s_1,s_2\in\mathbb{C}$. Moreover, for any $f\in \mathcal{S}([G])$ and $\Phi\in \mathcal{S}(\mathbb{A}^n)$, $\mathcal{P}^a(f,\Phi,s_1,s_2)$ defines an entire function on $\mathbb{C}^2$ which is bounded on vertical strips.

Let $\psi$ denote the same character on $N_{2n}(\mathbb{A})$ as in \Cref{twisted BF}. For any $f\in \mathcal{T}([G])$, we define the corresponding Fourier coefficient
$$f_{N_{r,2n,\psi}}(a,g)=\int_{[N_{r,2n}]}f(a,ug)\psi^{-1}(u)\mathrm{d}u$$
and its Whittaker function
$$W_f(a,g)=f_{N_{2n,\psi}}(a,g)=\int_{[N_{2n}]}f(a,ug)\psi^{-1}(u)\mathrm{d}u.$$

Let $f\in\mathcal{T}([G])$, $\Phi\in \mathcal{S}(\mathbb{A}^n)$ and $s_1,s_2\in \mathbb{C}$. Define the following integral:
$$Z^a(f,\Phi,s_1,s_2)=\int_{N_n(\mathbb{A})\backslash G_n(\mathbb{A})}\int_{N_n(\mathbb{A})\backslash G_n(\mathbb{A})}W_f(\nu(x,y))\Phi(e_ny)|x|^{s_1}|y|^{s_2+\frac{1}{2}}\mathrm{d}x\mathrm{d}y.$$
We have the following lemma, which provides the absolute convergence of the above expression of integral:

\begin{lemma}
    For any $N\in \mathbb{N}$, there exists $c_N>0$ such that
    \begin{itemize}
        \item[(1)] For any $f\in \mathcal{T}_N([G])$ and $\Phi\in \mathcal{S}(\mathbb{A}^n)$, the expression defining $Z^{a}(f,\Phi,s_1,s_2)$ converges absolutely for $s_1,s_2\in \mathcal{H}_{>c_N}$, and defines a holomorphic function on $\mathcal{H}_{>c_N}\times \mathcal{H}_{>c_N}$.
        \item[(2)] The functional $f\mapsto Z^a(f,\Phi,s_1,s_2)$ is continuous on $\mathcal{T}_N([G])$ for any $s_1,s_2\in \mathcal{H}_{>c_N}$ and $\Phi\in \mathcal{S}(\mathbb{A}^n)$.
    \end{itemize}
\end{lemma}
\begin{proof}
    The proof is very similar to the proof of \Cref{absolute convergence 1}. The only difference is the term $\|m\|_{M_P}^{-N}$ appearing in \Cref{important lemma}(2). It adds an extra $\|\textnormal{det}(x)^{-1}\|_{G_1}^N=\|\textnormal{det}(x)\|_{G_1}^N$ in the estimate of $W_f(\nu(x,y))$. Note that there exists $M>0$ such that $\|t\|_{G_1}\ll \max(|t|^M,|t|^{-M})$ for any $t\in \mathbb{A}^{\times}$, so the extra term will not affect the absolute convergence. We omit the details.
\end{proof}

\begin{defn}
    Let $\chi=\eta\boxtimes \chi'$ be a cuspidal datum of $G$, where $\eta$ is a Hecke character trivial on $A_{G_1}^{\infty}$ and $\chi'$ is a cuspidal datum of $G_{2n}$. We say $\chi$ is $\Delta_a^*$-regular, if $\chi'$ is $(\eta,\Delta^*)$-regular.
\end{defn}

Let $\mathfrak{X}_{\Delta_a^*}(G)\subset \mathfrak{X}(G)$ denote the set of $\Delta_a^*$-regular cuspidal datum. We write $\mathcal{S}_{\Delta_a^*}([G])$ (resp. $\mathcal{T}_{\Delta_a^*}([G])$) for $\mathcal{S}_{\mathfrak{X}_{\Delta_a^*}(G)}([G])$ (resp. $\mathcal{T}_{\mathfrak{X}_{\Delta_a^*}(G)}([G])$).

We now state the main result of this section:
\begin{thm}\label[thm]{main result 2}
    Let $\chi$ be a $\Delta_a^*$-regular cuspidal datum. Fix $\Phi\in \mathcal{S}(\mathbb{A}^n)$. We have the following statements:
    \begin{itemize}
        \item[(1)] The continuous functional $\mathcal{P}^a(\cdot,\Phi,0,0)$ on $\mathcal{S}_{\chi}([G])$ extends by continuity to a continuous functional $\mathcal{P}^{a,*}(\cdot,\Phi)$ on $\mathcal{T}_{\chi}([G])$.
        \item[(2)] For every $f\in \mathcal{T}_{\chi}([G])$, the function $(s_1,s_2)\mapsto Z^a(f,\Phi,s_1,s_2)$, a priori defined on the region $\mathcal{H}_{>c_N}\times \mathcal{H}_{>c_N}$, extends to an entire function on $\mathbb{C}^2$. Moreover, we have
        $$\mathcal{P}^{a,*}(f,\Phi)=Z^a(f,\Phi,0,0).$$
    \end{itemize}
\end{thm}

The proof will be given in the next subsection.

\subsection{Proof of \Cref{main result 2}}

We need to introduce two more zeta integrals. For $f\in \mathcal{S}([G])$, $\Phi\in\mathcal{S}(\mathbb{A}^n)$ and $1\leq r\leq n$, we put
$$Z_{r}^a(f,\Phi,s_1,s_2)=\int_{\mathcal{P}_r(F)N_{r,n}(\mathbb{A})\backslash G_n(\mathbb{A})}\int_{G_{r-1}(F)N_{r-1,n}(\mathbb{A})\backslash G_n(\mathbb{A})}f_{N_{2r-1,2n,\psi}}(\nu(x,y))\Phi(e_ny)|x|^{s_1}|y|^{s_2+\frac{1}{2}}\mathrm{d}x\mathrm{d}y.$$
and
$$Z_{r}^{'a}(f,\Phi,s_1,s_2)=\int_{G_r(F)N_{r,n}(\mathbb{A})\backslash G_n(\mathbb{A})}\int_{\mathcal{P}_{r}(F)N_{r,n}(\mathbb{A})\backslash G_n(\mathbb{A})}f_{N_{2r,2n,\psi}}(\nu(x,y))\Phi(e_ny)|x|^{s_1}|y|^{s_2+\frac{1}{2}}\mathrm{d}x\mathrm{d}y.$$
Note that $Z_{1}^a(f,\Phi,s_1,s_2)=Z^a(f,\Phi,s_1,s_2)$. We have the following lemma:
\begin{lemma}
    For any $f\in \mathcal{S}([G])$, $\Phi\in\mathcal{S}(\mathbb{A}^n)$, $s_1,s_2\in \mathbb{C}$ and $1\leq r\leq n$, the expressions defining $Z_{r}^a(f,\Phi,s_1,s_2)$ and $Z_{r}^{'a}(f,\Phi,s_1,s_2)$ are absolutely convergent.
\end{lemma}
\begin{proof}
    The proof is similar to the proof of \Cref{absolute convergence 3}. The only difference is the $\|m\|_{M_P}^{-N}$ appearing in \Cref{important lemma}(1). It adds an extra $\|\textnormal{det}(x)^{-1}\|_{G_1}^{-N_2}=\|\textnormal{det}(x)\|_{G_1}^{-N_2}$ in the estimate of $f_{N_{k,2n,\psi}}(\nu(x,y))$. Since $\|\textnormal{det}(x)\|_{G_1}\geq 1$, we can directly ignore this term. We omit the details.
\end{proof}

Then we have the unfolding equation:

\begin{prop}\label[prop]{unfolding 1}
    Let $\chi\in \mathfrak{X}_{\Delta_a^*}(G)$. Then for any $(f,\Phi)\in \mathcal{S}_{\chi}([G])\times \mathcal{S}(\mathbb{A}^n)$ and $s_1,s_2\in\mathbb{C}$, we have
    $$\mathcal{P}^a(f,\Phi,s_1,s_2)=Z_{n}^{'a}(f,\Phi,s_1,s_2)=Z_{n}^a(f,\Phi,s_1,s_2)=\cdots=Z_{1}^{'a}(f,\Phi,s_1,s_2)=Z_{1}^a(f,\Phi,s_1,s_2).$$
\end{prop}

\begin{proof}
    The unfolding step is identical to that of \Cref{unfolding twisted BF}. The vanishing statement follows directly from the two lemmas below. We omit the details.
\end{proof}

\begin{lemma}\label[lemma]{lemma 5}
    Fix an arbitrary positive integer $r$. Let $G=G_{2r}$ and $H=G_r\times G_r$. We embed $H$ into $G$ by $\nu$. Let $\chi\in \mathfrak{X}_{\Delta_a^*}(G)$. Then for any $f\in \mathcal{S}_{\chi}([G])$, we have
    \begin{equation}
    \int_{[H]}f(\nu(x,y))|x|^{s_1}|y|^{s_2}\mathrm{d}x\mathrm{d}y=0
    \end{equation}
    for any $s_1,s_2\in \mathbb{C}$.
\end{lemma}
\begin{proof}
    Assume $\chi=\eta\boxtimes\chi'$, where $\eta$ is a Hecke character trivial on $A_{G_1}^{\infty}$ and $\chi'\in \mathfrak{X}_{(\eta,\Delta^*)}(G_{2r})$. Since $\mathcal{S}_{\eta}([G_1])\otimes \mathcal{S}_{\chi'}([G_{2r}])$ is dense in $\mathcal{S}_{\chi}([G])$, we only need to consider the case when $f$ is pure tensor product. Write $f=\varphi\otimes f'$, where $\varphi\in \mathcal{S}_{\eta}([G_1])$ and $f'\in \mathcal{S}_{\chi'}([G_{2r}])$. Then the above integral becomes
    $$\int_{[H]}f(\nu(x,y))|x|^{s_1}|y|^{s_2}\mathrm{d}x\mathrm{d}y=\int_{[H]}f'(\iota(x,y))\varphi(\textnormal{det}(x)^{-1})|x|^{s_1}|y|^{s_2}\mathrm{d}x\mathrm{d}y.$$
    By the definition of $\mathcal{S}_{\eta}([G_1])$, we can write $$\varphi=\eta\cdot(\rho\circ H_{G_1}),$$
    where $\rho\in \mathcal{S}(\mathbb{R})$ and $H_{G_1}(a)=\log|a|$ for any $a\in \mathbb{A}^{\times}$. By Fourier inversion, we have 
    $$\rho(u)=\frac{1}{2\pi}\int_{\mathbb{R}}\hat{\rho}(t)e^{itu}\mathrm{d}t$$
    for any $u\in \mathbb{R}$. Substituting this into the above integral, we obtain
    $$\int_{[H]}f'(\iota(x,y))\varphi(\textnormal{det}(x)^{-1})|x|^{s_1}|y|^{s_2}\mathrm{d}x\mathrm{d}y=\frac{1}{2\pi}\int_{\mathbb{R}}\hat{\rho}(t)\int_{[H]}f'(\iota(x,y))\eta(\textnormal{det}(x)^{-1})|x|^{s_1-it}|y|^{s_2}\mathrm{d}x\mathrm{d}y\mathrm{d}t=0$$
    by \Cref{lemma 1}. The proof is complete.
\end{proof}

\begin{lemma}
    Fix an arbitrary positive integer $r$. Let $G=G_{2r+1}$ and $H=G_{r+1}\times G_r$. We embed $H$ into $G$ by $\nu(x,y)=(\textnormal{det}(x)^{-1},\iota(x,y))$. Let $\chi\in \mathfrak{X}_{\Delta_a^*}(G)$. Then for any $f\in \mathcal{S}_{\chi}([G])$, we have
    \begin{equation}
    \int_{[H]}f(\nu(x,y))|x|^{s_1}|y|^{s_2}\mathrm{d}x\mathrm{d}y=0
    \end{equation}
    for any $s_1,s_2\in \mathbb{C}$.
\end{lemma}
\begin{proof}
    The proof is identical to the above proof. The only difference is that we use \Cref{lemma 2} instead of \Cref{lemma 1}. We leave the details to the reader.
\end{proof}

We now turn to the proof of the main theorem:
\begin{proof}[Proof of \Cref{main result 2}]
    By \Cref{unfolding 1}, we can get
    $$\mathcal{P}^a(f,\Phi,s_1,s_2)=Z^a(f,\Phi,s_1,s_2)$$
    for any $f\in\mathcal{S}_{\chi}([G])$, $\Phi\in \mathcal{S}(\mathbb{A}^n)$ and $s_1,s_2\in \mathbb{C}$. Write $\tilde{f}(a,g)=f(a^{-1},{^tg^{-1}})$. If $f\in \mathcal{S}_{\chi}([G])$, then $\tilde{f}\in \mathcal{S}_{\chi^{\vee}}([G])$. Now by the change of variable $(x,y)\mapsto({^tx^{-1}},{^ty^{-1}})$ and \Cref{theta series}, we have
    $$Z^a(f,\Phi,s_1,s_2)=Z^a(\tilde{f},\hat{\Phi},-s_1,-s_2).$$
    By the definition of $\Delta_a^*$-regularity, we can see that $\chi\in \mathfrak{X}_{\Delta_a^*}(G)$ is equivalent to $\chi^{\vee}\in \mathfrak{X}_{\Delta_a^*}(G)$. The rest of the proof is the same as the proof of \Cref{main result}. We omit the details.
\end{proof}

\section{A new $\textnormal{SL}_{n+1}\times \textnormal{GL}_n$ period on $\textnormal{GL}_{2n+1}$}\label{new}
\subsection{Statements of the main results}
\subsubsection{Notation}\label{notation 5}
For any positive integer $n$, let $G=G_{2n+1}$ and $H=S_{n+1}\times G_n$. We embed $H$ into $G$ in the following way: we put $S_{n+1}$ in all positions $(i,j)$ where $i,j\in \{ 1,2,4,...,2n\}$ and $G_n$ in all positions $(i,j)$ where $i,j\in \{ 3,5,...,2n-1,2n+1\}$. We call this embedding $\iota'$. The following is an example:

\begin{eg}
    When $n=3$, the embedding of $H$ into $G$ is as follows:
    $$\begin{pmatrix}
        * & * &  & * &  & *&  \\
        * & * &  & * &  & *&  \\
         & & \circ & & \circ & & \circ \\
         * & * &  & * &  & *& \\
         & & \circ & & \circ & & \circ \\
         * & * &  & * &  & *& \\
         & & \circ & & \circ & & \circ \\
    \end{pmatrix}.$$
\end{eg}

Similarly, for $G=G_{2n+2}$ and $H=S_{n+2}\times G_n$, we can also define an embedding from $H$ to $G$, which is also called by $\iota'$, by putting the first $S_{n+2}$ factor into all positions $(i,j)$ where $i,j\in \{ 1,2,4,...,2n,2n+2\}$, and putting the second $G_n$ factor into all positions $(i,j)$ where $i,j\in \{ 3,5,...,2n-1,2n+1\}$. The following is an example:

\begin{eg}
    When $n=2$, the embedding of $H$ into $G$ is as follows:
    $$\begin{pmatrix}
        * & * &  & * &  & *\\
        * & * &  & * &  & *\\
         & & \circ & & \circ &\\
         * & * &  & * &  & *\\
         & & \circ & & \circ &\\
         * & * &  & * &  & *\\
    \end{pmatrix}.$$
\end{eg}
We denote both embedding by the same symbol $\iota'$. For a matrix with odd size, we mean the first embedding; and for a matrix with even size, we mean the second embedding.

Throughout this section, we fix a positive integer $n$, and let $G=G_{2n+1}$, $H=S_{n+1}\times G_n$. We embed $H$ into $G$ by $\iota'$. Let $\mathcal{P}_r^1$ denote the standard mirabolic subgroup of $S_r$ (i.e. the stabilizer of $e_r$ in $S_r$), and if $r\leq n+1$, we embed  $\mathcal{P}_r^1$ into $S_{n+1}$ by putting it into the upper left corner. 

We also introduce the following definitions:
\begin{itemize}
    \item For $0\leq r\leq n+1$, let $P_{r,n+1}^1$ denote the parabolic subgroup of $S_{n+1}$ corresponding to the partition $(r,1,\dots,1)$ of $n+1$, and let $N_{r,n+1}^1$ denote the unipotent radical of $P_{r,n+1}^1$. Let $N_{n+1}^1:=N_{0,n+1}^1=N_{1,n+1}^1$ be the upper triangular maximal unipotent subgroup of $S_{n+1}$.
    \item Let $\psi'$ denote the character on $N_{2n+1}(\mathbb{A})$ defined by $\psi'(u)=\psi(u_{2,3}+\dots +u_{2n,2n+1})$. The character induced by $\psi'$ on the subgroup $N_{r,2n+1}$ will also be denoted by $\psi'$ for any $0\leq r\leq 2n+1$. Note that $\psi'|_{N_{r,2n+1}}=\psi|_{N_{r,2n+1}}$ for any $2\leq r\leq 2n+1$ (for the definition of $\psi$, see \Cref{notation section 3}).
\end{itemize}

\subsubsection{The period}

We define a bilinear functional $\mathcal{S}([G])\times \mathcal{S}(\mathbb{A}^n)\mapsto \mathbb{C}$ by
$$\mathcal{P}(f,\Phi)=\int_{[S_{n+1}]\times [G_n]}f(\iota'(x,y))\Theta(y,\Phi)\mathrm{d}x\mathrm{d}y.$$

By the same argument as in the previous discussion, we can see that the integral defining $\mathcal{P}(f,\Phi)$ is absolutely convergent and defines a continuous bilinear functional on $\mathcal{S}([G])\times \mathcal{S}(\mathbb{A}^n)$.

\subsubsection{Zeta integrals}

For $f\in \mathcal{T}([G])$ and $0\leq r\leq 2n+1$, we define the corresponding Fourier coefficient
$$f_{N_{r,2n+1,\psi'}}(g)=\int_{[N_{r,2n+1}]}f(ug)\psi'^{-1}(u)\mathrm{d}u.$$
We associate its degenerate Whittaker function
$$V_f(g)=f_{N_{2n+1,\psi'}}(g)=\int_{[N_{2n+1}]}f(ug)\psi'^{-1}(u)\mathrm{d}u.$$
This is a degenerate Whittaker function because $\psi'$ is trivial in $u_{1,2}$.

Let $P_{(1,n)}^1$ denote the standard parabolic subgroup of $S_{n+1}$ corresponding to the partition $(1,n)$ of $n+1$.For any $f\in\mathcal{T}([G])$, $\Phi\in\mathcal{S}(\mathbb{A}^n)$, $\lambda\in \mathfrak{a}_{P_{(1,n)}^1,\mathbb{C}}^*$ and $s_2\in\mathbb{C}$, define the following integral:

$$Z(f,\Phi,\lambda,s_2)=\int_{N_{n+1}^1(\mathbb{A})\backslash S_{n+1}(\mathbb{A})}\int_{N_n(\mathbb{A})\backslash G_n(\mathbb{A})}V_f(\iota'(x,y))\Phi(e_ny)e^{\langle \lambda,H_{P_{(1,n)}^1}(x)\rangle}|y|^{s_2+\frac{1}{2}}\mathrm{d}x\mathrm{d}y$$
provided by the expression converges absolutely. Note that $\mathfrak{a}_{P_{(1,n)}^1,\mathbb{C}}^*$ is a one-dimensional complex vector space, and we can define the coordinate $s_{\lambda}$ of $\lambda$ such that
$$\exp(\langle \lambda,H_{P_{(1,n)}^1}(\begin{pmatrix}
    \textnormal{det}(g)^{-1} & \\ & g
\end{pmatrix})\rangle)=|g|^{s_{\lambda}}.$$

The following lemma will provide the absolute convergence of the above expression of integral.

\begin{lemma}\label[lemma]{absolute convergence 4}
    For any $N\in \mathbb{N}$, there exists $c_N>0$ such that
    \begin{itemize}
        \item[(1)] For any $f\in \mathcal{T}_N([G])$ and $\Phi\in \mathcal{S}(\mathbb{A}^n)$, the expression defining $Z(f,\Phi,\lambda,s_2)$ converges absolutely for $s_{\lambda},s_2\in \mathcal{H}_{>c_N}$, and defines a holomorphic function on the region $\mathcal{D}_{c_N}:=\{(\lambda,s_2)|\; (s_{\lambda},s_2)\in\mathcal{H}_{>c_N}\times \mathcal{H}_{>c_N}\}$.
        \item[(2)] The functional $f\mapsto Z(f,\Phi,\lambda,s_2)$ is continuous on $\mathcal{T}_N([G])$ for any $(\lambda,s_2)\in \mathcal{D}_{c_N}$ and $\Phi\in \mathcal{S}(\mathbb{A}^n)$.
    \end{itemize}
\end{lemma}
\begin{lemma}\label[lemma]{absolute convergence 5}
    We have the following statements:
    \begin{itemize}
        \item[(1)] For any $f\in \mathcal{S}([G])$, $\Phi\in \mathcal{S}(\mathbb{A}^n)$, $\lambda\in \mathfrak{a}_{P_{(1,n)}^1,\mathbb{C}}^*$ and $s_2\in\mathbb{C}$, the expression defining $Z(f,\Phi,\lambda,s_2)$ converges absolutely and defines an entire function on $\mathfrak{a}_{P_{(1,n)}^1,\mathbb{C}}^*\times \mathbb{C}$.
        \item[(2)] The functional $f\mapsto Z(f,\Phi,\lambda,s_2)$ is continuous on $\mathcal{S}([G])$ for any $\lambda\in \mathfrak{a}_{P_{(1,n)}^1,\mathbb{C}}^*$, $s_2\in \mathbb{C}$ and $\Phi\in \mathcal{S}(\mathbb{A}^n)$.
    \end{itemize}
\end{lemma}
The proof will be provided in the next subsection.

\subsubsection{$\Delta^*_1$-regular cuspidal datum}\label{delta1}
Let $r\geq 2$ be an integer, and set $G=G_r$. Let $\chi\in \mathfrak{X}(G)$ be a cuspidal datum of $G$, represented by a pair $(M,\pi)$, where
$$M=\prod_{i=1}^lG_{r_i},\qquad\pi=\pi_1\boxtimes\cdots\boxtimes\pi_l$$
and each $\pi_i$ is a cuspidal representation of $G_{r_i}$.

Let $J_1(\chi):=\{i\;|\;r_i=1\}$. For each $i\in J_1(\chi)$, write $\eta_i=\pi_i$ viewed as a Hecke character of $G_1(\mathbb{A})=\mathbb{A}^{\times}$ and let $(M^i,\pi^i):=(\prod_{j\neq i}G_{r_j},\boxtimes_{j\neq i}\pi_j)$. Denote by $\chi^i=[M^i,\pi^i]$ the cuspidal datum of $G_{r-1}$ by removing the $i$-th block.

\begin{defn}\label[defn]{definition1}
    We say that $\chi$ is $\Delta_1^*$-regular if, for every $i\in J_1(\chi)$, the cuspidal datum $\chi^{i}$ is $(\eta_i,\Delta^*)$-regular.

    In particular, if $J_1(\chi)=\varnothing$, then $\chi$ is automatically $\Delta_1^*$-regular.
\end{defn}

Let $\mathfrak{X}_{\Delta_1^*}(G)\subset \mathfrak{X}(G)$ denote the set of $\Delta_1^*$-regular cuspidal datum. We write $\mathcal{S}_{\Delta_1^*}([G])$ (resp. $\mathcal{T}_{\Delta_1^*}([G])$) for $\mathcal{S}_{\mathfrak{X}_{\Delta_1^*}(G)}([G])$ (resp. $\mathcal{T}_{\mathfrak{X}_{\Delta_1^*}(G)}([G])$).

\begin{rmk}
    The definition is independent of the ordering of the blocks of $M$, since the condition is imposed for every $G_1$-block. It is also invariant under twists in the $A_M^\infty$-direction, by \Cref{Remark 1}.
\end{rmk}

\subsubsection{Main results}

Here we state the main result of this section:

\begin{thm}\label[thm]{main result 1}
    Let $\chi\in \mathfrak{X}_{\Delta_1^*}(G)$ be a $\Delta_1^*$-regular cuspidal datum. Fix $\Phi\in \mathcal{S}(\mathbb{A}^n)$. We have the following statements:
    \begin{itemize}
        \item[(1)] The continuous functional $\mathcal{P}(\cdot,\Phi)$ on $\mathcal{S}_{\chi}([G])$ extends by continuity to a continuous functional $\mathcal{P}^*(\cdot,\Phi)$ on $\mathcal{T}_{\chi}([G])$.
        \item[(2)] For every $f\in \mathcal{T}_{\chi}([G])$, the function $(\lambda,s_2)\mapsto Z(f,\Phi,\lambda,s_2)$, a priori defined on $\mathcal{D}_{c_N}$, extends to an entire function on $\mathfrak{a}_{P_{(1,n)}^1,\mathbb{C}}^*\times \mathbb{C}$. Moreover, we have
        $$\mathcal{P}^*(f,\Phi)=Z(f,\Phi,0,0).$$
    \end{itemize}
\end{thm}
The proof of \Cref{main result 1} will be given in \Cref{proof of section 5}.

\subsection{Convergence of zeta integrals}
\subsubsection{More zeta integrals}
We will introduce two more zeta integrals. For $f\in \mathcal{S}([G])$, $\Phi\in\mathcal{S}(\mathbb{A}^n)$ and $1\leq r\leq n$, we put
$$Z_r(f,\Phi)=\int_{\mathcal{P}^1_{r+1}(F)N^1_{r+1,n+1}(\mathbb{A})\backslash S_{n+1}(\mathbb{A})}\int_{G_{r-1}(F)N_{r-1,n}(\mathbb{A})\backslash \textnormal{GL}_n(\mathbb{A})}f_{N_{2r,2n+1,\psi'}}(\iota'(x,y))\Phi(e_ny)|y|^{1/2}\mathrm{d}x \mathrm{d}y$$
and 
$$Z_r'(f,\Phi)=\int_{S_{r+1}(F)N^1_{r+1,n+1}(\mathbb{A})\backslash S_{n+1}(\mathbb{A})}\int_{\mathcal{P}_{r}(F)N_{r,n}(\mathbb{A})\backslash G_n(\mathbb{A})}f_{N_{2r+1,2n+1,\psi'}}(\iota'(x,y))\Phi(e_ny)|y|^{1/2}\mathrm{d}x \mathrm{d}y,$$
where the absolute convergence is provided by the following lemma:
\begin{lemma}\label[lemma]{absolute convergence 6}
    For any $f\in \mathcal{S}([G])$, $\Phi\in \mathcal{S}(\mathbb{A}^n)$ and $1\leq r\leq n$, the expressions defining $Z_r(f,\Phi)$ and $Z'_r(f,\Phi)$ are absolutely convergent.
\end{lemma}
The proof of this lemma will be given below.

\subsubsection{Proof of \Cref{absolute convergence 4}}\label{proof of 4}
\begin{proof}
    The proof is essentially the same as \Cref{T_N proof}. Let $B_H$ denote the standard Borel subgroup of $H$. By the Iwasawa decomposition $H(\mathbb{A})=B_H(\mathbb{A})K_H$, we need to find a real number $c_N>0$ such that the following integral
    \begin{align}\label{SH1}
        \int_{K_H}\int_{(\mathbb{A}^{\times})^{2n}} & |V_f(\iota'(D(a_0,a_1,\dots,a_n),D(b_1,\dots,b_n))k)| |\Phi(b_ne_nk)| \delta_{B_H}(D(a_0,a_1,\dots,a_n),D(b_1,\dots,b_n))^{-1}\notag \\ & \prod_{i=1}^n|a_i|^{\textnormal{Re}(s_{\lambda})}\prod_{i=1}^n|b_i|^{\textnormal{Re}(s_2)+\frac{1}{2}} \prod_{i=1}^n\mathrm{d}a_i\mathrm{d}b_i\;\mathrm{d}k
    \end{align}
    is absolutely convergent when $s_{\lambda},s_2\in \mathcal{H}_{>c_N}$, where
    $$D(a_0,a_1,\dots,a_n)=\textnormal{diag}(a_0,a_1,\dots,a_n),$$
    $$D(b_1,\dots,b_n)=\textnormal{diag}(b_1,\dots,b_n),$$
    $$a_0=(\prod_{i=1}^na_i)^{-1}$$
    and the modulus character is given by
    $$\delta_{B_H}(D(a_0,a_1,\dots,a_n),D(b_1,\dots,b_n))=\prod_{i=1}^n|a_i|^{-2i}\prod_{i=1}^n|b_i|^{n+1-2i}.$$
    By \Cref{important lemma}(2), for every $N_1\geq 0$,
    $$|V_f(\iota'(D(a_0,a_1,\dots,a_n),D(b_1,\dots,b_n))k)|\ll \prod_{i=1}^n\|\frac{a_i}{b_i}\|_{\mathbb{A}}^{-N_1}\prod_{i=1}^{n-1}\|\frac{b_i}{a_{i+1}}\|_{\mathbb{A}}^{-N_1}\prod_{i=1}^n\|a_i\|^N_{G_1}\|b_i\|^N_{G_1}$$
    for any $(k,a_1,\dots,a_n,b_1,\dots,b_n)\in K_H\times (\mathbb{A}^{\times})^{2n}$. Since $\Phi\in \mathcal{S}(\mathbb{A}^n)$, for all $N_1>0$, we have $|\Phi(b_ne_nk)|\ll\|b_n\|_{\mathbb{A}}^{-N_1}$ for any $(k,b_n)\in K_H\times \mathbb{A}^{\times}$. For any $N_2>0$, there exists $N_1$ such that
    $$\prod_{i=1}^n\|\frac{a_i}{b_i}\|_{\mathbb{A}}^{-N_1}\prod_{i=1}^{n-1}\|\frac{b_i}{a_{i+1}}\|_{\mathbb{A}}^{-N_1}\|b_n\|_{\mathbb{A}}^{-N_1}\ll \prod_{i=1}^n\|a_i\|_{\mathbb{A}}^{-N_2}\prod_{i=1}^n\|b_i\|_{\mathbb{A}}^{-N_2}.$$
    Therefore, \Cref{SH1} is essentially bounded by
    \begin{equation}\label{SH2}
        \prod_{i=1}^n\int_{G_1}\|a_i\|_{\mathbb{A}^{\times}}^N|a_i|^{\textnormal{Re}(s_{\lambda})+2i}\|a_i\|_{\mathbb{A}}^{-N_2}\prod_{i=1}^n\int_{\mathbb{A}^{\times}}\|b_i\|_{G_1}^N|b_i|^{\textnormal{Re}(s_2)-\frac{1}{2}+2i-n}\|b_i\|_{\mathbb{A}}^{-N_2}.
    \end{equation}
    Note that there exists $M>0$ such that $\|x\|_{G_1}\ll \max(|x|^M,|x|^{-M})$. Together with \cite[Lemma 2.2.5]{lu2025periodsdetectingeisensteinseries}, there exists $c_N>0$ such that when $(s_{\lambda},s_2)\in \mathcal{H}_{>c_N}\times \mathcal{H}_{>c_N}$, \Cref{SH2} is absolutely convergent. The continuity and holomorphy follow from the exact same argument as in \Cref{T_N proof}. The proof is complete.
\end{proof}

\subsubsection{Proof of \Cref{absolute convergence 6}}

\begin{proof}
    We first prove the absolute convergence of $Z_r(f,\Phi)$. Instead of proving that the original integral is absolutely convergent, wo consider the following integral:
    $$Z_r(f,\Phi,s)=\int_{\mathcal{P}^1_{r+1}(F)N^1_{r+1,n+1}(\mathbb{A})\backslash S_{n+1}(\mathbb{A})}\int_{G_{r-1}(F)N_{r-1,n}(\mathbb{A})\backslash \textnormal{GL}_n(\mathbb{A})}f_{N_{2r,2n+1,\psi'}}(\iota'(x,y))\Phi(e_ny)|y|^{s+\frac{1}{2}}\mathrm{d}x \mathrm{d}y.$$
    Note that $Z_r(f,\Phi)=Z_r(f,\Phi,0)$. Suppose we can prove that for some special values of $s$, the above integral is absolutely convergent for any $f\in \mathcal{S}([G])$ and $\Phi\in\mathcal{S}(\mathbb{A}^n)$. Consider the map $f\mapsto f|\cdot|^{-s}$ defined on $\mathcal{S}([G])$. Let $f'=f|\cdot|^{-s}$. We have that $f'\in \mathcal{S}([G])$, and $V_{f'}=V_{f}|\cdot|^{-s}$. Therefore,
    $$Z_r(f,\Phi,0)=Z_r(f',\Phi,s)$$
    and the absolute convergence of $Z_r(f,\Phi)$ is proved. So our aim is to find some special values of $s$ such that $Z_r(f,\Phi,s)$ is absolutely convergent for any $f\in \mathcal{S}([G])$ and $\Phi\in\mathcal{S}(\mathbb{A}^n)$.
    
    By the Iwasawa decomposition $S_{n+1}(\mathbb{A})=P_{r+1,n+1}^1(\mathbb{A})K^1$ and $G_n(\mathbb{A})=P_{r,n}(\mathbb{A})K^2$, we only need to show the absolute convergence of the following integral:
    \begin{align}\label{SH3}
        &\int_{K^1}\int_{(\mathbb{A}^{\times})^{n-r-1}}\int_{\mathcal{P}_{r+1}^1(F)\backslash G_{r+1}(\mathbb{A})}\int_{K^2}\int_{(\mathbb{A}^{\times})^{n-r+1}} \int_{[G_{r-1}]} \notag \\ & f_{N_{2r,2n+1,\psi'}}(\iota'(D(h_1, a_{r+1},a_{r+2},\dots,a_n)k^1,D(h_2,b_r,\dots,b_n)k^2)) \notag \\
        & \delta_{P_{r+1,n}^1}(D(h_1, a_{r+1},a_{r+2},\dots,a_n))^{-1}\delta_{P_{r-1,n}}(D(h_2,b_r,\dots,b_n))^{-1} \notag \\ & |\Phi(e_nb_nk^2)|\;|h_2|^{\textnormal{Re}(s)+\frac{1}{2}}\prod_{i=r}^n|b_i|^{\textnormal{Re}(s)+\frac{1}{2}} \;\;\;\mathrm{d}h_1\mathrm{d}k^1\prod_{i=r+2}^n\mathrm{d}a_i \mathrm{d}h_2\mathrm{d}k^2\prod_{i=r}^n\mathrm{d}b_i
    \end{align}
    where $a_{r+1}=\textnormal{det}(h_1)^{-1}\prod_{i=r+2}^n a_i^{-1}$.

    Let $m_1=D(h_1, a_{r+1},a_{r+2},...,a_n)$ and $m_2=D(h_2,b_r,...,b_n)$. Set $m=\iota'(m_1,m_2)\in M_{P_{2r,2n+1}}$. By \Cref{important lemma}(1), we can see that there exists $c>0$ such that, for any $(k^1,k^2)\in K^1\times K^2$ and $N_1,N_2\geq 0$,
    $$|f_{N_{2r,2n+1,\psi'}}(\iota'(m_1k^1,m_2k^2))|\ll \|\textnormal{Ad}^*(m^{-1})l\|_{\mathbb{A}^{2r}}^{-N_1}\|m\|_{M_{P_{2r,2n+1}}}^{-N_2}\delta_{P_{2r,2n+1}}(m)^{-cN_2},$$
    where $l=(0,...,0,1,1,...1)\in \mathbb{A}^{2n}$ (the first $1$ appears at the $2r$-th position). Therefore, 
    $$\|\textnormal{Ad}^*(m^{-1})l\|_{\mathbb{A}^{2r}}^{-N_1}\sim \|b_r^{-1}e_{r+1}h_1\|_{\mathbb{A}^{r+1}}^{-N_1}\prod_{i=r+1}^n\|a_ib_i^{-1}\|_{\mathbb{A}}^{-N_1}\prod_{i=r}^{n-1}\|b_ia_{i+1}^{-1}\|_{\mathbb{A}}^{-N_1}.$$

    Since $|\Phi(e_nb_nk^2)|\ll \|b_n\|_{\mathbb{A}}^{-N'}$ for any $k^2\in K^2$ and $N'\geq 0$, we can deduce that for any $N_3>0$, there exists $N_1>0$ such that
    $$\|b_r^{-1}e_{r+1}h_1\|_{\mathbb{A}^{r+1}}^{-N_1}\prod_{i=r+1}^n\|a_ib_i^{-1}\|_{\mathbb{A}}^{-N_1}\prod_{i=r}^{n-1}\|b_ia_{i+1}^{-1}\|_{\mathbb{A}}^{-N_1}\Phi(e_nb_nk^2)\ll \|e_{r+1}h_1\|_{\mathbb{A}^{r+1}}^{-N_3}\prod_{i=r+1}^n\|a_i\|_{\mathbb{A}}^{-N_3}\prod_{i=r}^{n}\|b_i\|_{\mathbb{A}}^{-N_3}$$
    Also, we can see that
    $$\|m\|_{M_{P_{2r,2n+1}}}^{-N_2}\sim \|h_1\|_{G_{r+1}}^{-N_2}\|h_2\|_{G_{r-1}}^{-N_2}\prod_{i=r+1}^n\|a_i\|_{G_1}^{-N_2}\prod_{i=r}^n\|b_i\|_{G_1}^{-N_2}<\|h_1\|_{G_{r+1}}^{-N_2}\|h_2\|_{G_{r-1}}^{-N_2}.$$
    This is because $M_{P_{r+1,n+1}}\times M_{P_{r-1,n}}$ embedded into $M_{P_{2r,2n+1}}$ as a subgroup, and by \cite[Proposition A.1.1]{Beuzart-Plessis_2021} we can get the desired result.

    Let us write out the explicit formulas of all the different modulus character.
    \begin{align}
        & \delta_{P_{2r,2n+1}}(\iota'(m_1,m_2)  = |h_1|^{2n-2r+1}|h_2|^{2n-2r+1}\prod_{i=r+1}^n|a_i|^{2n+2-4i}\prod_{i=r}^n|b_i|^{2n-4i}; \\ & \delta_{P_{r+1,n+1}^1}(D(h_1,a_{r+1},...,a_n))  =|h_1|^{n-r}\prod_{i=r+1}^n|a_i|^{n-2i}; \\ & \delta_{P_{r-1,n}}(D(h_2,b_r,...,b_n))  = |h_2|^{n-r+1}\prod_{i=r}^n|b_i|^{n+1-2i}.
    \end{align}
    Therefore, \Cref{SH3} is essentially bounded by the product of 
    \begin{align}\label{1}
        & \int_{(\mathbb{A}^{\times})^{n-r-1}}\int_{\mathcal{P}_{r+1}^1(F)\backslash G_{r+1}(\mathbb{A})} \|e_{r+1}h_1\|_{\mathbb{A}^{r+1}}^{-N_3}\prod_{i=r+1}^n\|a_i\|_{\mathbb{A}}^{-N_3} \|h_1\|_{G_{r+1}}^{-N_2}\notag \\ & |h_1|^{-cN_2(2n-2r+1)+r-n}\prod_{i=r+1}^n|a_i|^{-cN_2(2n+2-4i)+2i-n} \;\;\;\;\mathrm{d}h_1\prod_{i=r+2}^n\mathrm{d}a_i 
    \end{align}
    and
    \begin{align}\label{2}
        & \int_{(\mathbb{A}^{\times})^{n-r+1}}\int_{[G_{r-1}]} \prod_{i=r}^n\|b_i\|_{\mathbb{A}}^{-N_3}\|h_2\|_{G_{r-1}}^{-N_2}\notag \\ & |h_2|^{-cN_2(2n-2r+1)-n+r-\frac{1}{2}+\textnormal{Re}(s)}\prod_{i=r}^n|b_i|^{-cN_2(2n-4i)-n+2i-\frac{1}{2}+\textnormal{Re}(s)}\;\;\;\;  \mathrm{d}h_2\prod_{i=r}^n\mathrm{d}b_i
    \end{align}
    for any $N_2,N_3>0$. We will prove the absolute convergence of \Cref{1} and \Cref{2} respectively, for some suitable choices of $N_2$ and $N_3$.

    By the definition of $a_{r+1}$, we can get $\prod_{i=r+1}^na_{i}=\textnormal{det}(h_1)^{-1}$. Taking $N_4=N_3/2$, we can see that 
    $$\prod_{i=r+1}^n\|a_i\|_{\mathbb{A}}^{-N_3}\ll \|\textnormal{det}(h_1)^{-1}\|_{\mathbb{A}}^{-N_4}\prod_{i=r+2}^n\|a_i\|_{\mathbb{A}}^{-N_4}.$$ 
    Since $\|e_{r+1}h_1\|_{\mathbb{A}^{r+1}}\geq1$, \Cref{1} is essentially bounded by
    \begin{align}\label{3}
        & \int_{(\mathbb{A}^{\times})^{n-r-1}}\int_{\mathcal{P}_{r+1}^1(F)\backslash G_{r+1}(\mathbb{A})} \|e_{r+1}h_1\|_{\mathbb{A}^{r+1}}^{-N_4}\|\textnormal{det}(h_1)^{-1}\|_{\mathbb{A}}^{-N_4}\prod_{i=r+2}^n\|a_i\|_{\mathbb{A}}^{-N_4} \|h_1\|_{[\textnormal{GL}_{r+1}]}^{-N_2}\notag \\ & |h_1|^{-cN_2(2r+3)-r-2}\prod_{i=r+2}^n|a_i|^{4cN_2(i-r-1)+2(i-r-1)} \;\;\;\;\mathrm{d}h_1\prod_{i=r+2}^n\mathrm{d}a_i.
    \end{align}
    where $N_2,N_4>0$ can be picked arbitrarily.

    We can see that $G_{r+1}=S_{r+1}\rtimes G_1$, by embedding $G_1$ in the upper left corner. Thus, \Cref{3} is essentially bounded by the product of the following two integrals:
    \begin{equation}\label{5}
        \int_{G_1(\mathbb{A})}\int_{\mathcal{P}_{r+1}^1(F)\backslash S_{r+1}(\mathbb{A})}\|e_{r+1}h_1'a\|_{\mathbb{A}^{r+1}}^{-N_4}\|h_1'a\|_{G_{r+1}}^{-N_2} \|a^{-1}\|_{\mathbb{A}}^{-N_4}|a|^{-cN_2(2r+3)-r-2}\;\mathrm{d}h_1'\mathrm{d}a;
    \end{equation}
    \begin{equation}\label{6}
        \prod_{i=r+2}^n\int_{\mathbb{A}^{\times}}\|a_i\|_{\mathbb{A}}^{-N_4}|a_i|^{4cN_2(i-r-1)+2(i-r-1)}\mathrm{d}a_i.
    \end{equation}

    Note that \Cref{5} can be rewritten as:
    \begin{align}\label{4}
    \int_{G_1(\mathbb{A})}\int_{[S_{r+1}]} & \sum_{\gamma\in \mathcal{P}^1_{r+1}(F)\backslash S_{r+1}(F)}\|e_{r+1}\gamma h_1''a\|_{\mathbb{A}^{r+1}}^{-N_4}\|h_1''a\|_{G_{r+1}}^{-N_2} \|a^{-1}\|_{\mathbb{A}}^{-N_4}|a|^{-cN_2(2r+3)-r-2}\;\mathrm{d}h_1''\mathrm{d}a\notag \\ & \ll \int_{G_1(\mathbb{A})}\int_{[S_{r+1}]}\sum_{v\in F^{r+1}}\|v h_1''a\|_{\mathbb{A}^{r+1}}^{-N_4}\|h_1''a\|_{G_{r+1}}^{-N_2} \|a^{-1}\|_{\mathbb{A}}^{-N_4}|a|^{-cN_2(2r+3)-r-2}\;\mathrm{d}h_1''\mathrm{d}a.
    \end{align}

    By \cite[Lemma 2.3.3]{lu2025periodsdetectingeisensteinseries}, there exist $M>0$ and $N_0>0$ such that for any $N_4\geq N_0$, we have
    $$\sum_{v\in F^{r+1}}\|vh_1''a\|_{\mathbb{A}^{r+1}}^{-N_4}\ll \|h_1''a\|_{G_{r+1}}^M.$$
    Choosing $N_2>M$, we get
    $$\sum_{v\in F^{r+1}}\|v h_1''a\|_{\mathbb{A}^{r+1}}^{-N_4}\|h_1''a\|_{G_{r+1}}^{-N_2}\ll \|h_1''a\|_{G_{r+1}}^{-N_2+M}\ll 1.$$
    Therefore, \Cref{4} is essentially bounded by
    $$\int_{G_1(\mathbb{A})}\int_{[S_{r+1}]} \|a^{-1}\|_{\mathbb{A}}^{-N_4}|a|^{-cN_2(2r+3)-r-2}\;\mathrm{d}h_1'\mathrm{d}a\ll \textnormal{vol}([S_{r+1}])\cdot \int_{G_1(\mathbb{A})} \|a^{-1}\|_{\mathbb{A}}^{-N_4}|a|^{-cN_2(2r+3)-r-2}\;\mathrm{d}a.$$
    Replace $a^{-1}$ by $a$, and together with \Cref{c>1}, for $N_4$ large enough, \Cref{5} is absolutely convergent. Similarly, \Cref{6} is also absolutely convergent for $N_4\gg 0$. In other words, for $N_2$ and $N_3$ large enough, \Cref{1} is absolutely convergent.

    Next let us take a look at \Cref{2}. Set $\textnormal{Re}(s)=cN_2(2n-2r+1)+n-r+1/2$. Therefore, we can see that \Cref{2} is essentially bounded by the product of the following integrals:
    \begin{equation}\label{7}
        \int_{[G_{r-1}]}\|h_2\|_{G_{r-1}}^{-N_2}\;\mathrm{d}h_2;
    \end{equation}
    \begin{equation}\label{8}
        \prod_{i=r}^n\int_{\mathbb{A}^{\times}}\|b_i\|_{\mathbb{A}}^{-N_3}\|b_i\|_{\mathbb{A}^{\times}}^{-N_2}|b_i|^{cN_2(4i-2r+1)+2i-r}\;\mathrm{d}b_i.
    \end{equation}
    By \cite[Proposition A.1.1]{Beuzart-Plessis_2021}, there exists $d>0$ such that for any $N_2\geq d$, \Cref{7} is absolutely convergent. By \Cref{c>1} again, \Cref{8} is also absolutely convergent for $N_3\gg 0$.

    In conclusion, we can pick $N_2>\textnormal{max}(M,d)$ and $N_3\gg 0$, such that \Cref{5}, \Cref{6}, \Cref{7} and \Cref{8} are all absolutely convergent. And then for all complex number $s$ such that $\textnormal{Re}(s)=cN_2(2n-2r+1)+n-r+1/2$, we can get that $Z_r(f,\Phi,s)$ is absolutely convergent from the above discussion. Therefore, $Z_r(f,\Phi)$ is absolutely convergent. The absolute convergence of $Z_r'(f,\Phi)$ can be proved similarly. The proof is complete.
\end{proof}
\subsubsection{Proof of \Cref{absolute convergence 5}}

\begin{proof}
    Let $B$ denote the standard Borel subgroup of $G$. As the same argument in \Cref{proof of 4}, we need to prove that \Cref{SH1} is absolutely convergent for any $s_{\lambda},s_2\in \mathbb{C}$, $f\in \mathcal{S}([G])$ and $\Phi\in\mathcal{S}(\mathbb{A}^n)$. By \Cref{important lemma}(1), we have the following estimate:
    \begin{align*}
        |V_f(\iota'(D(a_0,a_1,\dots,a_n),D(b_1,\dots,b_n))k)|\ll & \prod_{i=1}^n\|\frac{a_i}{b_i}\|_{\mathbb{A}}^{-N_1}\prod_{i=1}^{n-1}\|\frac{b_i}{a_{i+1}}\|_{\mathbb{A}}^{-N_1}\prod_{i=1}^n\|a_i\|^{-N_2}_{G_1}\|b_i\|^{-N_2}_{G_1} \\
        & \delta_B(D(a_0,a_1,\dots,a_n),D(b_1,\dots,b_n))^{-cN_2}
    \end{align*}
    where
    $$\delta_B(D(a_0,a_1,\dots,a_n),D(b_1,\dots,b_n))=\prod_{i=1}^n|a_i|^{2-4i}|b_i|^{2n-4i}.$$

    By the same argument in \Cref{proof of 4}, for any $N_3\geq 0$, we can find $N_1\geq 0$ such that 
    $$\prod_{i=1}^n\|\frac{a_i}{b_i}\|_{\mathbb{A}}^{-N_1}\prod_{i=1}^{n-1}\|\frac{b_i}{a_{i+1}}\|_{\mathbb{A}}^{-N_1}\|b_n\|_{\mathbb{A}}^{-N_1}\ll \prod_{i=1}^n\|a_i\|_{\mathbb{A}}^{-N_3}\prod_{i=1}^n\|b_i\|_{\mathbb{A}}^{-N_3}.$$
    Therefore, \Cref{SH1} is essentially bounded by
    \begin{equation}\label{SH4}
        \prod_{i=1}^n\int_{G_1}\|a_i\|_{\mathbb{A}^{\times}}^{-N_2}|a_i|^{cN_2(4i-2)+\textnormal{Re}(s_{\lambda})+2i}\|a_i\|_{\mathbb{A}}^{-N_3}\prod_{i=1}^n\int_{\mathbb{A}^{\times}}\|b_i\|_{G_1}^{-N_2}|b_i|^{-cN_2(2n-4i)+\textnormal{Re}(s_2)-\frac{1}{2}+2i-n}\|b_i\|_{\mathbb{A}}^{-N_3}.
    \end{equation}
    Now fix $s_{\lambda}$. We can choose $N_2$ large enough such that $$cN_2(4i-2)+\textnormal{Re}(s_{\lambda})+2i>1$$
    for any $1\leq i\leq n$. Fix this $N_2$. Then we can choose a real number $b$ such that for any $s_2\in \mathcal{H}_{>b}$ and $1\leq i\leq n$, we have 
    $$-cN_2(2n-4i)+\textnormal{Re}(s_2)-\frac{1}{2}+2i-n \geq 1.$$
    Then by \cite[Lemma 2.2.5]{lu2025periodsdetectingeisensteinseries}, we can find $N_3$ large enough such that \Cref{SH4} is absolutely convergent. In conclusion, what we just proved is that for any $\lambda\in \mathfrak{a}_{P_{1,n}^1,\mathbb{C}}^*$, there exists a real number $b_{\lambda}$ such that $Z(f,\Phi,\lambda,s_2)$ is absolutely convergent when $s_2\in \mathcal{H}_{>b_{\lambda}}$. Now by using the `$f\mapsto f|\cdot|^{-s}$' technique we used in the last proof, we prove that $Z(f,\Phi,\lambda,s_2)$ is absolutely convergent for any $f\in\mathcal{S}([G])$, $\Phi\in \mathcal{S}(\mathbb{A}^n)$, $\lambda\in \mathfrak{a}_{P_{1,n}^1,\mathbb{C}}^*$ and $s_2\in \mathbb{C}$. Moreover, the above estimate is uniform when $(\lambda,s_2)$ ranges over a compact subset of $\mathfrak{a}_{P_{1,n}^1,\mathbb{C}}^*\times \mathbb{C}$. Hence the integral converges locally uniformly and defines an entire function of $(\lambda,s_2)$.

    Finally, the continuity follows from the fact that the map $f\mapsto f|\cdot|^{-s}$ is continuous for any $s\in \mathbb{C}$ and the implied constant in the Whittaker estimate is controlled by a continuous seminorm on $\mathcal{S}([G])$, as follows from \Cref{important lemma}(2). The proof is now complete.
\end{proof}

\subsection{Unfolding}\label{proof of section 5}

In this section, we develop a similar unfolding equation as in \Cref{unfolding twisted BF}:
\begin{prop}
    Let $\chi\in \mathfrak{X}_{\Delta_1^*}(G)$ be a $\Delta_1^*$-regular cuspidal datum. Then for any $(f,\Phi)\in \mathcal{S}_{\chi}([G])\times \mathcal{S}(\mathbb{A}^n)$, we have
    $$\mathcal{P}(f,\Phi)=Z_{n}'(f,\Phi)=Z_{n}(f,\Phi)=\cdots=Z_{1}'(f,\Phi)=Z_{1}(f,\Phi)=Z(f,\Phi,0,0).$$
\end{prop}

We first fix a Schwartz function $\Phi\in \mathcal{S}(\mathbb{A}^n)$. The proof is divided into four parts:
\begin{itemize}
    \item[(1)] $\mathcal{P}(f,\Phi)=Z'_{n}(f,\Phi)$;
    \item[(2)] $Z_{r+1}(f,\Phi)=Z_{}'(f,\Phi)$ for any $1\leq r\leq n-1$;
    \item[(3)] $Z_{r}'(f,\Phi)=Z_{r}(f,\Phi)$ for any $1\leq r\leq n$;
    \item[(4)] $Z_{1}(f,\Phi)=Z(f,\Phi,0,0)$.
\end{itemize}

We will start with the proof of the last equation:
\begin{proof}[Proof of \textnormal{(4)}]
    Let $U_2(\mathbb{A})=\{I_{n+1}+aE_{1,2}|a\in \mathbb{A}\}\subset S_{n+1}$. We have
\begin{align*}
    Z_1(f,\Phi) & =\int_{\mathcal{P}^1_{2}(F)N^1_{2,n+1}(\mathbb{A})\backslash S_{n+1}(\mathbb{A})}\int_{G_0(F)N_{0,n}(\mathbb{A})\backslash \textnormal{GL}_n(\mathbb{A})}f_{N_{2,2n+1,\psi'}}(\iota'(x,y))\Phi(e_ny)|y|^{1/2}\mathrm{d}x \mathrm{d}y \\
    & = \int_{N^1_{1,n+1}(\mathbb{A})\backslash S_{n+1}(\mathbb{A})}\int_{N_{0,n}(\mathbb{A})\backslash \textnormal{GL}_n(\mathbb{A})}\int_{[U_2]}f_{N_{2,2n+1,\psi'}}(\iota'(ux,y))\Phi(e_ny)|y|^{1/2}\mathrm{d}u\mathrm{d}x \mathrm{d}y \\
    & =\int_{N^1_{n+1}(\mathbb{A})\backslash S_{n+1}(\mathbb{A})}\int_{N_{n}(\mathbb{A})\backslash \textnormal{GL}_n(\mathbb{A})}f_{N_{1,2n+1,\psi'}}(\iota'(ux,y))\Phi(e_ny)|y|^{1/2}\mathrm{d}x \mathrm{d}y \\
    & = \int_{N^1_{n+1}(\mathbb{A})\backslash S_{n+1}(\mathbb{A})}\int_{N_{n}(\mathbb{A})\backslash \textnormal{GL}_n(\mathbb{A})}V_f(\iota'(x,y))\Phi(e_ny)|y|^{1/2}\mathrm{d}x \mathrm{d}y \\
    & = Z(f,\Phi,0,0).
\end{align*}
The proof is complete.
\end{proof}

Before we move on to the proof of the other three equations, we first need the following two lemmas:
\begin{lemma}\label[lemma]{lemma 3}
    Fix a positive integer $r$. Let $G=G_{2r+1}$, $H=S_{r+1}\times G_r$ and we embed $H$ into $G$ via $\iota'$. Let $\chi\in \mathfrak{X}_{\Delta_1^*}(G)$ be a $\Delta_1^*$-regular cuspidal datum. Then for any $f\in \mathcal{S}_{\chi}([G])$ and $\Phi\in \mathcal{S}(\mathbb{A}^n)$, we have
    \begin{equation}\label{vanishing by 1}
    \int_{[H]}f(\iota'(x,y))|y|^{\frac{1}{2}}\mathrm{d}x\mathrm{d}y=0.
    \end{equation}
\end{lemma}
\begin{proof}
    Let $H'=G_{r+1}\times G_r$. We embed $H'$ into $G$ by $\iota'$. By disintegration along the determinant map and Fourier inversion on the idele class group $F^{\times}\backslash \mathbb{A}^{\times}$, with compatible Haar measures, we have
    $$\int_{[H]}f(\iota'(x,y))|y|^{\frac{1}{2}}\mathrm{d}x\mathrm{d}y=\int_{\widehat{F^{\times}\backslash \mathbb{A}^{\times}}}\int_{[H']}f(\iota'(x,y))\eta^{-1}(\textnormal{det}(x))|y|^{\frac{1}{2}}\mathrm{d}x\mathrm{d}y\mathrm{d}\eta.$$
    Thus, it suffices to prove that for any unitary Hecke character $\eta$, we have
    \begin{equation}
        \int_{[H']}f(\iota'(x,y))\eta^{-1}(\textnormal{det}(x))|y|^{\frac{1}{2}}\mathrm{d}x\mathrm{d}y=0.
    \end{equation}

    Write $\chi=(M,\pi)$, where $M=\prod_{i=1}^lG_{r_i}$ and $\pi=\boxtimes_{i=1}^l\pi_i$. Let $P$ denote the standard parabolic subgroup of $G$ with Levi component $M$. For a double coset in $P(F)\backslash G(F)/H'(F)$, we choose a representative $w\in G(F)$. Then by \cite[Proposition 3.2]{Matringe+2015+119+170}, $w$ corresponds to a set
    $$s_w=\{r_{i,j},1\leq i<j\leq l,(r_{k,k}^+,r_{k,k}^-),1\leq k\leq l\}$$
    such that if we set $r_{i,j}=r_{j,i}$, and $r_{k,k}=r_{k,k}^++r_{k,k}^-$, then $r_i=\sum_{j=1}^lr_{i,j}$ and $\sum_{k=1}^lr_{k,k}^+=\sum_{k=1}^lr_{k,k}^-+1$. Now following the same argument as in the proof of \Cref{lemma 1}, we can reduce the case when there exists $j$ such that $r_{i,j}=r_i$ for every $1\leq i\leq l$, and only consider the vanishing of the following integrals:
    \begin{align}\label{SH5}
          & \prod_{i=1}^l\int_{[G_{r_{i,i}^+}]\times [G_{r_{i,i}^-}]}\phi_i\begin{pmatrix}
            m_{i,i}^+ & \\ & m_{i,i}^-
        \end{pmatrix}\eta^{-1}(\textnormal{det}(m_{i,i}^+))|m_{i,i}^-|^{\frac{1}{2}}\mathrm{d}m_{i,i}^+\mathrm{d}m_{i,i}^-   \notag \\ & \prod_{1\leq i<j\leq l}\int_{[G_{r_{i,j}}]}\phi_i(m_{i,j})\phi_j(m_{j,i})\eta^{-1}(\textnormal{det}(m_{i,j}))|m_{i,j}|^{\frac{1}{2}}\mathrm{d}m_{i,j},
    \end{align}
    where each function $\phi_i$ is a cusp form in $\pi_i$ twisted by a Schwartz function on $A_{G_{r_i}}^{\infty}$.

    By the formula $\sum_{k=1}^lr_{k,k}^+=\sum_{k=1}^lr_{k,k}^-+1$, we can imply that there must exist $i$ such that $G_{r_{i,i}}=G_{r_i}$ and $r_{i,i}^+>r_{i,i}^-$. If $r_i\geq 2$, then by \cite[Lemma 7.1]{matringe2025intertwining}, the integral
    $$\int_{[G_{r_{i,i}^+}]\times [G_{r_{i,i}^-}]}\phi_i\begin{pmatrix}
            m_{i,i}^+ & \\ & m_{i,i}^-
        \end{pmatrix}\eta^{-1}(\textnormal{det}(m_{i,i}^+))|m_{i,i}^-|^{\frac{1}{2}}\mathrm{d}m_{i,i}^+\mathrm{d}m_{i,i}^-$$
    vanishes, and then we prove the vanishing of \Cref{SH5}. If $r_i=1$, then we must have $r_{i,i}^+=1$ and $r_{i,i}^-=0$. If $\eta^{-1}\eta_i|_{\mathbb{A}^1}\neq 1$, then we have
    $$\int_{F^{\times}\backslash \mathbb{A}^1}(\eta^{-1}\eta_i)(x)\mathrm{d}x=0.$$
    Therefore, we prove the vanishing of \Cref{SH5}. The last case is $\eta|_{\mathbb{A}^1}=\eta_i|_{\mathbb{A}^1}$. Then we can write $\eta=\eta_i|\cdot|^c$ for some $c\in\mathbb{C}$. According to \Cref{Remark 1}, we know that $(\eta,\Delta^*)$-regularity is equivalent to $(\eta_i,\Delta^*)$-regularity. Then by the definition of $\Delta_1^*$-regularity, we know that $\chi^i$ is $(\eta,\Delta^*)$-regular. Now this follows from the proof of \Cref{lemma 1}. The proof is complete.
\end{proof}

\begin{lemma}\label[lemma]{lemma 4}
    Fix a positive integer $r$. Let $G=G_{2r+2}$, $H=S_{r+2}\times G_r$ and we embed $H$ into $G$ via $\iota'$. Let $\chi\in \mathfrak{X}_{\Delta_1^*}(G)$ be a $\Delta_1^*$-regular cuspidal datum. Then for any $f\in \mathcal{S}_{\chi}([G])$ and $\Phi\in \mathcal{S}(\mathbb{A}^n)$, we have
    \begin{equation}
    \int_{[H]}f(\iota'(x,y))|y|^{\frac{1}{2}}\mathrm{d}x\mathrm{d}y=0.
    \end{equation}
\end{lemma}
\begin{proof}
    The proof is identical to that of \Cref{lemma 3}, with the relation of $r_{i,i}^+$ and $r_{i,i}^-$ changing to
    $$\sum_{i=1}^lr_{i,i}^+=\sum_{i=1}^lr_{i,i}^-+2$$
    and \Cref{lemma 2} replacing \Cref{lemma 1}. We leave the details to the reader.
\end{proof}

Now we turn back to the proof of the unfolding identity:
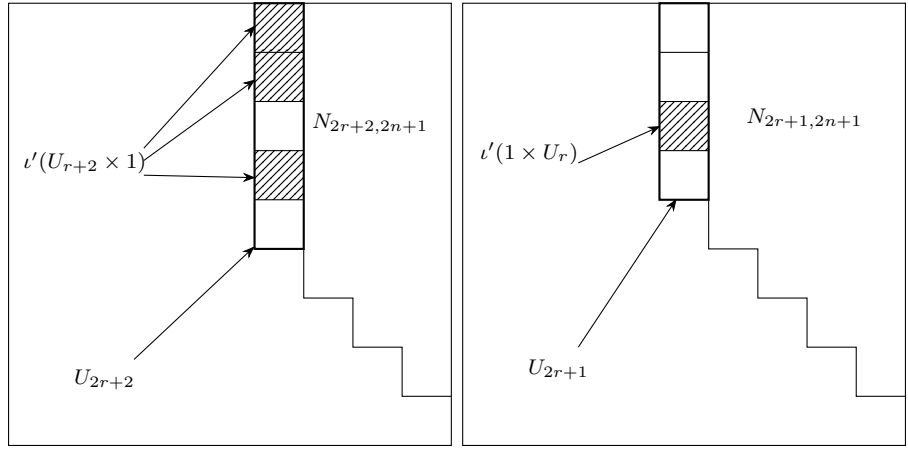
\begin{figure}[ht]
\centering
\begin{tikzpicture}[
    x=0.65cm,
    y=0.65cm,
    >=Stealth,
    arr/.style={->,thin},
    every node/.style={font=\footnotesize}
]

%==================================================
% Left picture
%==================================================
\begin{scope}

% outer 9x9 square
\draw (0,0) rectangle (9,9);

% column 6: x=5..6, five 1x1 boxes
\draw[thick] (5,4) rectangle (6,9);
\foreach \y in {5,6,7,8}
    \draw (5,\y) -- (6,\y);

% shaded boxes 1,2,4 from the top
\fill[pattern=north east lines] (5,8) rectangle (6,9);
\fill[pattern=north east lines] (5,7) rectangle (6,8);
\fill[pattern=north east lines] (5,5) rectangle (6,6);

% staircase
\draw (6,4) -- (6,3) -- (7,3) -- (7,2) -- (8,2) -- (8,1) -- (9,1);

% labels
\node at (7.35,6.55) {$N_{2r+2,2n+1}$};
\node at (1.55,5.75) {$\iota'(U_{r+2}\times 1)$};
\node at (1.95,1.35) {$U_{2r+2}$};

% arrows
\draw[arr] (2.75,6.05) -- (5.02,8.55);
\draw[arr] (2.75,5.80) -- (5.02,7.45);
\draw[arr] (2.75,5.50) -- (5.02,5.45);
\draw[arr] (2.55,1.75) -- (5.00,4.05);

\end{scope}

%==================================================
% Right picture
%==================================================
\begin{scope}[xshift=6.0cm]

% outer 9x9 square
\draw (0,0) rectangle (9,9);

% column 5: x=4..5, four 1x1 boxes
\draw[thick] (4,5) rectangle (5,9);
\foreach \y in {6,7,8}
    \draw (4,\y) -- (5,\y);

% shaded box
\fill[pattern=north east lines] (4,6) rectangle (5,7);

% staircase
\draw (5,5) -- (5,4) -- (6,4) -- (6,3) -- (7,3) -- (7,2) -- (8,2) -- (8,1) -- (9,1);

% labels
\node at (6.95,6.65) {$N_{2r+1,2n+1}$};
\node at (1.35,5.95) {$\iota'(1\times U_r)$};
\node at (1.95,1.55) {$U_{2r+1}$};

% arrows
\draw[arr] (2.35,5.75) -- (4.02,6.50);
\draw[arr] (2.35,2.00) -- (4.35,5.02);

\end{scope}

\end{tikzpicture}

\caption{Illustrations of the proof of (2) and (3) for $n=4$, $r=2$.}
\label{fig:illustration}
\end{figure}

\begin{proof}[Proof of \textnormal{(2)}]
     Let $U_{2r+2}(\mathbb{A})=\{I_{2n+1}+\sum_{i=1}^{2r+1}b_iE_{i,2r+2}|b_i\in \mathbb{A}\}\subset G(\mathbb{A})$ and $U_{r+2}(\mathbb{A})=\{I_{n+1}+\sum_{i=1}^{r+1}b_iE_{i,r+2}|b_i\in \mathbb{A}\}\subset S_{n+1}(\mathbb{A})$ (see Figure 2 for a concrete example). Using the Fourier analysis on the compact abelian group $U_{2r+2}(\mathbb{A})/\iota'(U_{r+2}\times \{1\})(\mathbb{A})U_{2r+2}(F)$, for any $g\in G(\mathbb{A})$, we can write
    \begin{align*}
        \int_{[U_{r+2}]}f_{N_{2r+2,2n+1,\psi'}}(\iota'(u,1)g)\mathrm{du} & = (f_{N_{2r+2,2n+1,\psi'}})_{U_{2r+2}}(g)+\sum_{\gamma\in \mathcal{P}_r(F)\backslash G_{r}(F)}(f_{N_{2r+2,2n+1,\psi'}})_{U_{2r+2},\psi'}(\iota'(1,\gamma)g) \\
        & = (f_{N_{2r+2,2n+1,\psi'}})_{U_{2r+2}}(g)+\sum_{\gamma\in \mathcal{P}_r(F)\backslash G_r(F)}f_{N_{2r+1,2n+1,\psi'}}(\iota'(1,\gamma)g),
    \end{align*}
    where $(f_{N_{2r+2,2n+1,\psi'}})_{U_{2r+2}}$ stands for the constant term on the subgroup $U_{2r+2}$ and $(f_{N_{2r+2,2n+1,\psi'}})_{U_{2r+2},\psi'}$ stands for the Fourier coefficient of $\psi'|_{U_{2r+2}}$. Note that $[U_{r+2}]=U_{r+2}(F)\backslash U_{r+2}(\mathbb{A})$, $\mathcal{P}_{r+2}^1(F)=S_{r+1}(F)U_{r+2}(F)$ and $N_{r+1,n+1}^1(\mathbb{A})=N_{r+2,n+1}^1(\mathbb{A})U_{r+2}(\mathbb{A})$. Together with \Cref{absolute convergence 6}, we can write
    \begin{align*}
        Z_{r+1}(f,\Phi) & =\int_{\mathcal{P}^1_{r+2}(F)N^1_{r+2,n+1}(\mathbb{A})\backslash S_{n+1}(\mathbb{A})}\int_{G_{r}(F)N_{r,n}(\mathbb{A})\backslash G_n(\mathbb{A})}f_{N_{2r+2,2n+1,\psi'}}(\iota'(x,y))\Phi(e_ny)|y|^{1/2}\mathrm{d}x \mathrm{d}y \\
        & =\int_{S_{r+1}(F)N_{r+1,n+1}^1(\mathbb{A})\backslash S_{n+1}(\mathbb{A})}\int_{\mathcal{P}_{r}(F)N_{r,n}(\mathbb{A})\backslash G_n(\mathbb{A})}f_{N_{2r+1,2n+1,\psi'}}(\iota'(x,y))\Phi(e_ny)|y|^{\frac{1}{2}}\mathrm{d}x\mathrm{d}y+F_{r}(f) \\
        & = Z_{r}'(f,\Phi)+F_{r}(f),
    \end{align*}
    where
    $$F_{r}(f)=\int_{S_{r+1}(F)N_{r+1,n+1}^1(\mathbb{A})\backslash S_{n+1}(\mathbb{A})}\int_{G_{r}(F)N_{r,n}(\mathbb{A})\backslash G_n(\mathbb{A})}(f_{N_{2r+2,2n+1,\psi'}})_{U_{2r+2}}(\iota'(x,y))\Phi(e_ny)|y|^{\frac{1}{2}}\mathrm{d}x\mathrm{d}y$$
    and the absolute convergence of the above expression of integral can be proved similarly as the proof for \Cref{absolute convergence 6}. One can prove $F_r(f)$ vanishes for any $f\in\mathcal{S}_{\chi}([G])$ where $\chi$ is $\Delta_1^*$-regular by using \Cref{lemma 3} together with the proof of \Cref{F_r=0}. We omit the details here.
\end{proof}

\begin{proof}[Proof of \textnormal{(3)}]
    The proof is similar to the proof of part (2). Let $U_{2r+1}(\mathbb{A})=\{I_{2n+1}+\sum_{i=1}^{2r}b_iE_{i,2r+1}|b_i\in \mathbb{A}\}\subset G(\mathbb{A})$ and $U_{r}(\mathbb{A})=\{I_{n}+\sum_{i=1}^{r-1}b_iE_{i,r}|b_i\in \mathbb{A}\}\subset G_n(\mathbb{A})$ (see Figure 2 for a concrete example). Using the Fourier analysis on the compact abelian group $U_{2r+1}(\mathbb{A})/\iota'(1\times U_r)(\mathbb{A})U_{2r+1}(F)$, we have
    $$\int_{[U_r]}f_{N_{2r+1,2n+1,\psi'}}(\iota'(1,u)g)\mathrm{d}u=(f_{N_{2r+1,2n+1,\psi'}})_{U_{2r+1}}(g)+\sum_{\gamma\in \mathcal{P}_{r+1}^1(F)\backslash S_{r+1}(F)}f_{N_{2r,2n+1,\psi'}}(\iota'(\gamma,1)g).$$

    By \Cref{absolute convergence 6}, we can write
    $$Z'_{r}(f,\Phi)=Z_{r}(f,\Phi)+F'_{r}(f),$$
    where
    $$F'_{r}(f)=\int_{S_{r+1}(F)N_{r+1,n+1}^1(\mathbb{A})\backslash S_{n+1}(\mathbb{A})}\int_{G_{r-1}(F)N_{r-1,n}(\mathbb{A})\backslash G_n(\mathbb{A})}(f_{N_{2r+1,2n+1,\psi'}})_{U_{2r+1}}(\iota'(x,y))\Phi(e_ny)|y|^{\frac{1}{2}}\mathrm{d}x\mathrm{d}y$$
    and the absolute convergence of the above expression of integral can be proved similarly as the proof for \Cref{absolute convergence 6}. Then, by \Cref{lemma 4} and the proof of \Cref{F_r=0} , we can show that $F'_{r}(f)=0$. The proof is complete
\end{proof}

\begin{proof}[Proof of \textnormal{(1)}]
    By \Cref{absolute convergence 6}, we can write
    \begin{align*}
        \mathcal{P}(f,\Phi) & =\int_{[S_{n+1}]\times [G_n]}f(\iota'(x,y))\Theta(y,\Phi)|\mathrm{d}x\mathrm{d}y \\
        & =\int_{[S_{n+1}]}\int_{[G_n]}f(\iota'(x,y))\sum_{v\in F^n}\Phi(vy)|y|^{\frac{1}{2}}\mathrm{d}x\mathrm{d}y \\
        & = \int_{[S_{n+1}]}\int_{[G_n]}f(\iota'(x,y))\sum_{\gamma\in \mathcal{P}_n(F)\backslash G_n(F)}\Phi(e_n\gamma y)|y|^{\frac{1}{2}}\mathrm{d}x\mathrm{d}y+F_{n}(f) \\ 
        &=Z_{n}'(f,\Phi)+F_{n}(f),
    \end{align*}
    where
    $$F_{n}(f)=\Phi(0)\int_{[S_{n+1}]}\int_{[G_n]}f(\iota'(x,y))|y|^{\frac{1}{2}}\mathrm{d}x\mathrm{d}y.$$
    The vanishing of $F_{n}(f)$ directly follows from \Cref{lemma 3}. The proof is complete.
\end{proof}

\begin{proof}[Proof of \Cref{main result 1}]
    As before, we write $\chi=(M,\pi)$ with $M=\prod_{i=1}^lG_{n_i}$ and $\pi=\boxtimes_{i=1}^l\pi_i$, where $\sum_{i=1}^ln_i=2n+1$ and each $\pi_i$ is a cuspidal representation of $G_{n_i}(\mathbb{A})$ with central character trivial on $A_{G_{n_i}}^{\infty}$.

    Let $Q=P_{(1,2n)}$. By the Iwasawa decomposition $S_{n+1}(\mathbb{A})=P_{(1,n)}^1(\mathbb{A})K^1$, we can rewrite the integral $Z(f,\Phi,\lambda,s_2)$ as
    \begin{align}\label{SH6}
    Z(f,\Phi,\lambda,s_2)  & =\int_{K^1}\int_{N_n(\mathbb{A})\backslash G_n(\mathbb{A})}\int_{N_n(\mathbb{A})\backslash G_n(\mathbb{A})}W_{(R(k)f)_Q}(\nu(x,y))\Phi(e_ny)|x|^{s_{\lambda}+n+1}|y|^{s_2+\frac{1}{2}}\mathrm{d}k\mathrm{d}x\mathrm{d}y \\
    & =\int_{K^1}Z^a((R(k)f)_Q,\Phi,s_{\lambda}+n+1,s_2)\mathrm{d}k.
    \end{align}
    where $R(k)f=R(\iota'(k,1))f$. By \cite[Lemma 2.5.4]{lu2025periodsdetectingeisensteinseries}, we can see that $(R(k)f)_Q|_{M_{Q}}\in \mathcal{T}_{\chi^{M_Q}}$, where $\chi^{M_Q}$ denotes the inverse image of $\chi$ in $\mathfrak{X}(M_Q)$. Note that $\chi^{M_Q}$ is indexed by $J_1(\chi)$. In other words, $\chi^{M_Q}=\{\eta_i\boxtimes\chi^i\}_{i\in J_1(\chi)}$. Then by \cite[Theorem 2.9.4.1]{beuzart2022global}, the function $(R(k)f)_Q|_{M_Q}$ can be written as
    $$(R(k)f)_Q=\sum_{i\in J_1(\chi)}(R(k)f)_Q^i,$$
    where $(R(k)f)_Q^i\in \mathcal{T}_{\eta_i\boxtimes\chi^i}([G_1\times G_{2n}])$ for each $i\in J_1(\chi)$. Note that from the definition of $\Delta_1^*$-regularity, $\eta_i\boxtimes \chi^i$ is $\Delta_a^*$-regular. By \Cref{main result 2}, each $Z^a((R(k)f)_Q^i,\Phi,s_{\lambda}+n+1,s_2)$ extends to an entire function. Therefore, $Z^a((R(k)f)_Q,\Phi,s_{\lambda}+n+1,s_2)$ admits an analytic continuation to an entire function on $\mathbb{C}^2$, such that for any $(k,s_{\lambda},s_2)\in K^1\times \mathbb{C}^2$, the functional $f\mapsto Z^a((R(k)f)_Q,\Phi,s_{\lambda}+n+1,s_2)$ is continuous on $\mathcal{T}_{\chi}([G])$. Then by \cite[Lemma 2.7.2]{lu2025periodsdetectingeisensteinseries}, the integral
    $$Z(f,\Phi,\lambda,s_2)=\int_{K^1}Z^a((R(k)f)_Q,\Phi,s_{\lambda}+n+1,s_2)\mathrm{d}k$$
    admits an entire continuation to $\mathfrak{a}_{P_{(1,n)}^1,\mathbb{C}}^*\times \mathbb{C}$, such that for any $(\lambda,s_2)\in \mathfrak{a}_{P_{(1,n)}^1,\mathbb{C}}^*\times \mathbb{C}$, the functional $Z(\cdot,\Phi,\lambda,s_2)$ is continuous on $\mathcal{T}_{\chi}([G])$. Set
    $$\mathcal{P}^*(f,\Phi)=Z(f,\Phi,0,0).$$
    By the continuity of $Z(\cdot,\Phi,0,0)$, we can see that $\mathcal{P}^*(\cdot,\Phi)$ is an extension of $\mathcal{P}(\cdot,\Phi)$. The proof is complete.
\end{proof}

\subsection{Evaluation on $\Delta_1^*$-regular Eisenstein series}

Let $P=M_PN_P$ be a standard parabolic subgroup of $G$. We write $M_P=\prod_{i=1}^lG_{n_i}$, where $\sum_{i=1}^ln_i=2n+1$. Let $\pi=\pi_1\boxtimes\cdots \boxtimes \pi_l$ be a unitary cuspidal representation of $M_P(\mathbb{A})$, such that the cuspidal datum $\chi=(M_P,\pi_0)$ (here $\pi_0$ represents the cuspidal representation on $M_P(\mathbb{A})$ which is a twist of $\pi$ such that $\pi_0$ is trivial on $A_{M_P}^{\infty}$) is $\Delta_1^*$-regular. Let $\varphi\in \Pi=\textnormal{Ind}_{P(\mathbb{A})}^{G(\mathbb{A})}\pi$ and write $E(\varphi)(g)=E(g,\varphi,0)$ for the Eisenstein series of $\varphi$. Note that $E(\varphi)\in \mathcal{T}_{\chi}([G])$. Similarly to \Cref{delta1}, we can define $J_1(\pi)$, $\eta_i$, $M^i$ and $\pi^i$ in the same manner. By the definition of $\Delta_1^*$-regularity, we can see that $\pi^i$ is $(\eta_i,\Delta^*)$-regular. Define $\Pi_i=\textnormal{Ind}_{M^i(\mathbb{A})}^{G_{2n}(\mathbb{A})}\pi^i$.

\begin{thm}\label[thm]{main}
    Fix a Schwartz function $\Phi\in \mathcal{S}(\mathbb{A}^n)$. Let $S$ be a sufficiently large finite set of places as specified below. Also we choose the normalization that the function
    $$W_{\varphi}^{M_P}(g):=\int_{[M_P\cap N_{2n+1}]}\varphi(ug)\psi^{-1}(u)\mathrm{d}u$$
    satisfies $W_{\varphi}^{M_P,S}(1)=1$.
    We have the following formula:
    \begin{align}\label{final equation}
    \mathcal{P}^*(E(\varphi),\Phi) & =(\Delta_{G_n\times G_n}^{S,*})^{-1} L(1,\pi,\hat{\mathfrak{n}}_P^-)^{-1}\sum_{i\in J_1(\pi)}L^S(1,\Pi_i^{\vee}\otimes \eta_i)L^S(1,\Pi_i\otimes \eta_i^{-1})L^S(\frac{1}{2},\Pi_i,\wedge^2\otimes \eta_i^{-1}) \notag \\
        & \;\;\;\;\;\;\;\;\;\;\;\;\;\;\;\;\;\;\;\;\;\;\;\;\;\;\;\;\;\;\;\;\;\;\;\;\;\;\;\;\;\;\;\;\;\;\;\;\;\;\;\;\;\;\;\;\;L_S(1,w_i\pi,\hat{\mathfrak{n}}_{Q_{w_i}}^-) \tilde{Z}_{i,S}(\varphi_S,\Phi_S,\frac{1}{2},-\frac{1}{2}).
    \end{align}
    All definitions will be made precise in the proof.
\end{thm}

The right-hand side of \Cref{final equation} is understood by meromorphic continuation. In particular, when $L(1,\pi,\hat{\mathfrak{n}}_P^-)$ has a pole, both sides vanish.

\begin{proof}
    Let $Q=P_{(1,2n)}\subset G$. By \cite[Lemma 6.10]{bernstein2024meromorphic}, we have the Langlands constant term formula
    $$(R(k)E(\varphi))_Q=\sum_{w\in {_QW_P}}E^{Q}(M(w)(R(k)\varphi)_{P_w})$$
    for any $k\in \iota'(K^1,1)$, where $(R(k)\varphi)_{P_w}$ denotes the constant term of $R(k)\varphi$ along the parabolic $P_w\cap M_P$ of $M_P$ and the superscript indicates that we replace the sum over $P(F)\backslash G(F)$ by the sum over $Q_w(F)\backslash Q(F)$. By the cuspidality of $\pi$, the only elements in ${_QW_P}$ that contribute to the constant term formula are those $w$ such that $P_w=P$, which means $(R(k)\varphi)_{P_w}=R(k)\varphi$. For such $w\in {_QW_P}$, it must move a $G_1$ block of $M_P$ to the upper left corner and keep the rest blocks in the same order. We can see that such $w$ are in one-to-one correspondence with $J_1(\pi)$, and for each $i\in J_1(\pi)$, we denote by $w_i$ the corresponding element in ${_QW_P}$. Therefore, we can rewrite the Langlands constant term formula as
    $$(R(k)E(\varphi))_Q=\sum_{i\in J_1(\pi)}E^{Q}(M(w_i)(R(k)\varphi)).$$

    Note that our description of $w_i$ shows that  $w_iM_Pw_i^{-1}=G_1\times M^i$ and $w_i\pi=\eta_i\boxtimes \pi^i$. Therefore, $M(w_i)(R(k)\varphi)\in \textnormal{Ind}_{Q_{w_i}(\mathbb{A})}^{G(\mathbb{A})}\eta_i\boxtimes \pi^i$. Let $P^i$ be the standard parabolic subgroup of $G_{2n}$ with Levi component $M^i$. Note that $G_1\times P^i$ can be regarded as a standard parabolic subgroup of $G_1\times G_{2n}$. We have the transitivity law of modulus characters:
    $$\delta_Q(r)\cdot \delta_{G_1\times P^i}(r)=\delta_{Q_{w_i}}(r)$$
    for any $r\in G_1\times M^i$.

    Now, consider the function $F_{i,k}=\delta_Q^{-1/2}M(w_i)(R(k)\varphi)$. By the definition of induced representations and the transitivity law of modulus characters, we can see that the function on $G_1(\mathbb{A})\times M^i(\mathbb{A})$ defined by $m\mapsto \delta_{P^i}^{-1/2}(m)F_{i,k}(mg)$ belongs to $\eta_i\boxtimes \pi^i$ for any $g\in G_1(\mathbb{A})\times G_{2n}(\mathbb{A})$. Thus $F_{i,k}|_{M_Q(\mathbb{A})}\in \textnormal{Ind}_{G_1(\mathbb{A})\times P_i(\mathbb{A})}^{G_1(\mathbb{A})\times G_{2n}(\mathbb{A})}\eta_i\boxtimes \pi^i$. Thus, for any $(a,g)\in G_1{(\mathbb{A})}\times G_{2n}(\mathbb{A})$, we can write $F_{i,k}(a,g)=\eta_i(a)\varphi_{i,k}(g)$ for a function $\varphi_{i,k}\in \Pi_i=\textnormal{Ind}_{P^i(\mathbb{A})}^{G_{2n}(\mathbb{A})}\pi^i$. 

    First, we need to replace the intertwining operator $M(w_i)$ by the normalized intertwining operator $N(w_i)$, because $N(w_i)$ locally sends spherical functions to spherical functions. By \cite[\S 2.4.3]{lu2025periodsdetectingeisensteinseries}, the relation between $M(w_i)$ and $N(w_i)$ is that
    $$M(w_i)=\frac{L(1,w_i\pi,\hat{\mathfrak{n}}_{Q_{w_i}}^-)}{L(1,\pi,\hat{\mathfrak{n}}_P^-)}N(w_i)=n(w_i)N(w_i).$$
    Set $\varphi_{i,k}^N=n(w_i)^{-1}\varphi_{i,k}\in \textnormal{Ind}_{P^i(\mathbb{A})}^{G_{2n}(\mathbb{A})}\pi^i$ and $F_{i,k}^N=n(w_i)^{-1}F_{i,k}=\delta_Q^{-1/2}N(w_i)(R(k)\varphi)$. Direct computation shows that $F_{i,k}^N(a,g)=\eta_i(a)\varphi_{i,k}^N(g)$.

    Let $S$ be a sufficiently large finite set of places of $F$, which we assume to contain Archimedean places as well as the places where $\Pi$, $\psi$, $\varphi$ is ramified. Furthermore, we assume that we can write $K^1=K_S^1K^{1,S}$ such that $\varphi$ is fixed by $K^{1,S}$ and write $\Phi=\Phi_S\Phi^S$, where $\Phi^S$ is the characteristic function of $(\mathcal{O}_F^S)^n$ and $\Phi_S\in \mathcal{S}(F_S^n)$. Now we can rewrite $\mathcal{P}^*(E(\varphi),\Phi)$ as follows:
    \begin{align}\label{L function}
        \mathcal{P}^*(E(\varphi),\Phi) & =(\Delta_{G_n\times G_n}^{S,*})^{-1}\int_{K^1_S}Z^a((R(k)E(\varphi))_Q,\Phi,n+1,0)\mathrm{d}k \notag \\
        & =(\Delta_{G_n\times G_n}^{S,*})^{-1}\sum_{i\in J_1(\pi)}\int_{K^1_S}Z^a(E^{Q}(M(w_i)(R(k)\varphi)),\Phi,n+1,0)\mathrm{d}k \notag \\
         & =(\Delta_{G_n\times G_n}^{S,*})^{-1}\sum_{i\in J_1(\pi)}n(w_i)\int_{K^1_S}Z^a(E^{Q}(N(w_i)(R(k)\varphi)),\Phi,n+1,0)\mathrm{d}k \notag \\
        & =(\Delta_{G_n\times G_n}^{S,*})^{-1}\sum_{i\in J_1(\pi)}n(w_i)\int_{K^1_S}Z_{\eta_i}^{\textnormal{BF}}(E^{G_{2n}}(\varphi_{i,k}^N),\Phi,\frac{1}{2},-\frac{1}{2})\mathrm{d}k.
    \end{align}
    However, this is not normalized. In other words, we do not have the condition $W_{E^{G_{2n}}(\varphi_{i,k}^N)}^S(1)=1$. By the condition that $W_{\varphi}^{M_P,S}(1)=1$ and \cite[\S 4]{shahidi2011certain}, we get that 
    $$W_{E^{G_{2n}}(\varphi_{i,k}^N)}^S(1)=L^S(1,\pi^i,\hat{\mathfrak{n}}_{P^i}^-)^{-1}.$$
    By \cite[Equation 2.6.2]{lu2025periodsdetectingeisensteinseries}, we have
    $$W_{E(\varphi_{i,k}^N),S}=L^S(1,\pi^i,\hat{\mathfrak{n}}_{P^i}^-)^{-1}\Omega_{i,S}(W_{\varphi_{i,k}^N,S}^{M^i}),$$
    where $\Omega_{i,S}:\textnormal{Ind}_{P^i(F_S)}^{G_{2n}(F_S)}\mathcal{W}(\pi^i_S,\psi_S)\rightarrow \mathcal{W}(\textnormal{Ind}_{P^i(F_S)}^{G_{2n}(F_S)}\pi^i_S,\psi_S)$ denotes the Jacquet functional (for more details, ene can check \cite[\S 2.6]{lu2025periodsdetectingeisensteinseries}). Therefore, we can further write \Cref{L function} as
    \begin{align*}
        \mathcal{P}^*(E(\varphi),\Phi) & =(\Delta_{G_n\times G_n}^{S,*})^{-1}\sum_{i\in J_1(\pi)}n(w_i)L^S(1,\Pi_i\otimes \eta_i^{-1})L^S(\frac{1}{2},\Pi_i,\wedge^2\otimes \eta_i^{-1})L^S(1,\pi^i,\hat{\mathfrak{n}}_{P^i}^-)^{-1} \\
        & \;\;\;\;\;\;\;\;\;\;\;\;\;\;\;\;\;\;\;\;\;\;\;\;\;\;\;\;\;\;\;\;\;\;\int_{K^1_S}\tilde{Z}_{\eta_i,S}^{\textnormal{BF}}(\Omega_{i,S}(W_{\varphi_{i,k}^N,S}^{M^i}),\Phi_S,\frac{1}{2},-\frac{1}{2})\mathrm{d}k \\
        & = (\Delta_{G_n\times G_n}^{S,*})^{-1} L(1,\pi,\hat{\mathfrak{n}}_P^-)^{-1}\sum_{i\in J_1(\pi)}L^S(1,\Pi_i^{\vee}\otimes \eta_i)L^S(1,\Pi_i\otimes \eta_i^{-1})L^S(\frac{1}{2},\Pi_i,\wedge^2\otimes \eta_i^{-1})\\
        & \;\;\;\;\;\;\;\;\;\;\;\;\;\;\;\;\;\;\;\;\;\;\;\;\;\;\;\;\;\;\;\;\;\;\;\;\;\;\;\;\;\;\;\;\;\;\;\;\;\;\;\;\;\;\;\;\;L_S(1,w_i\pi,\hat{\mathfrak{n}}_{Q_{w_i}}^-) \tilde{Z}_{i,S}(\varphi_S,\Phi_S,\frac{1}{2},-\frac{1}{2})
    \end{align*}
    where
    $$\tilde{Z}_{i,S}(\varphi_S,\Phi_S,\frac{1}{2},-\frac{1}{2})=\int_{K^1_S}\tilde{Z}_{\eta_i,S}^{\textnormal{BF}}(\Omega_{i,S}(W_{\varphi_{i,k}^N,S}^{M^i}),\Phi_S,\frac{1}{2},-\frac{1}{2})\mathrm{d}k.$$
    This is exactly \Cref{final equation}. The proof is complete.

\end{proof}

\subsection{Compatibility with the global numerical conjecture}\label{5.5}

By \Cref{definition1} together with \Cref{definition}(3), we can see that $\eta_i\neq \eta_j$ for different $i\neq j$ in $J_1(\pi)$. Our spherical variety corresponding to the case discussed in this section is $X=G\times ^{H}{\textnormal{std}_{G_n}}$, where ${\textnormal{std}_{G_n}}$ stands for the standard representation of $G_n$, and the equivalent condition is defined by $(g,v)\sim (gh,h^{-1}v)$ for any $g\in G$, $h\in H$ and $v\in V_{\textnormal{std}_{G_n}}$. Motivated by \cite[Table 14, Line 19]{tang2026anomaly} and the central modifications discussed in Remark 5.1 in loc. cit., we propose the following conjectural dual spherical variety:
$$\check{X}=G_{2n+1}\times^{G_1\times G_{2n}}(\textnormal{std}_{G_1}^{\vee}\boxtimes \wedge^2).$$

\subsubsection{The global Langlands correspondence}

We first review the hypothetical global Langlands correspondence, following \cite[\S A.1]{lu2025periodsdetectingeisensteinseries}:
\begin{itemize}
    \item There exists a locally compact topological group $\mathcal{L}_F$, such that there exists a bijection of isomorphism classes
    $$\{n\textnormal{-dimensional continuous irreducible rep. of } \mathcal{L}_F\}\leftrightarrow \{\textnormal{cuspidal rep. of } G_n(\mathbb{A})\}.$$
    For a cuspidal representation $\pi$ of $G_n(\mathbb{A})$, denote the corresponding representation $\mathcal{L}_F\rightarrow G_n(\mathbb{C})$ by $\phi_{\pi}$, and call it the $L$-parameter of $\pi$.

    \item Let $P=MN$ be a standard parabolic subgroup $G_n$. Let $\pi$ be a cuspidal representation of $M(\mathbb{A})$. By the above correspondence, we have an $L$-parameter $\mathcal{L}_F\rightarrow \check{M}$ of $\pi$. Let $\Pi=\textnormal{Ind}_{P(\mathbb{A})}^{G_n(\mathbb{A})}\pi$, realized as an Eisenstein series on $G_n(\mathbb{A})$. Then the $L$-parameter of $\Pi$ is given by $\mathcal{L}_F\rightarrow \check{M}\rightarrow G_n(\mathbb{C})$.
\end{itemize}

We will assume that the hypothetical global Langlands correspondence holds.

\subsubsection{Computation of the fixed points and the tangent spaces}

Let $\phi_i:\mathcal{L}_F\rightarrow G_{n_i}(\mathbb{C})$ denote the $L$-parameter of $\pi_i$ for $1\leq i\leq l$. Therefore, the $L$-parameter $\phi_{\Pi}$ of $\Pi$ can be written as
$$\phi_{\Pi}=\bigoplus_{i=1}^l\phi_i.$$
The action of $\mathcal{L}_F$ on $\check{X}$ via $\phi_{\Pi}$ composing with the natural action of $G_{2n+1}(\mathbb{C})$ on $\check{X}$. To be more precise, we may identify $\check{X}$ as the triple $(V,W,\alpha)$, where $V$ and $W$ are two subspaces of $\mathbb{C}^{2n+1}$ such that $\textnormal{dim}\; V=1$, $\textnormal{dim}\; W=2n$, and $\mathbb{C}^{2n+1}=V\oplus W$, and $\alpha$ is a vector in $V^{\vee}\otimes \wedge^2W$. The action of $G_{2n+1}(\mathbb{C})$ on $\check{X}$ is given by $g\cdot (V,W,\alpha)=(gV,gW,g\alpha)$. Then by \cite[Lemma A.2.2]{lu2025periodsdetectingeisensteinseries}, the fixed point of the action $\mathcal{L}_F$ on $\check{X}$ are of the form $(V_i,W_i,\alpha)$, where $i\in J_1(\pi)$, $V_i$ is the space corresponds to $\phi_i$, $W_i$ is the space corresponds to $\bigoplus_{j\neq i}\phi_j$, and $\alpha\in (V_i^{\vee}\otimes \wedge^2W_i)^{\mathcal{L}_F}$.

Since $V_i$ is a one-dimensional vector space, the action of $\mathcal{L}_F$ on $V_i^{\vee}\otimes \wedge^2W_i$ can be understood as 
$$\phi_i^{-1}\otimes \wedge^2 \bigoplus_{j\neq i}\phi_j=\bigoplus_{j\neq i}\phi_i^{-1}\otimes \wedge^2 \phi_j\oplus \bigoplus_{j,k\neq i\\j<k,}\phi_i^{-1}\otimes\phi_j\otimes \phi_k.$$
By Schur's lemma, $(\phi_i^{-1}\otimes\phi_j\otimes \phi_k)^{\mathcal{L}_F}\neq 0$ would imply $\phi_i^{-1}\otimes \phi_j\simeq \phi_k^{\vee}$, which is excluded by \Cref{definition}(1).
The fact that $(\phi_i^{-1}\otimes \wedge^2 \phi_j)^{\mathcal{L}_F}=0$ follows automatically if $n_j$ is odd, and follows from \Cref{definition}(2) if $n_j$ is even. Thus, the set $\check{X}^{\mathcal{L}_F}$ is in one-to-one correspondence with the set $J_1(\pi)$, and for each $i\in J_1(\pi)$, we write the corresponding fixed point as $x_i=(V_i,W_i,0)$.
Then by \cite[Lemma A.2.3]{lu2025periodsdetectingeisensteinseries}, we have
$$T_{x_i}\check{X}=V_i^{\vee}\otimes W_i\oplus V_i\otimes W_i^{\vee}\oplus V_i^{\vee}\otimes \wedge^2W_i,$$
which is compatible with \Cref{final equation}.

\printbibliography

@misc{lu2025periodsdetectingeisensteinseries,
      title={Periods detecting Eisenstein series and sums of $L$-values I}, 
      author={Weixiao Lu and Guodong Xi},
      year={2025},
      eprint={2510.12446},
      archivePrefix={arXiv},
      primaryClass={math.NT},
      url={https://arxiv.org/abs/2510.12446}, 
}

@article{Beuzart-Plessis_2021, title={COMPARISON OF LOCAL RELATIVE CHARACTERS AND THE ICHINO–IKEDA CONJECTURE FOR UNITARY GROUPS}, volume={20}, DOI={10.1017/S1474748019000707}, number={6}, journal={Journal of the Institute of Mathematics of Jussieu}, author={Beuzart-Plessis, Raphaël}, year={2021}, pages={1803–1854}}

@article{Matringe+2015+119+170,
url = {https://doi.org/10.1515/crelle-2013-0083},
title = {On the local Bump–Friedberg L-function},
title = {},
author = {Nadir Matringe},
pages = {119--170},
volume = {2015},
number = {709},
journal = {Journal für die reine und angewandte Mathematik (Crelles Journal)},
doi = {doi:10.1515/crelle-2013-0083},
year = {2015},
lastchecked = {2026-06-21}
}

@article{matringe2025intertwining,
  title={Intertwining periods, L-functions and local-global principles for distinction of automorphic representations},
  author={Matringe, Nadir and Offen, Omer and Yang, Chang},
  journal={arXiv preprint arXiv:2509.00441},
  year={2025}
}

@article{xue2025twisted,
  title={Twisted linear periods and a new relative trace formula},
  author={Xue, Hang and Zhang, Wei},
  journal={Peking Mathematical Journal},
  volume={8},
  number={3},
  pages={533--600},
  year={2025},
  publisher={Springer}
}

@article{beuzart2022global,
  title={The global Gan-Gross-Prasad conjecture for unitary groups: the endoscopic case},
  author={Beuzart-Plessis, Rapha{\"e}l and Chaudouard, Pierre-Henri and Zydor, Micha{\l}},
  journal={Publications math{\'e}matiques de l'IH{\'E}S},
  volume={135},
  number={1},
  pages={183--336},
  year={2022},
  publisher={Springer}
}

@article{leslie2025unitary,
  title={Unitary Friedberg-Jacquet periods and their twists: Relative trace formulas},
  author={Leslie, Spencer and Xiao, Jingwei and Zhang, Wei},
  journal={arXiv preprint arXiv:2503.09664},
  year={2025}
}

@article{bernstein2024meromorphic,
  title={On the meromorphic continuation of Eisenstein series},
  author={Bernstein, Joseph and Lapid, Erez},
  journal={Journal of the American Mathematical Society},
  volume={37},
  number={1},
  pages={187--234},
  year={2024}
}

@book{langlands2006functional,
  title={On the functional equations satisfied by Eisenstein series},
  author={Langlands, Robert P},
  year={2006},
  publisher={Springer}
}

@incollection{lapid2008remark,
  title={A remark on Eisenstein series},
  author={Lapid, Erez M},
  booktitle={Eisenstein series and applications},
  pages={239--249},
  year={2008},
  publisher={Springer}
}

@article{tang2026anomaly,
  title={Anomaly-free Hyperspherical Hamiltonian spaces for simple reductive groups},
  author={Tang, Guodong and Wan, Chen and Zhang, Lei},
  journal={arXiv preprint arXiv:2602.12637},
  year={2026}
}

@book{shahidi2011certain,
  title={On Certain $ L $-Functions: Conference on Certain L-functions in Honor of Freydoon Shahidi, July 23-27, 2007, Purdue University, West Lafayette, Indiana},
  author={Shahidi, Freydoon},
  volume={13},
  year={2011},
  publisher={American Mathematical Soc.}
}

@article{mao2026relative,
  title={Relative Langlands duality for some strongly tempered spherical varieties},
  author={Mao, Zhengyu and Wan, Chen and Zhang, Lei},
  journal={Inventiones mathematicae},
  volume={243},
  number={3},
  pages={993--1036},
  year={2026},
  publisher={Springer}
}

@article{ben2024relative,
  title={Relative langlands duality},
  author={Ben-Zvi, David and Sakellaridis, Yiannis and Venkatesh, Akshay},
  journal={arXiv preprint arXiv:2409.04677},
  year={2024}
}

@incollection{BF90,
  author    = {Bump, Daniel and Friedberg, Solomon},
  title     = {The exterior square automorphic {$L$}-functions on {$\mathrm{GL}(n)$}},
  booktitle = {Festschrift in honor of I. I. Piatetski-Shapiro on the occasion of his sixtieth birthday, Part II},
  series    = {Israel Math. Conf. Proc.},
  volume    = {3},
  pages     = {47--65},
  publisher = {Weizmann},
  address   = {Jerusalem},
  year      = {1990}
}

%\bibliographystyle{plain}
%\bibliography{references} 

\end{document}